\documentclass[11pt,a4paper]{article}
\usepackage[utf8]{inputenc}
\usepackage[english]{babel}
\usepackage{amsmath}
\usepackage{amsfonts}
\usepackage{amssymb}
\usepackage{amsthm}

\usepackage{graphicx}

\usepackage[left=2cm,right=2cm,top=1cm,bottom=1cm]{geometry}

\usepackage{color}

\newtheorem{theorem}{Theorem}[section]

\newtheorem{definition}{Definition}
\newtheorem{lemma}{Lemma}

\newtheorem{remark}{Remark}[section]

\allowdisplaybreaks

\usepackage[colorlinks=false]{hyperref}

\title{\Large Fractional Boundary Value Problems and elastic sticky Brownian motions, II:\\ Non-local dynamic boundary conditions on smooth domains}

\author{Mirko~D'Ovidio\footnote{Sapienza University of Rome.
    {mirkodovidio@uniroma1.it}}}

\date{}

\newcommand{\magenta}{black}

\newcommand{\indep}{\perp \! \! \! \perp}

\begin{document}


\maketitle

\vspace{-1cm}

\abstract{{\color{\magenta} We introduce a wide class of sticky Brownian motions 
(not necessarily Markovian) and study their boundary occupation time through 
different notions of holding times. This class consists of all processes 
obtained from various symbols of subordinators, say $\Phi$. Our discussion 
focuses on the infinite activity case and, as a reference model, we consider 
the stable subordinator corresponding to $\Phi(\lambda)=\lambda^\alpha$.} 
Sticky diffusion processes on bounded domains can spend a finite time (with a 
finite mean) on the lower-dimensional space given by the boundary. Once the 
process hits the boundary, {\color{\magenta} it restarts either instantaneously 
or after a random period of time.} While on the boundary, it can stay or move 
according to dynamics that differ from those in the interior. Such processes 
may be characterized by a time derivative appearing in the boundary condition 
of the governing problem. {\color{\magenta} We restrict our attention to static 
behaviour without lateral diffusion (i.e., the boundary trace process is a pure 
jump process).} We use suitable time-changes in order to describe fractional 
sticky conditions and the associated boundary behaviours. We show that 
fractional boundary value problems (involving fractional dynamic boundary 
conditions) lead to sticky diffusions, which are strong Markov in the interior 
and spend a finite time (with an infinite mean) on the boundary. Such 
behaviour can be interpreted as a trapping effect from a macroscopic point of view.
}

\tableofcontents

\section{Introduction}

We show that a fractional dynamic boundary condition for the heat equation can be associated with a suitable time-change for a Brownian motion in a bounded domain. The associated fractional boundary value problem leads to the characterization of the time or the mean time the process spends on the boundary. 

In the literature, we found many papers dealing with fractional boundary value problems, but in a completely different context. Indeed, they typically deal with fractional Cauchy problems (FCP) or non-local PDEs with local boundary conditions. 

\subsection{A brief review of the state of the art}
\label{nutshell}

\subsubsection{Non-Local Initial Value Problems (NLIVP).} 

Consider the Caputo-D\v{z}rba\v{s}jan fractional derivative $D^\alpha_t$ introduced in \cite{caputoBook, CapMai71, CapMai71b}  by Caputo and separately in a series of works starting from \cite{Dzh66, DzhNers68} by D\v{z}rba\v{s}jan. The fractional Cauchy problem $D^\alpha_t u = \Delta u$, $u_0=f$ (of the type in formula \eqref{probXL} below) has been investigated by many researchers from both the analytical and the probabilistic points of view (see, for example,  \cite{BaeMee2001, Baz2000, Kiryakova94, Koc89, MLP2001, MNV09} and the recent works \cite{Dov12, GLY15, Kolo19, luchko20}). Here, we have \emph{local boundary conditions}. It is well-known that the probabilistic solution is obtained through time change given by the inverse to an $\alpha$-stable subordinator. We refer to fractional operators with a clear identification of the fractional order $\alpha \in (0,1)$. In the general case, we may refer to the symbol $\Phi$ and write $D^\Phi_t$ (see for example \cite{Chen17, Koc2011, Kolo19, Toaldo2015}). The general class of non-local operators obviously includes the fractional ones.

\subsubsection{Non-Local Boundary Value Problems (NLBVP).} In the present paper we are interested in local problems with \emph{non-local boundary conditions} written in terms of fractional (time) operators (of the type in formula \eqref{probXbar} below). To the best of our knowledge there are no works dealing with this problem. Our results seem to be new and accord well with the existing literature. Our study has been inspired by \cite{SW05} where the drifted sticky Brownian motion and the Ray process have been considered as solutions for a problem with dynamic boundary conditions. Further references for the sticky (or Feller-Wentzell) condition are given by the pioneering works \cite{Feller52, Wentzell59}, whereas for the probabilistic description in terms of sample paths we recall \cite[Section 10]{ItoMcK}. Further investigation has been carried out in many subsequent works, starting, for example, from \cite{Amir91, Warren97} for the sticky Brownian motion or \cite{CleTim86, TaiFavRom2000} for dynamic boundary conditions. 

For the physical interpretation of the dynamic boundary condition we refer to \cite{Gold2006}. From the probabilistic viewpoint, the (pure) sticky condition (given by $\Delta u=0$ on $\partial \Omega$) represents instantaneous absorption whereas, the (reflecting) sticky condition ($\eta \Delta u + \sigma \partial_{\bf n} u=0$ on $\partial \Omega$) represents slow reflection. Thus, a reflected sticky (or slowly reflected) Brownian motion spends some time on the boundary by means of a suitable time change. This means that the process is delayed on the boundary with each visit (see Section \ref{sec:secX} below).

Our result is concerned with the elastic sticky Brownian motion ($\eta \Delta u + \sigma \partial_{\bf n}u + c u = 0$ on $\partial \Omega$) and the fractional dynamic boundary condition written in terms of the so called Caputo-D\v{z}rba\v{s}jan derivative $D^\alpha_t$. We show that the associated process turns out to be delayed only on the boundary so that we get an infinite average of (occupation) time on the boundary. The reflections slow down according to independent extra times. This can be regarded as a simple case of trapping boundaries and finds many possible applications for the characterization of motions on trap boundaries (see \cite{BCM06}) or in trap vertices of metric graphs (see \cite{ColDov, BonDov}). 

We provide new results in the case of order $\alpha \in (0,1]$ and dimension $d>1$. The existing literature is only concerned with the case $\alpha=1$ from the analytical point of view and $\alpha=1$, $d=1$ from the probabilistic viewpoint.

\subsection{Notations}
\label{sec:intro}

Let $\eta, \sigma, c$ be positive constants. For an open bounded set $\Omega \subset \mathbb{R}^d$, $d>1$, with smooth boundary $\partial \Omega$, we can equip the compact Euclidean space $\bar{\Omega} = \Omega \cup \partial \Omega$ with a finite Borel measure
\begin{align}
m(dx) = \mathbf{1}_\Omega\, dx + (\eta/\sigma)\, \mathbf{1}_{\partial \Omega} \, m_\partial(dx)
\label{mMeasure}
\end{align}
defined as the sum of the $d$-dimensional Lebesgue measure supported on the interior and the $(d-1)$-dimensional Hausdorff measure supported on the boundary. In particular, $m_\partial$ is the $(d-1)$-dimensional volume measure on $\partial \Omega$. 

We briefly introduce the processes we deal with further on. As usual, we denote by $\mathbf{E}_x$ the expectation operator under the measure $\mathbf{P}_x$ where $x$ is a starting point:
\begin{itemize}
\item[i)] $X^+= \{X^+_t\}_{t\geq 0}$ is a reflecting Brownian motion on $\bar{\Omega}$ with boundary local time $\gamma^+ = \{\gamma^+_t\}_{t\geq 0}$. The process $X^+$ has generator $(G^+, D(G^+))$ where $G^+=\Delta$ is the Neumann Laplacian with 
\begin{align*}
D(G^+)= \{\varphi, \Delta \varphi \in C(\bar{\Omega}),\, \varphi \in H^1(\Omega)\,:\, \partial_{\bf n} \varphi = 0\};
\end{align*}
\item[ii)] $X^{el}=\{X^{el}_t\}_{t\geq 0}$ is an elastic Brownian motion on $\bar{\Omega}$ with
\begin{align*}
\mathbf{E}_x[f(X^{el}_t)] = \mathbf{E}_x[f(X^+_t)\, e^{-(c/\sigma) \gamma^+_t}], \quad t>0,\; x \in \bar{\Omega}.
\end{align*}
Let $(G^{el}, D(G^{el}))$ be the generator of $X^{el}$. Then, $G^{el}=\Delta$ and
\begin{align*}
D(G^{el}) = \{\varphi, \Delta \varphi \in C(\bar{\Omega}), \, \varphi \in H^1(\Omega)\,:\, 0 = \sigma \partial_{\bf n} \varphi  + c \varphi |_{\partial \Omega}\};
\end{align*}
\item[iii)] $X^\dagger = \{X^\dagger_t\}_{t\geq 0}$ is a Brownian motion on $\Omega$ killed upon reaching the boundary $\partial \Omega$ for which
\begin{align*}
\mathbf{E}_x[f(X^\dagger_t)] = \mathbf{E}_x[f(X^+_t), t< \tau_{\partial \Omega}], \quad t>0,\; x \in \Omega
\end{align*}
where $\tau_{\partial \Omega} = \inf \{t \geq 0\,:\, X^+_t \in \partial \Omega\}$. The process $X^\dagger$ has generator $(G^\dagger, D(G^\dagger))$ where $G^\dagger=\Delta$ is the Dirichlet Laplacian and
\begin{align*}
D(G^\dagger) = \{ \varphi, \Delta \varphi \in C(\bar{\Omega})\,:\, \varphi |_{\partial \Omega} = 0\};
\end{align*}
\item[iv)] $X = \{X_t\}_{t\geq 0}$ is an elastic sticky Brownian motion on $\bar{\Omega}$ with boundary local time $\gamma=\{\gamma_t\}_{t \geq 0}$. The process is termed elastic sticky according to the Wentzell-Robin boundary condition.  
The sticky process $X$ with elastic kill has generator $(G, D(G))$ with $G=\Delta$ and
\begin{align*}
D(G) = \{\varphi, \Delta \varphi \in C(\bar{\Omega}), \, \varphi \in H^1(\Omega)\,:\, \eta (\Delta \varphi) |_{\partial \Omega} = - \sigma \partial_{\bf n} \varphi  - c \varphi |_{\partial \Omega}\}
\end{align*}
where $\partial_{\bf n}\varphi$ is the outer normal derivative with respect to $m$ introduced in \eqref{mMeasure}. We recall that $u|_{\partial \Omega}$ is the trace function (continuous operator) from $H^1(\Omega)$ into $L^2(\partial \Omega, m_\partial)$, such that
\begin{align*}
u \mapsto Tu=u|_{\partial \Omega}
\end{align*}
for $u \in H^1(\Omega)\cap C(\bar{\Omega})$. The semigroup generated by $(G, D(G))$ is a compact, positive $C_0$-semigroup on $C(\bar{\Omega})$. As $\sigma \to \infty$ (or equivalently $\eta=c=0$), we have the process $X^+$ with generator $(G^+, D(G^+))$ whereas, as $c\to \infty$ (or equivalently $\eta=\sigma=0$) we have the process $X^\dagger$ with generator $(G^\dagger, D(G^\dagger))$. The case $\eta \to \infty$ leads to the (pure) sticky condition whereas, in case $\eta \to 0$ we get $X^{el}$ under Robin condition;
\item[v)] $H=\{H_t\}_{t\geq 0}$ is a stable subordinator with $\mathbf{E}_0[\exp(-\lambda H_t)] = \exp(-t \lambda^\alpha)$, $\alpha \in (0,1)$;
\item[vi)] $L=\{L_t\}_{t\geq 0}$ is the inverse $L_t= \inf\{s \geq 0\,:\, H_s >t\}$ to $H$.
\end{itemize}
With no abuse of notation we refer to $\tau_{\partial \Omega}$ as the first hitting time of the boundary $\partial \Omega$ of the processes previously introduced if no confusion arises. We also use the notation
\begin{align*}
X_t = X \circ t \quad \textrm{and} \quad X_{T_t} = X \circ T_t
\end{align*} 
for a given random time $T_t$.

\subsection{Results}
Our results are concerned with the processes:
\begin{itemize}
\item $X^L = \{X^L_t\}_{t\geq 0}$ with $X^L_t := X\circ L_t$ (see Section \ref{sec:secXL} for details). The process $X^L$ has been extensively investigated. It can be associated with a {\bf NLIVP} also known as FCP;
\item $\hat{X} = \{\hat{X}_t\}_{t\geq 0}$ with boundary local time $\hat{\gamma}= \{\hat{\gamma}_t\}_{t\geq 0}$.  We introduce $\hat{X}_t$ as the composition $X \circ \hat{L}_t$ where $\hat{L} = \{\hat{L}_t\}_{t \geq 0}$ is a special time change defined below (see Section \ref{sec:hatX} for $\hat{L}$ and  Section \ref{sec:hatX} for $\hat{X}$). We show that the process $\hat{X}$ can be associated with a {\bf NLIVP}. The process $\hat{L}$ gives a sort of bridge between NLIVPs and NLBVPs. After the definition and the characterization of the process $\hat{X}$ we state the main result in Theorem \ref{lemmaWrond} giving the PDEs connection;
\item $\bar{X} = \{\bar{X}_t\}_{t \geq 0}$ with boundary local time $\bar{\gamma}= \{\bar{\gamma}_t\}_{t\geq 0}$. We show that the process $\bar{X}$ can be associated with a {\bf NLBVP}, the main result is stated in Theorem \ref{thm:MAINI}. In particular, $\bar{X}_t$ is obtained as the composition $X^{el} \circ \bar{V}^{-1}_t$ where $\bar{V}^{-1}$ is the inverse of $\bar{V} = \{\bar{V}_t\}_{t\geq 0}$ defined in formula \eqref{Vbar} below (see also Section \ref{sec:barX}). The elastic Brownian motion $X^{el}$ will be written is terms of the pair $(X^+, \gamma^+)$ and the random time $\bar{V}^{-1}$ introduces a fractional dynamic condition for the reflected diffusion $X^+$.  It is worth noticing that $\hat{X}$ and $\bar{X}$ have the same (delayed) behaviour on the boundary.
\end{itemize}

We proceed with a presentation of the results. Our aim is to study the heat equation on $\Omega$ with non-local dynamic condition
\begin{align}
\eta D^\alpha_t u = - \sigma \partial_{\bf n} u - c u \quad \textrm{on} \quad \partial \Omega 
\label{DBCintro}
\end{align}
where $D^\alpha_t$ is a non-local operator in time (Caputo-D\v{z}rba\v{s}jan  derivative) defined as 
\begin{align}
D^\alpha_t u(t,x) =
\left\lbrace
\begin{array}{ll}
\displaystyle  \frac{1}{\Gamma(1-\alpha)} \int_0^t \frac{\partial u}{\partial s}(s,x) (t-s)^{-\alpha} ds, & \alpha \in (0,1),\\
\\
\displaystyle \frac{\partial u}{\partial t}(t,x), & \alpha=1.
\end{array}
\right.
\label{CDderDef}
\end{align} 
For $z>0$, $\Gamma(z) = \int_0^\infty s^{z-1} e^{-s}ds$ is the absolutely convergent Euler integral. As usual, we refer to $\Gamma(\cdot)$ as the gamma function.

For $\alpha=1$ the previous condition takes the form
\begin{align*}
\eta \frac{\partial u}{\partial t} = - \sigma \partial_{\bf n} u - c u \quad \textrm{on} \quad \partial \Omega 
\end{align*}
which corresponds to the Wentzell-Robin boundary condition
\begin{align*}
\eta \Delta u = - \sigma \partial_{\bf n}u - cu \quad \textrm{on} \quad \partial \Omega.
\end{align*}
Thus, we obtain a well-known problem, the Laplacian generates a positive contraction semigroup (\cite{AMPR03}). If $\partial \Omega$ is $C^\infty$ then  we may consult \cite{LionMagenes72} whereas, if $\partial \Omega$ is Lipschitz, then we may refer to \cite{JerisonKenig1981}. From the probability view point, the Wentzell boundary condition leads to a sticky Brownian motion and the Robin boundary condition leads to an elastic Brownian motion. The associated process is $X$ on $\bar{\Omega}$ with generator $(G,D(G))$ introduced above. 

For $\alpha\in (0,1)$ we provide a description of the associated stochastic process and focus on the stochastic dynamic near the boundary, then we mainly restrict our analysis to a bounded domain $\Omega \subset \mathbb{R}^d$ meaning an open, connected and non-empty set with $C^\infty$ boundary $\partial \Omega$.

In order to have a clear picture of our results we first discuss a collection of related problems. Recall that $\eta, \sigma, c$ are positive constants. Let us consider the dynamic boundary value problems
\begin{align}
\left\lbrace
\begin{array}{ll}
\displaystyle \frac{\partial v}{\partial t}(t,x) = A v(t,x), \quad t>0, \; x \in \Omega,\\
\\
\displaystyle v(0,x) = f(x), \quad x \in \Omega,\\
\\
\displaystyle \mathfrak{v}(t,x) = Tv(t,x), \quad t>0, \; x \in \partial \Omega,\\ 
\\
\displaystyle \eta \frac{\partial \mathfrak{v}}{\partial t} (t,x) = B v(t,x), \quad t>0,\; x \in \partial \Omega,\\
\\
\displaystyle \mathfrak{v}(0,x) = Tv(0,x) = f(x), \quad x \in \partial \Omega,
\end{array}
\right .
\label{probX}
\end{align}
\begin{align}
\left\lbrace
\begin{array}{l}
\displaystyle D^\alpha_t w(t,x) = A w(t,x), \quad t>0, \; x \in \Omega, \\
\\
\displaystyle w(0,x) = f(x), \quad x \in \Omega\\
\\
\displaystyle \mathfrak{w}(t,x)= Tw(t,x), \quad t>0,\, x \in \partial \Omega\\
\\
\displaystyle \eta D^\alpha_t \mathfrak{w}(t,x) = B w(t,x), \quad t>0,\, x \in \partial \Omega,\\
\\
\displaystyle \mathfrak{w}(0,x) = f(x), \quad x \in \partial \Omega
\end{array}
\right .
\label{probXL1}
\end{align}
and
\begin{align}
\left\lbrace
\begin{array}{l}
\displaystyle \frac{\partial u}{\partial t}(t,x) = A u(t,x), \quad t>0, \; x \in \Omega, \\
\\
\displaystyle u(0,x) = f(x), \quad x \in \Omega\\
\\
\displaystyle \mathfrak{u}(t,x)= Tu(t,x), \quad t>0,\, x \in \partial \Omega\\
\\
\displaystyle \eta D^\alpha_t \mathfrak{u}(t,x) = B u(t,x), \quad t>0,\, x \in \partial \Omega,\\
\\
\displaystyle \mathfrak{u}(0,x) = f(x), \quad x \in \partial \Omega
\end{array}
\right .
\label{probXbar1}
\end{align}
for $f \in C(\bar{\Omega})$ and $D(B) \subset D(A)$. Further on we consider $A=\Delta $ and $B = -\sigma \partial_{\bf n} - c$. 

Focus on the problem \eqref{probX}. If $\Delta u \in C(\bar{\Omega})$, then (consult for example \cite{EngFra05})
\begin{align}
\label{SameArgs}
\eta \frac{\partial \mathfrak{v}}{\partial t} = \eta \frac{\partial Tv}{\partial t} = \eta T \frac{\partial v}{\partial t} = \eta T \Delta v,
\end{align}
that is $\eta \Delta v = - \sigma \partial_{\bf n} v - c v$ on $\partial \Omega$ due to the definition of $B$. The solution to the problem \eqref{probX} can be written in terms of $X$ with generator $(G,D(G))$, 
\begin{align*}
v(t,x) = \mathbf{E}_x[f(X_t)], \quad t\geq 0,\, x \in \bar{\Omega}.
\end{align*}

({\bf NLIVP}) Under $\Delta w \in C(\bar{\Omega})$, the problem \eqref{probXL1} can be written as
\begin{align}
\left\lbrace
\begin{array}{l}
\displaystyle D^\alpha_t w(t,x) = G w(t,x), \quad t>0, \; x \in \bar{\Omega}, \\
\\
\displaystyle w(0,x) = f(x), \quad x \in \bar{\Omega}, \quad f \in D(G),
\end{array}
\right .
\label{probXL}
\end{align}
and this actually brings our problem in the context of NLIVPs. Since $(G, D(G))$ generates the elastic sticky Brownian motion $X$, we have the probabilistic representation  
\begin{align*}
w(t,x) = \mathbf{E}_x[f(X^L_t)], \quad t\geq 0, \, x\in \bar{\Omega}
\end{align*}
for the solution $w$ to \eqref{probXL1}. This result follows from well-known results after suitable adaptation (we discuss on this in Section \ref{sec:secXL}). We will show that $\hat{X}$ can be associated with \eqref{probXL} only if $\operatorname{supp}[f] \subseteq \partial \Omega$ (Theorem \ref{lemmaWrond}). We observe that the problem associated with $X^L$ (and therefore with $\hat{X}$) can be formulated in terms of the operator matrix on $C(\bar{\Omega})$ given by
\begin{align}
\mathcal{A} = \left( \begin{array}{cc}
A & 0\\
B & 0
\end{array} \right), \quad D(\mathcal{A}) = \left\lbrace \left( \begin{array}{c} w\\ Tw \end{array} \right) \in D(A) \times C(\partial \Omega)\, \right\rbrace
\end{align}
for which we have
\begin{align*}
D^\alpha_t \left( \begin{array}{c} w\\ Tw \end{array} \right) = \mathcal{A} \left( \begin{array}{c} w\\ Tw \end{array} \right), \quad \alpha \in (0,1].
\end{align*}
This formulation does not hold for the problem \eqref{probXbar1}.\\

({\bf NLBVP}) By following the same arguments as in \eqref{SameArgs}, we write \eqref{probXbar1} as 
\begin{align}
\left\lbrace
\begin{array}{l}
\displaystyle \frac{\partial u}{\partial t}(t,x) = \Delta u(t,x), \quad t>0, \; x \in \Omega, \\
\\
\displaystyle \eta D^\alpha_t Tu(t,x) = - \sigma \partial_{\bf n} u(t,x) - c \, u(t,x), \quad t>0, \; x \in \partial \Omega,\\
\\
\displaystyle u(0,x) = f(x), \quad x \in \bar{\Omega}.
\end{array}
\right .
\label{probXbar}
\end{align}
We relate the problem \eqref{probXbar} with $\bar{X}$. In particular, we show (Theorem \ref{thm:MAINI}) that the solution $u(t,x) := \mathbf{E}_x[f(\bar{X}_t)]$ to \eqref{probXbar} can be written as 
\begin{align*}
u(t,x) = \mathbf{E}_x[f(X^{el} \circ \bar{V}^{-1}_t)] = \mathbf{E}_x [f(X^+ \circ \bar{V}^{-1}_t) \, M^+ \circ \bar{V}^{-1}_t]
\end{align*}
where $\bar{V}^{-1}_t := \inf\{ s \, :\, \bar{V}_s > t\}$ is the inverse to the process
\begin{align}
\label{Vbar}
\bar{V}_t := t +  H \circ (\eta/\sigma) \gamma^+_t
\end{align}
and $M^+_t = e^{-(c/\sigma) \gamma^+_t}$ is the multiplicative functional associated with the elastic condition (see formulae \eqref{lawExpLT} and \eqref{equivgammaBAR} below). This representation introduces a time change written in terms of a subordinator $H$ independent from $X^+$ and provides a clear meaning for the non-local condition appearing in \eqref{probXbar}. Since the process $H$ may have jumps, then the inverse to $\bar{V}_t$ slows down the process according to the plateaus associated with that jumps.  \\

({\bf Holding Times}) {\color{\magenta} We refer the reader to Definition \ref{defSHTandHT}. The sticky holding times play a crucial role in the present work and encompass the case of holding times, which correspond to the permanent occupation of a state (see Remark \ref{rmk:HThatX} and Remark \ref{rmk:HTbarX}). Given the specific nature of the process $X$ we introduce sticky holding times (Definition \ref{def:HTX}) to describe the sticky occupation times of $\partial \Omega \subset \Lambda_\epsilon$ in a visit to the $\epsilon$-neighbourhood $\Lambda_\epsilon$ of the boundary (more precisely, $\epsilon$-tubular neighbourhood).} In the present work a process is termed sticky if its associated boundary occupation set has positive Lebesgue measure. 

Sticky holding times are strictly related to the new clock of $X^+$ introduced by $\bar{V}_t$ which is continuous only in the case $\alpha=1$ (in which case we write $V_t$). For $\alpha \in (0,1)$ the time change provides a clear understanding of the occupation time on $\partial \Omega$ in terms of holding times due to the plateaus of $\bar{V}_t^{-1}$. Thus, the process can also occupy a single point for a positive random time whereas for $\alpha=1$ we only have sticky holding time, that is the time accumulated on $\partial \Omega$ by $X$ is only due to the infinitely many crossings of $X^+$. 

As a global behaviour, the processes $\bar{X}$ and $\hat{X}$ spend an infinite mean amount of time on the boundary. Since $H$ is independent of $X^+$, then the delay on the boundary for $\bar{X}$ and $\hat{X}$ can be regarded as a trapping effect. Such behaviour can be considered in many contexts (see Section \ref{Sec:ConOP}). For $\bar{X}$ (or $\hat{X}$) on metric graphs (see \cite{BonDov}) for example, the holding times on a set of vertices provide a clear reading for data/traffic flows. Moreover, due to the independence of $H$, the holding times and the sticky holding times can be charged to exogenous causes, for example to some property of the ("true") boundary in case of real world applications. We say that $\bar{X}$ on a regular domain ($\Omega$ smooth as it is here) can be considered as a model for data coming from an elastic Brownian motion $X^{el}$ on domains with irregular boundary (trap domains for instance). We clarify this point in the last section of the work. 

An interesting reading can be given in terms of occupation measures. For a set $\Lambda \subset \mathbb{R}^d$, assume that
\begin{align}
\mathbf{E}_x \left[ \int_0^\infty \mathbf{1}_\Lambda(X_t) dt \right] < \infty \quad \textrm{for} \quad  \Lambda \subseteq \bar{\Omega}, \quad x \in \bar{\Omega}.
\label{ASShtINTRO}
\end{align}
Due to the heavy tailed time change on the entire domain, we have that (see \eqref{occupMeasXL} below) 
\begin{align}
\textrm{ for all } x \in \bar{\Omega}, \quad \mathbf{E}_x \left[ \int_0^\infty \mathbf{1}_\Lambda(X^L_t) dt \right] = \infty \quad \textrm{for} \quad \Lambda \subseteq \bar{\Omega}   
\label{occupation1intro}
\end{align}
whereas, because of the heavy tailed sticky holding times on the boundary (see \eqref{occupMeasXhat} below)
\begin{align}
\textrm{ for all } x \in \bar{\Omega}, \quad  \mathbf{E}_x \left[ \int_0^\infty \mathbf{1}_\Lambda(\hat{X}_t) dt \right] = \infty \quad \textrm{only if} \quad \Lambda \cap \partial \Omega \neq \emptyset
\label{occupation2intro}
\end{align}
and (see \eqref{occupMeasXbar} below)
\begin{align}
\textrm{ for all } x \in \bar{\Omega}, \quad  \mathbf{E}_x \left[ \int_0^\infty \mathbf{1}_\Lambda(\bar{X}_t) dt \right]
= \infty \quad \textrm{only if} \quad \Lambda \cap \partial \Omega \neq \emptyset.
\label{occupation3intro}
\end{align}
Formula \eqref{occupation1intro} is well-known as a consequence of the delaying process $L$. Formulae \eqref{occupation2intro} and \eqref{occupation3intro} are obtained below as a by-product of our results.

\section{Preliminaries}
\label{Sec:NotPre}

For the reader's convenience we introduce here some definitions together with some known facts and provide some details that turn out to be useful further on.

\subsection{Positive Continuous Additive Functionals}

We recall some properties of this special class of functionals. Consider the set $\mathbf{A}^+_c$ of PCAFs and denote by $\mu_A$ the Revuz measure of $A \in \mathbf{A}^+_c$. We denote by $\mathbf{A}^+_c(\Lambda)$ the set of PCAFs associated with Revuz measures compactly supported in $\Lambda \subseteq \bar{\Omega}$. That is, 
\begin{align*}
\mathbf{A}^+_c(\Lambda)=\{A \in \mathbf{A}^+_c \,:\, \operatorname{supp}[\mu_A] = \Lambda\}, \quad \Lambda \subseteq \bar{\Omega}.
\end{align*}

Let $Z=\{Z_t\}_{t\geq 0}$ be a process on $\bar{\Omega}$ with $\mathbf{P}_x(Z_s \in dy) = p_Z(s,x,y)\mu(dy)$ and $\mu$ with support on $\bar{\Omega}$. Denote by $\mathbf{E}_\mu$ the integral $\int \mathbf{E}_x \mu(dx)$ where, as usual, $\mathbf{E}_x$ is the mean value under $\mathbf{P}_x$ and $x=Z_0$ is a point in $\bar{\Omega}$. For simplicity we also introduce the set of Borel functions 
\begin{align*}
\mathcal{B}(\Lambda) = \{ \textrm{measurable functions on the Borel set $\Lambda\subseteq \bar{\Omega}$}\}.
\end{align*} 
If $A_t \in \mathbf{A}^+_c(\Lambda)$ is a PCAF of the process $Z$ on $\bar{\Omega}$, we have that, for $f \in \mathcal{B}(\Lambda)$,
\begin{align}
\label{limitRevuz}
\lim_{t \downarrow 0} \mathbf{E}_\mu \left[ \frac{1}{t} \int_0^t f(Z_s) dA_s \right] = \int_{\Lambda} f(x) \mu_A(dx)
\end{align}
and  
\begin{align}
\label{INTmuA}
\mathbf{E}_x \left[ \int_0^t f(Z_s)dA_s \right] = \int_0^t \int_\Lambda f(y) p_Z(s,x,y) \mu_A(dy) ds, \quad x \in \bar{\Omega}.
\end{align}
Further on we also use the fact that
\begin{align}
\label{limitRESadditive}
\lim_{t \downarrow 0} \mathbf{E}_\mu \left[ \frac{1}{t} \int_0^t f(Z_s) dA_s \right] = \lim_{\lambda \to \infty} \lambda \mathbf{E}_\mu \left[\int_0^\infty e^{-\lambda t} f(Z_t) dA_t \right]
\end{align}
and
\begin{align}
\label{potintdA}
\mathbf{E}_x \left[ \int_0^\infty e^{-\lambda t} f(Z_t)dA_t \right] = \lambda \int_0^\infty e^{-\lambda t} \mathbf{E}_x\left[ \int_0^t f(Z_s) dA_s \right] dt.
\end{align}
In particular, we deal with Lebesgue-Stieltjes integrals such as $\int_0^t dA_s$ and $\int_0^t dA^{-1}_s$ where $A_t$ denotes the boundary local time or the interior occupation time of a process on $\bar{\Omega}$ and $A^{-1}_t$ will be the inverse process
\begin{align*}
A^{-1}_t = \inf\{s\,:\, A_s>t\}, \quad 0\leq t < A_{\zeta_Z-} \quad \textrm{with} \quad  A^{-1}_t = 0, \quad t\geq A_{\zeta_Z -}
\end{align*}
where $\zeta_Z = \inf\{s\,:\, Z_s \in \partial \}$ is the lifetime of $Z$ and $\partial$ is a ''cemetery point''. 

We recall the following lemma.
\begin{lemma}
\label{lemmaBluGet}
(\cite[Lemma 2.2, Chapter V]{BluGet68}) Let $a(t)$ be a function from $[0, \infty)$ to $[0, \infty)$ which is non-decreasing and right-continuous and satisfies $a(0)=0$, $a(\infty)=\lim_{t\uparrow \infty} a(t)$. Define
\begin{align*}
a^{-1}(t) = \inf\{s\,:\, a(s)>t \}
\end{align*}
for $0\leq t < \infty$ where as usual we set $a^{-1}(t)=\infty$ if the set in braces is empty. The function $a^{-1}$ from $[0, \infty)$ to $[0, \infty)$ is called the inverse of $a$. It is right-continuous and non-decreasing. Define $a^{-1}(\infty) = \lim_{t \uparrow} a^{-1}(t)$. Then
\begin{align*}
a(t) = \inf\{s\,:\, a^{-1}(s) > t\}, \quad 0 \leq t < \infty
\end{align*}
and if $f$ is non negative Borel measurable function on $[0, \infty]$ vanishing at $\infty$ one has
\begin{align*}
\int_{(0, \infty)} f(t) da(t) = \int_0^\infty f(a^{-1}(t)) dt.
\end{align*}
\end{lemma}
We also observe that in case $a$ is continuous, then
\begin{align*}
a^{-1}(t) = \max\{s\,:\, a(s)=t\}, \quad 0 \leq t < a(\infty)
\end{align*} 
and $a^{-1}(t)=\infty$ for $t> a(\infty)$. The inverse $a^{-1}$ is strictly increasing on $[0, a(\infty))$ and $a(a^{-1}(t)) = t$ for $0 \leq t < a(\infty)$.

\subsection{The process $X^+$}
\label{sec:secXref}

We recall that $X^+$ is the reflecting Brownian motion on $\bar{\Omega}$ for which, with abuse of notation, we respectively denote by
\begin{align}
\gamma^+_t = \int_0^t \mathbf{1}_{\partial \Omega} (X^+_s) ds, \quad t\geq 0 \quad \textrm{and} \quad \Gamma^+_t = \int_0^t \mathbf{1}_{\Omega} (X^+_s) ds, \quad t\geq 0
\label{defROUG}
\end{align}
the local time on the boundary $\partial \Omega$ and the occupation time of (the interior) $\Omega$. It is well known that by using It\^{o}'s formula, the local time can be obtained (in the $L^2$ sense for all $t\geq 0$ and a.s. uniformly on bounded time intervals) as
\begin{align*}
\gamma^+_t = \lim_{\epsilon \downarrow 0} \int_0^t \frac{\mathbf{1}_{\Lambda_\epsilon}(X^+_s)}{\epsilon} ds
\end{align*}
where $\Lambda_\epsilon = \{x \in \bar{\Omega}\,:\, \operatorname{dist}(x, \partial \Omega) \leq  \epsilon\}$. The representations in \eqref{defROUG} make sense in terms of the associated Revuz measures and $\gamma^+$, $\Gamma^+$ are PCAF for which
\begin{align}
\int_0^\infty \mathbf{1}_{\Omega}(X^+_t) d\gamma^+_t = 0, \quad \int_0^\infty \mathbf{1}_{\partial \Omega}(X^+_t) d\Gamma^+_t = 0.
\label{defgammaGammaRigorous}
\end{align} 
Let us consider 
\begin{align}
\tilde{m}(dy) = \mathbf{1}_\Omega dy + \mathbf{1}_{\partial \Omega} m_\partial(dy)
\label{mMeasureTilde}
\end{align}
in place of $m$ defined in \eqref{mMeasure}. Let $f \in \mathcal{B}(\bar{\Omega})$. We recall that $\gamma^+_t$ and $\Gamma^+_t$ are PCAFs for which
\begin{align}
\label{limitPCAFgamma}
\lim_{t \downarrow 0} \mathbf{E}_{\tilde{m}} \left[ \frac{1}{t} \int_0^t f(X^+_s) d\gamma^+_s \right] = \int_{\partial \Omega} f(x) m_\partial(dx)
\end{align}
and
\begin{align}
\label{limitPCAFGamma}
\lim_{t \downarrow 0} \mathbf{E}_{\tilde{m}} \left[ \frac{1}{t} \int_0^t f(X^+_s) d\Gamma^+_s \right]  = \int_\Omega f(x) dx.
\end{align}

Now we quote some results on the path integral representation of the solution to the boundary value elliptic problem
\begin{equation}
\label{prob31}
\left\lbrace
\begin{array}{ll}
\displaystyle \Delta g = c_1 g , & \textrm{ in } \Omega,\\
\displaystyle \partial_{\bf n} g + c_2 g = f , &  \textrm{ on } \partial \Omega,
\end{array}
\right.
\end{equation}
with $f \in \mathcal{B}(\partial \Omega)$.

\begin{lemma}
\label{lemmaPap}
(\cite{Pap90}) Let $c_1, c_2$ be two positive constants. 
\begin{itemize}
\item[i)] The gauge of \eqref{prob31}
\begin{align*}
\mathfrak{g}(x) = \mathbf{E}_x \left[ \int_0^\infty e^{-c_1 t - c_2 \gamma^+_t} d\gamma^+_t  \right]
\end{align*}
is continuous on $\bar{\Omega}$ if there exists $x_0 \in \bar{\Omega}$ such that $\mathfrak{g}(x_0)< \infty$.
\item[ii)] The unique (weak) solution to \eqref{prob31} has the probabilistic representation
\begin{align*}
g(x) = \mathbf{E}_x \left[ \int_0^\infty e^{-c_1 t - c_2 \gamma^+_t} f(X^+_t) d\gamma^+_t \right], \quad f \in \mathcal{B}(\partial \Omega) 
\end{align*}
if, for some $x_0 \in \bar{\Omega}$, $\mathfrak{g}(x_0)<\infty$.
\item[iii)] Given the Feynman-Kac semigroup on $L^2(\Omega)$
\begin{align*}
T_t f(x) = \mathbf{E}_x [e^{-c_1 t - c_2 \gamma^+_t} f(X^+_t)] = \int_{\Omega} f(y) \, k(t,x,y)\,  \tilde{m}(dy)
\end{align*}
we have that
\begin{align*}
g(x) = \int_{\partial \Omega} \left(\int_0^\infty k(t,x,y) dt \right) f(y) \tilde{m}(dy)
\end{align*}
where $\tilde{m}$ is defined in \eqref{mMeasureTilde}. Moreover, $k$ is a symmetric continuous kernel whose eigenfunctions are in $C(\bar{\Omega})$ and $T_t$ is a compact operator in $C(\bar{\Omega})$. $T_t$ is a Feller semigroup.
\end{itemize}
\end{lemma}
We note that Lemma \ref{lemmaPap} deals with continuous (and bounded) functions on $\bar{\Omega}$ where $\Omega$ is a bounded domain (open, connected and non-empty
set) with $C^3$ boundary $\partial \Omega$.

\subsection{The process $X$}
\label{sec:secX}

We have already introduced the process $X$ with generator $(G, D(G))$, now we complete the presentation of $X$ by recalling some further details. Our discussion  mainly follows the well-known book \cite{BluGet68} and the pioneering work \cite{ItoMcK}, both of which deal with the case $d=1$. 

Let us consider the natural filtration $ \mathcal{F}_t = \sigma\{X_s, \, 0\leq s < t\}$ and a good function $f$ (bounded Borel measurable, for example) for which $\mathbf{E}[f(X_s)| \mathcal{F}_t] = \mathbf{E}[f(X_s)|X_t]$, $t\leq s$, and $\mathbf{E}_x [f(X_{t+s}) | \mathcal{F}_t] = \mathbf{E}_{X_t} [f(X_s)]$, $s,t>0$. Hereafter, we consider the local time and the occupation time
\begin{align}
\gamma_t = \int_0^t \mathbf{1}_{\partial \Omega} (X_s) ds, \quad t\geq 0 \quad \textrm{and} \quad \Gamma_t = \int_0^t \mathbf{1}_{\Omega} (X_s) ds, \quad t\geq 0
\label{defgammaGammaRigorous2}
\end{align}
for which we observe that the inverses $\gamma^{-1}_t = \inf\{s\geq 0\,:\, \gamma_s>t\}$ and $\Gamma^{-1}_t = \inf\{s\geq 0\;:\; \Gamma_s >t\}$ are such that $\gamma \circ \gamma^{-1}_t = t$ and $\Gamma \circ \Gamma^{-1}_t = t$ almost surely. The processes in \eqref{defgammaGammaRigorous2} are well-defined as non-decreasing processes. They only increase as $X$ respectively is on $\partial \Omega$ and $\Omega$.  The following facts play a crucial role.

We say that $X$ is an elastic process meaning that 
$$\mathbf{E}_x[f(X_t)] = \mathbf{E}_x[f(X_t), t < \zeta]$$ 
is written in terms of the multiplicative functional $M_t$ and the lifetime $\zeta$ giving the elastic kill. In particular, there exists an independent exponential random variable (with parameter $c/\eta$) for which
\begin{align}
\label{lawExpLT} 
\mathbf{P}(\zeta > t | X_t) = e^{-(c/\eta) \gamma_t}=: M_t.
\end{align}

We say that $X$ is a sticky process meaning that $\{t\,:\, X_t \in \partial \Omega\}$ is a set of positive Lebesgue measure giving the sticky holding times discussed below in Section \ref{sec:comparisonXhatX}.

We write
\begin{align*}
S_t f(x)=\mathbf{E}_x[f(X_t)], \quad t\geq 0,\; x \in \bar{\Omega}
\end{align*}
where $S_t$ is the semigroup generated by $(G, D(G))$ and provide arguments for the following representation. For all $t>0$,
\begin{align}
S_t f(x) = \int_{\bar{\Omega}} f(y) p(t,x,y)\, m(dy), \quad f \in L^2(\bar{\Omega}, m),\; x \in \bar{\Omega}
\label{repScontKernel}
\end{align}
with
\begin{align*}
L^2(\bar{\Omega}, m) = L^2(\Omega, dx) \oplus L^2(\partial \Omega, (\eta/\sigma)m_\partial(dx)).
\end{align*}
We recall the following facts:
\begin{itemize}
\item[{\bf S1:}] (see \cite{AMPR03}) The semigroup $S_t$ is a compact, positive $C_0$-semigroup on $C(\bar{\Omega})$;
\item[{\bf S2:}] (see \cite{ArendtSauter2022}) The $C_0$-semigroup $S_t$ on $L^2(\bar{\Omega},m)$ has a continuous kernel on $\bar{\Omega}$. That is, for all $t>0$ there exists $p(t, \cdot, \cdot): \bar{\Omega} \times \bar{\Omega} \to \mathbb{R}$ such that \eqref{repScontKernel} holds true.
\end{itemize} 
In particular, the $C_0$-semigroup $S_t$ on $C(\bar{\Omega})$ can be given as 
\begin{align*}
S_tf(x) = \int_\Omega f(y) p(t,x,y) dy + (\eta / \sigma) \int_{\partial \Omega} f(y) p(t,x,y) m_\partial (dy), \quad f \in C(\bar{\Omega}), \quad t>0,\; x \in \bar{\Omega}
\end{align*}
from which we can also write
\begin{align*}
S_t f = S_t f|_{\Omega} + S_t f|_{\partial \Omega}
\end{align*}
where by $f|_\Omega$ we mean the restriction of $f$ on $\Omega$ . In general, we use the notation $S_t f|_{K}$ meaning that $S_0 f|_{K} = f|_{K}$ and $f|_{K}(x) = f(x)$, $x \in K \subseteq \bar{\Omega}$. We also write 
\begin{align*}
\mathbf{P}_x(X_t \in dy) = p(t,x,y)m(dy)
\end{align*}
where $X_0=x \in \bar{\Omega}$ and $m(dy)$ has been given in \eqref{mMeasure}.

We denote by 
\begin{align*}
R_\lambda f(x) := \mathbf{E}_x \left[ \int_0^\infty e^{-\lambda t} f(X_t) dt \right] \quad \lambda>0, \quad x \in \bar{\Omega}
\end{align*}
the associated $\lambda$-potential and
\begin{center}
$(\mathcal{L}p)$ denotes the Laplace transform of $p$.
\end{center}
The previous arguments suggest that
\begin{align*}
R_\lambda f(x) = \int_{\bar{\Omega}} f(y) r_\lambda(x,y) m(dy)
\end{align*}
where, for all $\lambda >0$,  $r_\lambda(\cdot, \cdot) = (\mathcal{L} p)(\lambda, \cdot, \cdot)$ is a continuous kernel. From Proposition 4.3 and Proposition 4.4 in \cite{ArendtSauter2022} we know that there exists $c_R>0$ such that
\begin{align*}
\| R_\lambda f \|_{L^\infty(\bar{\Omega})}  \leq c_R\, \big( \| f \|_{L^p(\Omega)} + \| Tf \|_{L^q(\partial \Omega)} \big), \quad p>d/2,\; q > (d-1)
\end{align*}
and 
\begin{align*}
R_\lambda f \in C(\bar{\Omega})
\end{align*}
under suitable characterization of $f$. A key ingredient in the present paper is the continuity up to the boundary of $R_\lambda f$ and $\Delta R_\lambda f$. We refer to \cite{ArendtSauter2022} and the references therein for a detailed discussion including related previous results. For example (\cite{AMPR03}), 
\begin{align*}
\max\{ \|R_\lambda f\|_{L^\infty(\Omega)}, \| TR_\lambda f \|_{L^\infty(\partial \Omega)}\} \leq \frac{1}{\lambda} \max\{ \|f\|_{L^\infty(\Omega)}, \| Tf \|_{L^\infty(\partial \Omega)}\}, \quad \lambda>0.
\end{align*}
We also write $L^p(\bar{\Omega})$ in place of $L^p(\bar{\Omega},m)$.\\

We underline the fact that a strong Markov process can only leave a holding point by a jump. Indeed, a strong Markov process cannot stay at a point for a positive (Lebesgue) amount of time and then leave that point as a continuous motion. It can jump away or leave instantaneously by reflection. Thus, in view of the time changes of $X$ we will deal with, we refer to $X$ as a strong Markov process on $\Omega$. Evidently, $X$ is a Markov process on $\bar{\Omega}$. $X$ as a time-changed Markov process will be defined in Theorem \ref{thm:MAINI}.

\subsection{The time-changed process $X^L$}
\label{sec:secXL}

The random time $L$ is the inverse of the $\alpha$-stable subordinator $H$ defined by $L_t = \inf\{s\geq 0\,:\, H_s>t \}$ as already introduced in Section \ref{sec:intro} and for which 
\begin{align}
\mathbf{P}_0(L_t < s) = \mathbf{P}_0(t < H_s), \quad t,s>0.
\label{probLprobH}
\end{align}
We assume that $H_0=0=L_0$ and write accordingly $\mathbf{P}_0$ for the associated probability measure. Denote by $l$ and $h$ the corresponding probability densities for which
\begin{align*}
\mathbf{P}_0(H_t \in ds) = h(t,s)ds, \quad \mathbf{P}_0(L_t \in ds) = l(t,s)ds
\end{align*}
and 
\begin{align}
\label{LapHandL}
\int_0^\infty e^{-\xi s} h(t,s) ds = e^{-t \xi^\alpha}, \quad \int_0^\infty e^{-\lambda t} l(t,s)dt = \frac{\lambda^\alpha}{\lambda} e^{-s \lambda^\alpha}, \quad \xi, \lambda>0.
\end{align}
We recall that
\begin{align}
\label{LapML}
\int_0^\infty e^{-\xi s} l(t,s) ds = E_\alpha(- \xi t^\alpha) \quad \textrm{with} \quad \int_0^\infty e^{-\lambda t} E_\alpha (-\xi t^\alpha) dt = \frac{\lambda^{\alpha-1}}{\lambda^\alpha + \xi}, \quad \lambda>0
\end{align}
where the Mittag-Leffler function $E_\alpha$ is analytic and such that
\begin{align}
\label{MLbound}
| E_\alpha(-\mu t^\alpha)| \leq \frac{c_E}{1+ \mu t^\alpha}, \quad t\geq 0, \; \mu>0 \quad \textrm{for a constant } c_E>0 
\end{align}
(see \cite{Bingham71, Kra03}). We underline that $E_\alpha \notin L^1(0, \infty)$ for $\alpha \in (0,1)$ and $E_\alpha \in L^2(0, \infty)$ for $\alpha \in (1/2,1)$. Moreover,
\begin{align}
\label{eigML}
D^\alpha_t E_\alpha(-\mu t^\alpha) = - \mu E_\alpha(-\mu t^\alpha), \quad t>0,\; \mu >0, \quad E_\alpha(0)=1.
\end{align}
The relation \eqref{eigML} can be easily verified from \eqref{LapML} and the fact that $D^\alpha_t$ is a convolution with
\begin{align*}
\int_0^\infty e^{-\lambda t} D^\alpha_t E_\alpha\, dt = \left( \int_0^\infty e^{-\lambda t} \frac{t^{-\alpha}}{\Gamma(1-\alpha)} dt \right) \left(\int_0^\infty e^{-\lambda t} \frac{d E_\alpha}{dt} dt \right).
\end{align*}

We have introduced the process $X^L_t := X \circ L_t$ as the process $X$ delayed by $L$. The random time $L$ may have plateaus determining plateaux for the path of $X\circ L$. The process $L$ has been extensively investigated in relation with the FCPs on bounded and unbounded domains (see for example \cite{MNV09, Chen17}). From the Markov process $X$ on $\bar{\Omega}$ with generator $(G, D(G))$, we get 
\begin{align}
\label{potFCP}
\mathbf{E}_x\left[\int_0^\infty e^{-\lambda t} f(X \circ L_t) dt \right]
= & \int_{\bar{\Omega}} f(y) \int_0^\infty e^{-\lambda t} \int_0^\infty \mathbf{P}_x(X_s \in dy) \mathbf{P}(L_t \in ds)\, dt\notag \\
= & \int_{\bar{\Omega}} f(y)  \int_0^\infty \frac{\lambda^\alpha}{\lambda} e^{-\lambda^\alpha s}\, \mathbf{P}_x(X_s \in dy) \, ds \notag \\
= & \frac{\lambda^\alpha}{\lambda} \mathbf{E}_x \left[ \int_0^\infty e^{-\lambda^\alpha t} f(X_t) dt \right]
\end{align}
which is the $\lambda$-potential for the non-Markov process $X^L$ driven by the FCP 
\begin{align}
\label{NLIVP}
D^\alpha_t w = G w, \quad w_0 = f \in D(G).
\end{align}
Thus, $X^L$ can be considered to solve \eqref{probXL1} or \eqref{probXL} as we already discussed above. In particular, the unique solution to the NLIVP  \eqref{NLIVP} can be written as
\begin{align*}
w(t,x) = \mathbf{E}_x[f(X^L_t)] = \int_0^\infty S_s f(x)\, l(t,s)\,ds, \quad t\geq 0,\; x \in \bar{\Omega}
\end{align*}
where $S_t$ is the $C_0$-semigroup on $C(\bar{\Omega})$ generated by $(G,D(G))$. References on the NLIVP \eqref{NLIVP} have been given in Section \ref{nutshell}. We also refer to \cite{CapDovFEforms19} for a general discussion and a short review. As discussed in \cite{CapDovDelRus17}, the process $X^L$ can be considered as a delayed process with infinite occupation measure for every set $\Lambda \subset \bar{\Omega}$. Assume that \eqref{ASShtINTRO} holds true. As we can see from \eqref{potFCP}, for all $x \in \bar{\Omega}$,
\begin{align}
\lim_{\lambda \to 0} \mathbf{E}_x \left[ \int_0^\infty e^{-\lambda t}\mathbf{1}_\Lambda(X^L_t) dt \right] = \mathbf{E}_x \left[ \int_0^\infty \mathbf{1}_\Lambda(X^L_t) dt \right] = \infty \quad \Lambda \subset \bar{\Omega}.
\label{occupMeasXL}
\end{align}
We refer to \cite{NLBVP-ItheHline} for a discussion on the occupation measures of $X$ on $[0, \infty)$.

\section{The process $\hat{X}$}
\label{sec:hatX}

We first discuss the process $\hat{L}$ to be considered in the time-changed process $\hat{X} := X \circ \hat{L}$. The main difficulty in the characterization of this process relies on the fact that $\hat{L}$ depends on the base process $X$, that is the process serving as the base process for the time change leading to $\hat{X}$. The random time $\hat{L} = \hat{L}(X)$ of $X$ can be roughly described as the new clock such that
\begin{align}
\hat{L}_t(X) = \left\lbrace
\begin{array}{ll}
\displaystyle t, &  \textrm{if  } \hat{X}_t \in \Omega,\\
\\ 
\displaystyle L_t, & \textrm{if } \hat{X}_t \in \partial \Omega.
\end{array}
\right. \quad  0 \leq t < (\hat{\tau}^\mathsf{ex}_\Omega \vee \hat{\tau}^\mathsf{ex}_{\partial \Omega})
\label{def1L}
\end{align}
where $\hat{\tau}^\mathsf{ex}_\Lambda = \inf\{t\,:\, \hat{X}_t \notin \Lambda \}$ if $\hat{X}_0 \in \Lambda$ and $\hat{\tau}^\mathsf{ex}_\Lambda = 0$ if $\hat{X}_0 \in \Lambda^c$. Formula \eqref{def1L} gives a picture of our process. For example, we can write
\begin{align*}
\hat{X}_t = X_t \, \mathbf{1}_{\Omega} (\hat{X}_t) + X \circ L_t\, \mathbf{1}_{\partial \Omega} (\hat{X}_t), \quad 0 \leq t < (\hat{\tau}^\mathsf{ex}_\Omega \vee \hat{\tau}^\mathsf{ex}_{\partial \Omega})
\end{align*}
However, it does not allow a clear construction of $\hat{X}$ via time change. \begin{definition}
\label{defLhat}
Let $\hat{A}$ be a PCAF for the process $X \circ \hat{L}$ on $\bar{\Omega}$. The random time $\hat{L}=\hat{L}(X)$ is the time change such that, for $f \in \mathcal{B}(\bar{\Omega})$:
\begin{itemize}
\item[i)] $\hat{A} \in \mathbf{A}^+_c(\partial \Omega) $ implies that
\begin{align}
\label{defLhatbyA1}
\int_0^t f(X\circ \hat{L}_s) d\hat{A}_s = \int_0^t f(X \circ L_s)d\hat{A}_s, 
\end{align}
\item[ii)] $\hat{A} \in \mathbf{A}^+_c(\Omega)$ implies that
\begin{align}
\int_0^t f(X\circ \hat{L}_s) d\hat{A}_s = \int_0^t f(X_s)d\hat{A}_s.
\end{align}
\end{itemize}
\end{definition}

We proceed with the following further characterization of the process $\hat{L}$.  We introduce the function $\alpha: \bar{\Omega} \to (0,1]$ defined as
\begin{align*}
\alpha(y) = 
\left\lbrace 
\begin{array}{ll}
1, & y \in \Omega,\\
\\
\alpha \in (0,1), & y \in \partial \Omega.
\end{array}
\right . , \quad y \in \bar{\Omega}.
\end{align*}
Denote by $\hat{l}$ the density of $\hat{L}$ and write
\begin{align*}
\int_0^\infty \mathbf{P}_x(\hat{L}_t(y) \in ds,  X_s \in dy) = \mathbf{P}_x(X \circ \hat{L}_t \in dy)
\end{align*}
with
\begin{align*}
\int_\Omega \hat{l}(t,s,y)m(dy) = \delta(t-s) m(\Omega), \qquad \int_{\partial \Omega} \hat{l}(t,s,y) m(dy) = l(t,s) m(\partial \Omega)
\end{align*}
where $\delta$ is the Dirac delta function (recall that $L_t \to t$ almost surely as $\alpha \uparrow 1$) and
\begin{align*}
\int_0^\infty \hat{l}(t,s,y) ds= 1, \quad y \in \bar{\Omega},\; t>0.
\end{align*}

In particular, for $\Lambda \subset \bar{\Omega}$,
\begin{align*}
\mathbf{P}_x(X \circ \hat{L}_t \in \Lambda) 
= &  \int_\Lambda \int_0^\infty \mathbf{P}_x(\hat{L}_t(y) \in ds,  X_s \in dy)\\
= & \int_{\Lambda \cap \Omega} p(t,x,y)\,dy +  (\eta/\sigma) \int_{\Lambda \cap \partial \Omega}  \left( \int_0^\infty p(s,x,y) l(t,s) ds \right) m_\partial(dy). 
\end{align*}
From  \eqref{LapHandL}, with $\lambda>0$, $x,y \in \bar{\Omega}$, we get
\begin{align}
\label{LapHATp}
\int_0^\infty e^{-\lambda t} \mathbf{P}_x(\hat{L}_t(y) \in ds,  X_s \in dy) dt  =  \frac{\lambda^{\alpha(y)}}{\lambda} e^{-s \lambda^{\alpha(y)}}\, p(s, x,y)m(dy)\, ds.
\end{align}

The time change $\hat{L}$ can be also regarded as the multi-parameter process $\{\hat{L}_t(y), (t,y) \in (0, \infty) \times \bar{\Omega} \}$ by means of which we obtain the switching process $X \circ \hat{L}_t$. The law $\hat{l}$ of $\hat{L}$, for a given $y \in \bar{\Omega}$, solves the problem of finding $\hat{l} \in C(W^{1, \infty}(0,\infty), (0, \infty))$ such that (consult \cite[Section 2]{Dov22} and \cite[Section 2]{ColDov} for general subordinator with symbol $\Phi$),
\begin{align*}
D^{\alpha(y)}_t \hat{l} = - \frac{\partial \hat{l}}{\partial x}, \quad \hat{l}(0,x) = f(x)
\end{align*} 
for a good function $f$ (for example, bounded Borel measurable), that is
\begin{align*}
\hat{l}(t,x) = \int_0^x f(x-s) \hat{l}(t,s;y)\, ds = \mathbf{E}_0 \left[f(x - \hat{L}_t(y)) \mathbf{1}_{(\hat{L}_t(y) < x)} \right].
\end{align*}
The case $\alpha(y)=1$ (for which we have the ordinary derivative) must be managed in view of the fact that $\hat{l}=\delta$ is the Dirac delta function. In case $\alpha(y) \in (0,1)$, we observe that $\hat{l}=l$ is a smooth function in both time and space variables. Moreover, $\hat{L}$ can be regarded as inverse of $\hat{H}$, that is a stable subordinator of order $\alpha(y)$.\\

An example of $\alpha(\cdot) : \bar{\Omega} \to (0,1]$ is given by
\begin{align*}
\alpha(y) = 1 - (1-\alpha) (1 - \mathbf{P}_y(\tau_{\partial \Omega} > 0)) \quad \textrm{as a function $y \to \{1, \alpha\}$ with $\alpha \in (0,1)$}  
\end{align*} 
under the Blumenthal’s 0-1 law and the prescription for $y$ to be a regular point. We remark that the dependence with $X$ can be realized in a very general way. For example,  $\alpha(t,y)$ can be written in terms of $\mathbf{P}_y(\gamma_t >0)$. In this case, due to the different scenario, we need to manage the theory discussed so far.\\

We introduce $\hat{X}$ in place of $X^L$ because of its Markovian behaviour on the interior $\Omega$. The process $\hat{X}$ can be regarded as a switching process, that is a process changing its behaviour in a given region $\Lambda$ of the domain $\bar{\Omega}$. In general, we can consider $\alpha(\cdot)$ as the space-variable order of $\hat{L}$,
\begin{align*}
\alpha(y) = 
\left\lbrace
\begin{array}{ll}
1, & y \notin \Lambda\\
\alpha \in (0,1), & y \in \Lambda
\end{array}
\right.
, \quad \Lambda \subset \bar{\Omega}
\end{align*} 
to be associated with some subordinator, for example, according to \eqref{def1L},
\begin{align*}
\hat{L}_t (y) = t \mathbf{1}_{\Omega}(y) + L_t \mathbf{1}_{\partial \Omega}(y) \quad \textrm{and} \quad \hat{H}_t (y) = H_t \mathbf{1}_{\Omega}(y) + t \mathbf{1}_{\partial \Omega}(y)
\end{align*}
with $L=H^{-1}$ and $H$ stable subordinator of order $\alpha \in (0,1)$. Then,
\begin{align*}
\hat{L}_t(y) = L \circ \hat{H}_t(y), \quad y \in \bar{\Omega}, \quad 0 \leq t < (\hat{\tau}^\mathsf{ex}_\Omega \vee \hat{\tau}^\mathsf{ex}_{\partial \Omega})
\end{align*}
and $\hat{H}_t(y)$ can be associated with the order $\alpha/\alpha(y)$. We continue our discussion in case $\Lambda=\partial \Omega$. That is, the process $\hat{X}$ changes its behaviour on the boundary. For the process $\hat{X} := X \circ \hat{L}$, we respectively consider the local time and the occupation time
\begin{align*}
\hat{\gamma}_t = \int_0^t \mathbf{1}_{\partial \Omega} (\hat{X}_s) ds, \quad t\geq 0 \quad \textrm{and} \quad \hat{\Gamma}_t = \int_0^t \mathbf{1}_{\Omega} (\hat{X}_s) ds, \quad t\geq 0.
\end{align*}
We stress the fact that the process spends a positive (Lebesgue) amount of time on the boundary. Following the previous section, we write
\begin{align*}
\mathbf{P}_x(\hat{X}_t \in dy) = \mathbf{P}_x(X \circ \hat{L}_t \in dy) = \int_0^\infty \mathbf{P}_x ( X_s \in dy,  \hat{L}_t(y) \in ds) 
\end{align*}
and, for $f \in C(\bar{\Omega})$,
\begin{align}
\label{repS}
\mathbf{E}_x[f(\hat{X}_t)]
= & \int_{\bar{\Omega}} f(y) \hat{p}(t,x,y) m(dy) \\
= & \int_\Omega f(y) p(t,x,y) dy + (\eta/\sigma)\int_{\partial \Omega} f(y) \int_0^\infty l(t,s) p(s,x,y)ds\, m_\partial(dy) \notag \\
= & S_t f|_\Omega (x) + \int_0^\infty  S_s f|_{\partial \Omega} (x)\,l(t,s)\, ds, \quad x \in \bar{\Omega}\notag \\  
= & \mathbf{E}_x[f|_{\Omega}(\hat{X}_t)] + \mathbf{E}_x[f|_{\partial \Omega}(\hat{X}_t)], \notag
\end{align}
where $l(t,s)ds = \mathbf{P}_0(L_t \in ds)$ and $S_t$ is the semigroup generated by $(G, D(G))$. Further on, with \eqref{repS} at hand, we also write
\begin{align}
\label{wXhatLaw} 
\mathsf{w}(t,x) = \mathbf{E}_x[f(\hat{X}_t)] \quad \textrm{where} \quad \mathsf{w}=\mathsf{w}_f \quad \textrm{and} \quad  \mathsf{w}_f = \mathsf{w}_{f|_\Omega} + \mathsf{w}_{f|_{\partial \Omega}}
\end{align}
with the following meaning:  $\mathsf{w}_f=\mathsf{w}_f(t,x)$ is such that $\mathsf{w}_f(0,x) = f(x)$ and for $f \in C(\bar{\Omega})$, 
\begin{align}
\label{repf12}
f = f_1 + f_2 \quad \textrm{where} \quad  f_1 = f \mathbf{1}_\Omega \quad \textrm{and} \quad f_2 = f \mathbf{1}_{\partial \Omega}.
\end{align}
Notice that, both $l$ and $p$ are continuous functions and for all $t\geq 0$, $x \in \bar{\Omega}$,
\begin{align*}
\mathbf{E}_x[f_1(\hat{X}_t)] = S_t f|_{\Omega}(x) \quad \textrm{and} \quad \mathbf{E}_x[f_2(\hat{X}_t)] = \int_0^\infty  S_s f|_{\partial \Omega} (x)\,l(t,s)\, ds.
\end{align*} 
For $\alpha=1$, since $L_t=t$ almost surely, then 
\begin{align*}
\mathbf{E}_x[f(\hat{X}_t)] = \int_\Omega f(y) p(t,x,y) dy + (\eta/\sigma) \int_{\partial \Omega} f(y) p(t,x,y) m_\partial(dy) 
\end{align*} 
for $t>0$, $x \in \bar{\Omega}$ and $f \in C(\bar{\Omega})$. Indeed, $\hat{X}$ equals in law $X$ for $\alpha=1$. \\

Assume \eqref{ASShtINTRO} holds true. We also notice that, from \eqref{LapHATp}, for all $x \in \bar{\Omega}$, $\Lambda \subseteq \bar{\Omega}$,
\begin{align*}
\lim_{\lambda \to 0} \mathbf{E}_x \left[ \int_0^\infty e^{-\lambda t}\mathbf{1}_\Lambda(\hat{X}_t) dt \right] = \mathbf{E}_x \left[ \int_0^\infty \mathbf{1}_\Lambda(X_t) dt \right] < \infty \quad \textrm{only if} \quad \Lambda \cap \partial \Omega = \emptyset
\end{align*}
whereas, 
\begin{align}
\lim_{\lambda \to 0} \mathbf{E}_x \left[ \int_0^\infty e^{-\lambda t}\mathbf{1}_\Lambda(\hat{X}_t) dt \right] = \infty \quad \textrm{if} \quad \Lambda \cap \partial \Omega \neq \emptyset.
\label{occupMeasXhat}
\end{align}
Indeed, $\hat{X}$ on $\partial \Omega$ behaves like $X^L$ for which (see \eqref{occupMeasXL})
\begin{align*}
\mathbf{E}_x \left[ \int_0^\infty \mathbf{1}_{\Lambda\cap \partial \Omega}(X^L_t) dt \right] = \infty
\end{align*}
in case $\Lambda \cap \partial \Omega \neq \emptyset$.\\

The process $\hat{X}$ inherits the behaviour of $X$. We say that $\hat{X}$ is a sticky process meaning that $\{t\,:\, \hat{X}_t \in \partial \Omega\}$ is a set of positive Lebesgue measure related to the sticky holding times $\{\hat{e}_i\}_i$ discussed in Section \ref{sec:comparisonXhatX} and the holding times discussed in Remark \ref{rmk:HThatX}. Indeed, due to the jumps of $H$,  
\begin{align*}
\mathbf{P}_x(\hat{X}_s = x, s \in [0,\mathfrak{T}))>0 \quad \textrm{if} \quad x \in \partial \Omega
\end{align*}
for some random time $\mathfrak{T}$. Thus, $\mathfrak{T}$ is a holding time.

\section{Non-local PDEs}

The Caputo-D\v{z}rba\v{s}jan derivative of a function $\varrho: [0, \infty) \to \mathbb{R}$ turns out to be well-defined if, for $\varrho^\prime = \partial \varrho/\partial s$, we have $\varrho^\prime(s, \cdot) (t-s)^{-\alpha} \in L^1(0,t)$ for all $t \geq 0$. In a bounded interval, $D^\alpha_t \varrho$ is well-defined for $\varrho(\cdot, x) \in AC(a,b)$ for all $x \in \bar{\Omega}$, $(a,b) \subset (0, \infty)$ where the set $AC(a,b)$ of the absolutely continuous functions on $(a,b)$ is the natural set of functions to be considered for the convolution operator $D^\alpha_t$. We recall that $AC(a,b) = \{\varphi \in C(a, b)\,:\, \varphi^\prime \in L^1(a,b)\}$ and observe that $AC(a,b)$ coincides with $W^{1,1}(a,b)$.  On the half line we may consider $W^{1,1}(0,\infty)$ which embeds into $L^\infty(0,\infty)$. For a function $\varrho : [0, \infty) \times \bar{\Omega} \to \mathbb{R}$, $D^\alpha_t \varrho$ is well-defined if $\varrho(\cdot, x) \in W^{1,\infty} (0,\infty)$ for all $x \in \bar{\Omega}$ and this ensures existence of the Laplace transform
\begin{align}
\int_0^\infty e^{-\lambda t} D^\alpha_t \varrho(t,x)\, dt 
= & \left( \int_0^\infty e^{-\lambda t} \frac{t^{-\alpha}}{\Gamma(1-\alpha)}dt \right) \left(\int_0^\infty e^{-\lambda t} \frac{\partial \varrho}{\partial t}(t,x) \, dt \right)\notag \\ 
= & \frac{\lambda^\alpha}{\lambda} \big(\lambda (\mathcal{L}\varrho)(\lambda, x) - \varrho(0, x) \big), \quad \lambda>0.
\label{CD-LT}
\end{align}
Formula \eqref{CD-LT} is obtained from the fact that $D^\alpha_t$ is a convolution operator. Notice that we are not asking for $\varrho(\cdot, x) \in L^1(0, \infty)$ for some $x \in \bar{\Omega}$ in the spaces below.\\

We introduce the spaces 
\begin{align*}
\hat{D}_L = \bigg\{ & \varphi: (0, \infty) \times \bar{\Omega} \to \mathbb{R} \textrm{ with $\phi= T\varphi$ such that } \frac{\partial \varphi}{\partial t}  \in C((0, \infty) \times \bar{\Omega})     \\ 
 &  \textrm{ and } \Big| \frac{\partial \phi}{\partial t}(t,x) \Big| \leq \mathfrak{g}(x)\, t^{\alpha-1},\, t >0,\, x \in \partial \Omega,\; \textrm{for some } \mathfrak{g} \in L^\infty(\partial \Omega) \bigg\}
\end{align*}
and 
\begin{align*}
\bar{D}_L = \bigg\{ & \varphi: (0, \infty) \times \bar{\Omega} \to \mathbb{R} \textrm{ with } \phi= T\varphi \textrm{ such that } \frac{\partial \phi}{\partial t} \in C((0, \infty) \times \partial \Omega)\\ 
 & \textrm{ and } \ell: s \to \frac{\partial \phi}{\partial s}(s, x) (t-s)^{-\alpha}\, \textrm{ is in } L^1(0, t) \textrm{ for all } x \in \partial \Omega,\, t>0, \bigg\}.
\end{align*}

Concerning the problems \eqref{probXL} and \eqref{probXbar}, for $f \in C(\bar{\Omega})$, we address the problem of finding a solution in $C((0, \infty) \times \bar{\Omega}) \cap \hat{D}_L$ associated with $\hat{X}$ and a solution in $C((0, \infty)\times \bar{\Omega}) \cap \bar{D}_L$ associated with $\bar{X}$. We observe that $u \in \bar{D}_L$ ensures the existence of $D^\alpha_t u$ on the boundary $\partial \Omega$. On the other hand, for $u \in \hat{D}_L$ we get 
\begin{align*}
\big| D^\alpha_t Tu \big| \leq  \frac{1}{\Gamma(1-\alpha)} \int_0^t |\mathfrak{g} (x)| s^{\alpha - 1} (t-s)^{-\alpha} ds \leq \Gamma(\alpha)\, \| \mathfrak{g}\|_{L^\infty (\partial \Omega)} .
\end{align*} 
The process $X$ with generator $(G, D(G))$ is driven by \eqref{probXbar} in case $\alpha=1$. For $\alpha \in (0,1)$, we do not have a semigroup on $C(\bar{\Omega})$ and the associated process is not a Markov process on $\bar{\Omega}$.

\subsection{Main results on $\hat{X}$: the associated NLIVP}
\label{Sec:MainXhat}

Here we obtain the characterization in terms of a non-local PDE of the time-changed process introduced above. We first obtain the following preliminary results.
\begin{lemma}
\label{lemmaequivPCAFhatNOhat}
For $f \in \mathcal{B}(\bar{\Omega})$, $\lambda>0$, $x \in \bar{\Omega}$:
\begin{itemize}
\item[i)]
\begin{align*}
\mathbf{E}_x\left[\int_0^\infty e^{-\lambda t} f(\hat{X}_t) d\hat{\gamma}_t \right] = \frac{\lambda^\alpha}{\lambda} \mathbf{E}_x \left[\int_0^\infty e^{-\lambda^\alpha t} f(X_t) d\gamma_t \right]
\end{align*}
\item[ii)]
\end{itemize} 
\begin{align*}
\mathbf{E}_x\left[\int_0^\infty e^{-\lambda t} f(\hat{X}_t) d\hat{\Gamma}_t \right] = \mathbf{E}_x \left[\int_0^\infty e^{-\lambda t} f(X_t) d\Gamma_t \right]
\end{align*}
\end{lemma}
\begin{proof}
For a PCAF $\hat{A}$ of $\hat{X}$ with $\operatorname{supp}[\mu_{\hat{A}}] = \Lambda \subset \bar{\Omega}$, according to \eqref{INTmuA} and \eqref{potintdA}, we obtain
\begin{align}
\label{AhatMeasure}
\mathbf{E}_x\left[\int_0^t f(\hat{X}_s) d\hat{A}_s \right] = \int_0^t \int_\Lambda f(y) \hat{p}(s,x,y) \mu_{\hat{A}}(dy) ds
\end{align}
and
\begin{align}
\label{AhayMeasureLap}
\mathbf{E}_x\left[ \int_0^\infty e^{-\lambda t} f(\hat{X}_t) d\hat{A}_t \right] = \int_\Lambda f(y) (\mathcal{L}\hat{p})(\lambda, x,y) \mu_{\hat{A}}(dy), \quad \lambda>0
\end{align}
where, from \eqref{LapHATp},
\begin{align*}
(\mathcal{L}\hat{p})(\lambda, x, y) 
& :=  \int_0^\infty e^{-\lambda t} \hat{p}(t,x,y) dt = \int_0^\infty \frac{\lambda^{\alpha(y)}}{\lambda} e^{-s\lambda^{\alpha(y)}} p(s,x,y) ds \\
& =: \frac{\lambda^{\alpha(y)}}{\lambda} (\mathcal{L}p)(\lambda^{\alpha(y)},x,y).
\end{align*}
Recall that $p(t,x,y)$ is the transition kernel of $X$ on $\bar{\Omega}$. 
Then,
\begin{align}
\label{INTdAtmp}
\mathbf{E}_x\left[\int_0^\infty e^{-\lambda t} f(\hat{X}_t) d\hat{A}_t \right] = 
\left\lbrace
\begin{array}{ll}
\displaystyle \frac{\lambda^\alpha}{\lambda} \int_\Lambda f(y) (\mathcal{L}p)(\lambda^\alpha, x,y) \mu_{\hat{A}}(dy), & \Lambda \subseteq \partial \Omega,\\ 
\\
\displaystyle \int_\Lambda f(y) (\mathcal{L}p)(\lambda, x,y) \mu_{\hat{A}}(dy), & \Lambda \subseteq \Omega.
\end{array}
\right.
\end{align}
From \eqref{INTdAtmp} and \eqref{AhayMeasureLap} we get the claim.
\end{proof}

Let us introduce the set of functions
\begin{align*}
\mathcal{K}(\Lambda) = \{ f: \bar{\Omega} \to \mathbb{R}\; \textrm{s.t.} \; f|_{\Lambda^c} =0,\, f|_{\Lambda} \in C(\Lambda)\}, \quad \Lambda \subseteq \bar{\Omega}.
\end{align*}
We observe that $\mathcal{K}(\Lambda)$ is not a subset of $C(\bar{\Omega}, \mathbb{R})$.

\begin{theorem}
Let us consider \eqref{wXhatLaw} and define $\tilde{\mathsf{w}}=\tilde{\mathsf{w}}(\lambda, x)$ as
\begin{align*}
\tilde{\mathsf{w}} (\lambda, x) = \int_0^\infty e^{-\lambda t} \mathsf{w}(t,x) dt, \quad \lambda >0, \quad x \in \bar{\Omega}.
\end{align*}
Then:
\begin{itemize}
\item[i)] For $f \in \mathcal{K}(\Omega)$, $\tilde{\mathsf{w}}$ is continuous on $\bar{\Omega}$ with representation
\begin{align*}
\tilde{\mathsf{w}}(\lambda,x) = \mathbf{E}_x\left[ \int_0^\infty e^{-\lambda t} f(\hat{X}_t) d\hat{\Gamma}_t \right]
\end{align*}
and satisfies
\begin{equation}
\label{pr1wHatLaw}
\left\lbrace
\begin{array}{ll}
\displaystyle \lambda \tilde{\mathsf{w}} - f = \Delta \tilde{\mathsf{w}}, &  \textrm{on } \Omega,\\
\\
\displaystyle \eta \lambda \tilde{\mathsf{w}} = - \sigma \partial_{\bf n} \tilde{\mathsf{w}} - c \tilde{\mathsf{w}}, & \textrm{on } \partial \Omega.
\end{array}
\right.
\end{equation}
\item[ii)] For $f \in \mathcal{K}(\partial \Omega)$, $\tilde{\mathsf{w}}$ is continuous on $\bar{\Omega}$ with representation
\begin{align*}
\tilde{\mathsf{w}}(\lambda, x) = \mathbf{E}_x\left[ \int_0^\infty e^{-\lambda t} f(\hat{X}_t) d\hat{\gamma}_t \right]
\end{align*}
and satisfies
\begin{equation}
\label{pr2wHatLaw}
\left\lbrace
\begin{array}{ll}
\displaystyle \lambda^\alpha \tilde{\mathsf{w}}  = \Delta \tilde{\mathsf{w}}, & \textrm{on } \Omega,\\
\\
\displaystyle \eta \frac{\lambda^\alpha}{\lambda} \big( \lambda \tilde{\mathsf{w}} - f \big) = - \sigma \partial_{\bf n}\tilde{\mathsf{w}} - c \tilde{\mathsf{w}}, & \textrm{on } \partial\Omega.
\end{array}
\right.
\end{equation}
\end{itemize}
\label{thm:alternativPap}
\end{theorem}

\begin{proof}
From Section \ref{sec:secX}, 
\begin{align*}
\int_{\bar{\Omega}} f(y) (\mathcal{L}p)(\lambda,x,y)m(dx) = R_\lambda f(x) \quad \textrm{satisfies} \quad \lambda R_\lambda f - f = \Delta R_\lambda f
\end{align*}
and $R_\lambda f \in C(\bar{\Omega})$ for $f \in L^p(\Omega)$ with $p> d/2$ and $f \in L^q(\partial \Omega)$ with $q>d-1$ which is the case here. In particular, $\sigma \partial_{\bf n} R_\lambda f + c (R_\lambda f)|_{\partial \Omega} \in C(\partial \Omega)$. 
From \eqref{wXhatLaw} and \eqref{repf12} write $\mathsf{w}_i=\mathsf{w}_{f_i}$ and
\begin{align*}
\tilde{\mathsf{w}}_i (\lambda, x) = \int_0^\infty e^{-\lambda t} \mathsf{w}_i(t,x) dt, \quad \lambda >0, \quad x \in \bar{\Omega}, \quad i=1,2
\end{align*}
where $f_1 \in \mathcal{K}(\Omega)$ and $f_2 \in \mathcal{K}(\partial \Omega)$. Thus, according to \eqref{repf12},
\begin{align}
\tilde{\mathsf{w}}_1 (\lambda, x) + \tilde{\mathsf{w}}_2(\lambda, x) = \tilde{\mathsf{w}}(\lambda, x) := \mathbf{E}_x \left[ \int_0^\infty e^{-\lambda t} f(\hat{X}_t) dt \right], \quad f_1 + f_2 = f.
\label{Rhat}
\end{align}
From Lemma \ref{lemmaequivPCAFhatNOhat}, for $\alpha \in (0,1]$,
\begin{align}
\tilde{\mathsf{w}}_1 (\lambda, x) = R_\lambda f_1 (x), \quad \tilde{\mathsf{w}}_2 (\lambda, x) = \frac{\lambda^\alpha}{\lambda} R_{\lambda^\alpha} f_2 (x), \quad \lambda >0, \quad x \in \bar{\Omega}.
\label{wRsolves}
\end{align}
Then, $\tilde{\mathsf{w}}_1$ and $\tilde{\mathsf{w}}_2$ inherit the regularity of $R_\lambda$.\\

From \eqref{repS} and $ii)$ of Lemma \ref{lemmaequivPCAFhatNOhat}, we write
\begin{align*}
\tilde{\mathsf{w}}_1(\lambda, x) = \int_\Omega f_1(y) (\mathcal{L}p)(\lambda ,x,y) \, dy = \mathbf{E}_x\left[ \int_0^\infty e^{-\lambda t} f_1(\hat{X}_t) d\hat{\Gamma}_t \right], \quad \lambda >0,\, x \in \bar{\Omega}
\end{align*}
which solves (from \eqref{wRsolves}) 
\begin{equation*}
\left\lbrace
\begin{array}{ll}
\displaystyle \lambda \tilde{\mathsf{w}}_1 - f_1 = \Delta \tilde{\mathsf{w}}_1, & \textrm{on }  \Omega,\\
\\
\displaystyle \eta (\lambda \tilde{\mathsf{w}}_1 - f_1) = - \sigma \partial_{\bf n} \tilde{\mathsf{w}}_1 - c \tilde{\mathsf{w}}_1, & \textrm{on }  \partial \Omega.
\end{array}
\right.
\end{equation*}
Since $f_1=0$ on $\partial \Omega$ and $f_1 =f|_{\Omega}$ on $\Omega$, we proved the point \textit{i)}.\\ 

From \eqref{repS} and $i)$ of Lemma \ref{lemmaequivPCAFhatNOhat} we get
\begin{align}
\tilde{\mathsf{w}}_2 (\lambda, x)  
= & (\eta/\sigma) \int_{\partial \Omega} f_2(y) \int_0^\infty \int_0^\infty e^{-\lambda t} l(t,z) dt\, p(z,x,y)dz\, m_\partial(dy)\notag \\
= & (\eta/\sigma) \frac{\lambda^\alpha}{\lambda} \int_{\partial \Omega} f_2(y) (\mathcal{L}p)(\lambda^\alpha ,x,y)\, m_\partial(dy) \label{wsfProof}\\
= & \frac{\lambda^\alpha}{\lambda} \mathbf{E}_x \left[\int_0^\infty e^{-\lambda^\alpha t} f_2(X_t) d\gamma_t \right], \quad t>0,\, x \in \bar{\Omega} \notag\\
= & \mathbf{E}_x\left[ \int_0^\infty e^{-\lambda t} f_2(\hat{X}_t) d\hat{\gamma}_t \right], \quad \lambda >0,\, x \in \bar{\Omega}. \notag
\end{align}
We proceed by considering the continuous function $R_\lambda f_2$. Since
\begin{equation*}
\left\lbrace
\begin{array}{ll}
\displaystyle \big( \lambda^\alpha R_{\lambda^\alpha} f_2 - f_2 \big) = \Delta  R_{\lambda^\alpha} f_2, & \textrm{on } \Omega,\\
\\
\displaystyle \eta (\lambda^\alpha R_{\lambda^\alpha} f_2 - f_2) = -\sigma \partial_{\bf n} R_{\lambda^\alpha} f_2 - c R_{\lambda^\alpha} f_2, & \textrm{on }  \partial\Omega
\end{array}
\right.
\end{equation*}
then by multiplying both sides for $\lambda^{\alpha}/\lambda$ and recalling \eqref{wRsolves}, we get
\begin{equation*}
\left\lbrace
\begin{array}{ll}
\displaystyle \frac{\lambda^\alpha}{\lambda} \big( \lambda \tilde{\mathsf{w}}_2 - f_2 \big) = \Delta \tilde{\mathsf{w}}_2, &  \textrm{on } \Omega,\\
\\
\displaystyle \eta \frac{\lambda^\alpha}{\lambda} \big( \lambda \tilde{\mathsf{w}}_2 - f_2 \big) = - \sigma \partial_{\bf n}\tilde{\mathsf{w}}_2 - c \tilde{\mathsf{w}}_2, & \textrm{on } \partial\Omega
\end{array}
\right.
\end{equation*}
where $f_2=0$ on $\Omega$ and $f_2 = f|_{\partial \Omega}$ on $\partial \Omega$. Thus, we proved the point \textit{ii)}.\\

We close this proof by observing that \eqref{pr2wHatLaw} can be also obtained  from Lemma \ref{lemmaPap}. Moreover, we recognize \eqref{potFCP} in \eqref{wsfProof}, this can be considered together with \eqref{CD-LT} and the Laplace transform of \eqref{NLIVP} to obtain \eqref{pr2wHatLaw}. We will come back to this point in the proof of the next theorem.

\end{proof}

Now we focus on 
\begin{align*}
\mathsf{w}_{f|_{\partial \Omega}}(t, x) = \int_{\partial \Omega} \hat{p}(t,x,y) f(y) m(dy), \quad t\geq 0,\; x \in \bar{\Omega}
\end{align*}
and study its compact representation on $L^2(\bar{\Omega}, m)$. Recall that 
\begin{align*}
\tilde{\mathsf{w}}_{f|_{\partial \Omega}}(\lambda, x) = \mathbf{E}_x\left[\int_0^\infty e^{-\lambda t} f(\hat{X}_t) d\hat{\gamma}_t \right], \quad \lambda > 0, \; x \in \bar{\Omega}.
\end{align*}

Here we ask for $f|_{\partial \Omega} \in C(\partial \Omega)$ with $f=g\mathbf{1}_{\partial \Omega}$ and $g \in C(\bar{\Omega})$. In particular, we consider
\begin{align*}
\mathcal{F}(\Lambda) = \{f = g \mathbf{1}_{\Lambda}\, :\,  g \in D(G)\}
\end{align*} 
and observe that $\mathcal{F}(\Lambda) \subset \mathcal{K}(\Lambda)$ defined above. For $f \in \mathcal{F}(\partial \Omega)$, in order to underline the dependence on $\alpha \in (0,1)$, we write
\begin{align*}
R^\alpha_\lambda f(x) = \mathbf{E}_x\left[\int_0^\infty e^{-\lambda t} f(\hat{X}_t) d\hat{\gamma}_t \right], \quad \lambda>0,\; x \in \bar{\Omega}
\end{align*}
and (from formula \eqref{LapHandL}) we recall that
\begin{align*}
\lambda R^\alpha_\lambda f = \lambda^\alpha \int_0^\infty e^{- s \lambda^\alpha} S_s f\, ds = \lambda^\alpha R_{\lambda^\alpha} f, \quad \lambda >0
\end{align*}
where $S_t$ and $R_\lambda$ has been introduced in Section \ref{sec:secX}.

\begin{theorem}
\label{lemmaWrond}
The following hold true:
\begin{itemize}
\item[i)] $\mathsf{w}_{f|_{\partial \Omega}}(t,x)$ is continuous and bounded on $(0, \infty) \times \bar{\Omega}$;
\item[]
\item[ii)] $\mathsf{w}_{f|_{\partial \Omega}}(t,x)$ is the unique solution to 
\begin{align}
\label{probFracOmega}
\mathsf{w} \in C((0,\infty) \times \bar{\Omega}),  \quad D^\alpha_t \mathsf{w} = G \mathsf{w} \;\; \textrm{in } \bar{\Omega} \quad \textrm{with} \quad \mathsf{w}_0 = f \in \mathcal{F}(\partial \Omega);
\end{align}
\item[iii)] $\mathsf{w}_{f|_{\partial \Omega}}(t,x)$ has a compact representation for all $t\geq 0$ on $L^2(\bar{\Omega})$; 
\item[]
\item[iv)] $\mathsf{w}_{f|_{\partial \Omega}}(t,x)$ has the probabilistic representation $\mathbf{E}_x[f|_{\partial \Omega}(\hat{X}_t)]$ for $t\geq 0, \; x \in \bar{\Omega}$.
\end{itemize}
\end{theorem}

\begin{proof}
We proceed point by point. \\

\textit{i)} Notice that for all $\lambda > 0$, $\tilde{\mathsf{w}}_{f|_{\partial \Omega}}(\lambda, \cdot) = \tilde{\mathsf{w}}_2(\lambda, \cdot)$ which is continuous (and bounded) on $\bar{\Omega}$ as a consequence of the previous theorem and the point $i)$ of Lemma \ref{lemmaPap}. From Lemma \ref{lemmaPap}, we also conclude that $e^{-\lambda t} \hat{p}(t,x,y)$ is continuous on $\bar{\Omega} \times \bar{\Omega}$ for all $t> 0$. Thus, $\mathsf{w}_{f|_{\partial \Omega}}$ is continuous on $(0, \infty)\times \bar{\Omega}$. This proves \textit{i)}. \\

\textit{ii)} 
We have already proved that $R^\alpha_\lambda f$ solves \eqref{pr2wHatLaw} for all $\lambda>0$. Thus, the fact that $\mathsf{w}_{f|_{\partial \Omega}}$ solves the problem \eqref{probFracOmega} comes directly from Theorem \ref{thm:alternativPap} by observing that \eqref{CD-LT} brings \eqref{probFracOmega} into \eqref{pr2wHatLaw}. Since the inverse Laplace transform $\mathsf{w}_{f|_{\partial \Omega}}$ of $R^\alpha_\lambda f$ is continuous on $(0,\infty)$, the injectivity of the Laplace transform ensures a unique correspondence between the solutions of the elliptic and parabolic problems. This proves existence. Uniqueness follows from energy method. Consider two solutions $\varpi_1, \varpi_2$ of \eqref{probFracOmega} with initial datum $f$. Then, $\varpi^* := \varpi_1 - \varpi_2$ is a solution with zero initial datum. Our aim is to show that
\begin{align*}
D^\alpha_t \| \varpi^*(t, \cdot) \|^2_{L^2(\bar{\Omega})} \leq  \int_{\bar{\Omega}} 2 \varpi^* G \varpi^*\, m(dx) \leq 0.
\end{align*} 
Recall that \eqref{probFracOmega} can be written as \eqref{probXL1} and observe that $m_\partial$ is the restriction to $\partial \Omega$ of the Hausdorff measure $\mathcal{H}^{d-1}$  with $\mathcal{H}^{d-1}(\partial\Omega)< \infty$ and therefore, for every $u \in H^1(\Omega)$ there exists a unique $u_\partial \in L^2(\partial \Omega)$ such that $Tu = u_\partial$. Indeed, $C(\bar{\Omega}) \cap H^1(\Omega)$ is dense in $H^1(\Omega)$. In case $T\varpi_1 = T\varpi_2$ we have $\varpi^* =0$ on $\partial \Omega$. This is to say, for example, that the problem on the boundary $D^\alpha_t \mathfrak{w} = B w$ with $\mathfrak{w}_0=0$ has the unique solution $\mathfrak{w}=T \varpi^* =0$. For solutions with the same boundary trace we get
\begin{align*}
\int_{\bar{\Omega}} 2 \varpi^* G \varpi^*\, m(dx) = \int_{\Omega} 2 \varpi^* G \varpi^*\, dx = - 2 \int_\Omega |\nabla \varpi^* |^2 dx \leq 0.
\end{align*}
In case we do not ask for $T\varpi_1 = T \varpi_2$ (for example $\mathfrak{w}$ could be a non-zero solution), then
\begin{align*}
& \int_{\bar{\Omega}} 2 \varpi^* G \varpi^*\, m(dx) \\
= & - 2 \int_{\Omega} \nabla \varpi^* \nabla \varpi^* dx + 2 \int_{\partial \Omega} \varpi^* \partial_{\bf n} \varpi^* m_\partial(dx)  + (\eta/\sigma) \int_{\partial \Omega} 2 \varpi^* G \varpi^* m_\partial(dx)\\
= &  - 2\int_{\Omega} |\nabla \varpi^*|^2 dx - 2 \int_{\partial \Omega} c|\varpi^*|^2 m_\partial(dx) \leq 0 . 
\end{align*} 
{\color{\magenta} Now, we only need to prove 
\begin{align}
D^\alpha_t \| \varpi^*(t, \cdot) \|^2_{L^2(\bar{\Omega})} \leq  \int_{\bar{\Omega}} 2 \varpi^* G \varpi^*\, m(dx).
\label{OnlyNeedToProve}
\end{align}}
Observe that \eqref{probFracOmega} implies
 \begin{align*}
 \int_{\bar{\Omega}} \varpi^* D^\alpha_t \varpi^*\, m(dx) = \int_{\bar{\Omega}} \varpi^* G \varpi^*\, m(dx).
 \end{align*}
 Moreover, for all $x \in \bar{\Omega}$, 
 \begin{align*}
	0 \geq & D^\alpha_t | \varpi^* |^2 - 2\varpi^* D^\alpha_t \varpi^* \\ 
	= & \frac{1}{\Gamma(1-\alpha)} \int_0^t  \left[ 2\bigg(\varpi^*(s,x) - \varpi^*(t,x) \bigg) \frac{d \varpi^*}{d s}(s,x) \right] (t-s)^{-\alpha} ds\\
	= & \frac{1}{\Gamma(1-\alpha)} \int_0^t  \left[ \frac{d}{d s} \bigg(\varpi^*(s,x) - \varpi^*(t,x) \bigg)^2  \right] (t-s)^{-\alpha} ds\\
	= & \frac{1}{\Gamma(1-\alpha)} \left( - \frac{(\varpi^*(0,x) - \varpi^*(t,x))^2}{t^\alpha} -\alpha \int_0^t \frac{\big(\varpi^*(s,x) - \varpi^*(t,x) \big)^2}{(t-s)^{\alpha+1}} ds \right), \quad t>0.
 \end{align*}
Indeed, for $(t-s) \to 0$, we use the fact that $\varpi^* \in C[0, \infty) \cap C^1(0, \infty)$ with
\begin{align*}
\frac{(\varpi^*(s,x) - \varpi^*(t,x))^2}{(t-s)^\alpha} = \left(\frac{\varpi^*(s,x) - \varpi^*(s + (t-s),x)}{t-s} \right)^2 \, (t-s)^{2-\alpha} 
\end{align*} 
and  
\begin{align*}
\frac{\big(\varpi^*(s,x) - \varpi^*(t,x)\big)^2}{(t-s)^{\alpha+1}} = \left( \frac{\varpi^*(s,x) - \varpi^*(s + (t-s),x)}{t-s} \right)^2 \, (t-s)^{1-\alpha}.
\end{align*}
Thus, {\color{\magenta} \eqref{OnlyNeedToProve} holds true and,} for the energy 
$$\mathcal{E}(t) = \| \varpi^*(t, \cdot) \|^2_{L^2(\bar{\Omega})},$$ 
we have 
$$D^\alpha_t \mathcal{E}(t) \leq 0.$$ 
Recall the fractional integral $I^\alpha_t$ for which 
$$I^\alpha_t D^\alpha_t u = u(t)- u(0)$$ 
and find
\begin{align*}
I^\alpha_t D^\alpha_t \mathcal{E}(t):= \frac{1}{\Gamma(\alpha)} \int_0^t (t-s)^{\alpha-1} D^\alpha_t \mathcal{E}(s)\, ds = \mathcal{E}(t) - \mathcal{E}(0) \leq 0
\end{align*}
for which $\mathcal{E}(t) \leq \mathcal{E}(0)=0$. We conclude that, for all $t\geq 0$, $\varpi^*$ is $m$-a.e. zero with $T\varpi^*=0$ for almost all $x \in \partial \Omega$. Since $\varpi^*$ is continuous on $\bar{\Omega}$, then $\varpi_1 = \varpi_2$ on $[0, \infty) \times \bar{\Omega}$. This proves uniqueness and \textit{ii)}.\\

\textit{iii)} Observe that, for all $t \geq 0$, we have the following representation on $L^2(\bar{\Omega}, m)$,
\begin{align*}
\mathsf{w}_{f|_{\partial \Omega}} (t,x) = \int_0^\infty l(t,s)\, S_s f|_{\partial \Omega} (x)\, ds, \quad t\geq 0,\; x \in \bar{\Omega}
\end{align*}
where $l$ has been introduced in Section \ref{sec:secXL} and the semigroup $S_t$ on $C(\bar{\Omega})$ coincides with a symmetric Markov semigroup on $L^2(\bar{\Omega})$, see \cite[Theorem 2.3]{AMPR03}. Since the semigroup is compact (\cite[Corollary 2.7]{AMPR03}), there exists a sequence of eigenvalues $\{\mu_k,\, k \in \mathbb{N}_0\}$ and a sequence of eigenfunctions $\{\psi_k,\, k \in \mathbb{N}_0\}$ such that $S_t f$ has the eigenfunction expansion on $L^2(\bar{\Omega})$ (see \cite{BvonBR06,vBelFra})  
\begin{align}
S_t f = \sum_k e^{-t \mu_k } (f, \psi_k) \, \psi_k
\label{uf0}
\end{align}
with $(f,g) := (f,g)_\Omega + (f,g)_{\partial \Omega}$ w.r.t. $m(dy) = \mathbf{1}_{\Omega} dy +(\eta/\sigma) \mathbf{1}_{\partial \Omega} m_\partial (dy)$ {\color{\magenta} where, for brevity, throughout this paper we use the following notation
\begin{align*}
\sum_{k \in \mathbb{N}_0} = \sum_k,
\end{align*}
unless otherwise specified.} We write
\begin{align*}
\int_\Lambda f \, \psi_k m(dx) = (f, \psi_k )_\Lambda, \quad \Lambda \subseteq \bar{\Omega} \quad \textrm{and} \quad (f, \psi_k)_{\bar{\Omega}} = (f, \psi_k).
\end{align*} 
From \eqref{repS} we get
\begin{align}
\mathsf{w}_{f|_{\Omega}} (t,x) = \sum_k e^{-\mu_k t}(f|_{\Omega}, \psi_k )_{\Omega}\, \psi_k(x), \quad t\geq 0,\; x \in \bar{\Omega}
\label{uf1}
\end{align}
and
\begin{align}
\mathsf{w}_{f|_{\partial \Omega}} (t,x) = \sum_k E_\alpha(-\mu_k t^\alpha) (f|_{\partial \Omega}, \psi_k )_{\partial \Omega}\, \psi_k(x) , \quad t\geq 0,\; x \in \bar{\Omega}
\label{uf2}
\end{align}
on $L^2(\bar{\Omega})$ where the last equality follows from \eqref{LapML}. The monotonicity property given in \cite[Theorem 3.5]{AMPR03} says that 
\begin{align}
\label{monotS}
S^\dagger_t \leq S_t \leq S^+_t
\end{align}
where $S^\dagger_t$ is the (Dirichlet) semigroup for the killed Brownian motion and $S^+_t$ is the semigroup for the reflected sticky Brownian motion (corresponding to the limits of $c \in (0, \infty)$ already mentioned in Section \ref{sec:intro}). Let us introduce the set of eigenvalues and eigenvetors $\{\mu^0_k, \psi^0_k \}$ corresponding to $S^+_t$ (that is, as $c\to 0$). From \cite[Theorem 3.1]{vBelFra} we know that 
\begin{align*}
0 = \mu^0_0 < \mu^0_1 \leq \mu^0_2 \leq \cdots \leq \mu^0_k
\end{align*}
and $\mu_k^0 \to \infty$ as $k \to \infty$. With \eqref{uf1} and \eqref{uf2} at hand, by considering \eqref{monotS} and the Parseval's identity, we have
\begin{align*}
\| \mathsf{w}_{f|_{\Omega}} \|_{L^2(\bar{\Omega})} \leq \exp \left(-\mu_0^0\, t \right)  \| f \|_{L^2(\Omega)} 
\end{align*}
and
\begin{align*}
\| \mathsf{w}_{f|_{\partial \Omega}} \|_{L^2(\bar{\Omega})} \leq E_\alpha(-\mu_0^0\, t^\alpha) \| f\|_{L^2(\partial \Omega)}.
\end{align*}
In general, formula \eqref{uf1} can be treated by known arguments, we refer to \cite{BvonBR06}. Since $f \in \mathcal{F}(\partial \Omega)$, then we focus on \eqref{uf2} and recall that $E_\alpha(-z)$ defined in Section \ref{sec:secXL} is completely monotone (a non increasing function of $z\geq 0$). For $f\in \mathcal{F}(\partial \Omega)$, note that
\begin{align}
\sum_k |(f, \psi_k)|^2 = \int_{\bar{\Omega}} |f|^2 m(dx) \leq   \int_{\bar{\Omega}} |g|^2 m(dx) = \sum_k |(g, \psi_k)|^2.
\label{fFourier}
\end{align}
Let {\color{\magenta} $\mu^* = \min \{\mu_k : k \in \mathbb{N}_0 \}=\mu_0>0$} with $c>0$. Observe that 
\begin{align*}
& \int_{\bar{\Omega}} \Big| \sum_{k} E_\alpha(-\mu_k t^\alpha)\, (f|_{\partial \Omega}, \psi_k)_{\partial \Omega}\, \psi_k(x) \Big|^2 m(dx) \\
\leq &  \big|E_\alpha(-\mu^* t^\alpha)\big|^2 \sum_{k} \big| (f|_{\partial \Omega}, \psi_k)_{\partial \Omega}|^2 \int_{\bar{\Omega}} |\psi_k(x) \big|^2 m(dx) \\
\leq &  \left(\frac{c_E}{1+ \mu^* t^\alpha}\right)^2 \sum_{k} \big| (f|_{\partial \Omega}, \psi_k)_{\partial \Omega} \big|^2 < \infty
\end{align*}
where we used \eqref{MLbound} and
\begin{align}
\sum_{k} \big| (f|_{\partial \Omega}, \psi_k)_{\partial \Omega} \big|^2 = \| f|_{\partial \Omega} \|^2_{L^2(\bar{\Omega})} = \| Tg \|^2_{L^2(\partial \Omega)} < \infty.
\label{L2fourierf}
\end{align}
Now write
\begin{align}
\mathsf{w}_n = \sum_{k \leq n} E_\alpha(-\mu_k t^\alpha) (f|_{\partial \Omega}, \psi_k)_{\partial \Omega}\, \psi_k(x), \quad n \in \mathbb{N}_0
\label{defSERIESwn}
\end{align}
and observe that
\begin{align*}
\int_{\bar{\Omega}}| \mathsf{w}_{f|_{\partial \Omega}} - \mathsf{w}_{n} |^2 m(dx)
\leq &  \sum_{k > n} | E_\alpha(-\mu_k t^\alpha) (f|_{\partial \Omega}, \psi_k)_{\partial \Omega}|^2 \int_{\bar{\Omega}}|\psi_k(x) |^2 m(dx)\\ 
\leq &  \sum_{k > n} \left(\frac{c_E}{1+ \mu_k t^\alpha}\right)^2 \big| (f|_{\partial \Omega}, \psi_k)_{\partial \Omega} \big|^2\\
\leq &  \left(\frac{c_E}{1+ \mu^* t^\alpha}\right)^2 \sum_{k > n} \big| (f|_{\partial \Omega}, \psi_k)_{\partial \Omega} \big|^2 \to 0 \quad \textrm{as } n\to \infty.
\end{align*}
Then, we conclude that \eqref{uf2} converges in $L^2(\bar{\Omega})$ uniformly in $t\geq 0$. The continuity w.r.t. the initial datum can be obtained by considering that $E_\alpha(z) \to 1$ as $z\to 0$ for all $\alpha \in (0,1]$ and
\begin{align*}
\Big\| \sum_k E_\alpha(-\mu_k t^\alpha) (f|_{\partial \Omega}, \psi_k)_{\partial \Omega}\, \psi_k(x)  - f \Big\|^2_{L^2(\bar{\Omega})} \leq  | E_\alpha(-\mu^* t^\alpha) - 1 |^2 \|f \|^2_{L^2(\bar{\Omega})}.
\end{align*}

Now recall that $\{\mu_k\}_k$ is a non-decreasing sequence with $\mu_k \to \infty$ as $k \to \infty$ and 
\begin{align}
\mu_k t^\alpha\, E_\alpha (-\mu_k t^\alpha) \leq c_E \frac{\mu_k t^\alpha}{1 + \mu_k t^\alpha}.
\label{boundMLfunction}
\end{align}
In particular, for $t>0$
\begin{align*}
\sum_{k > n} \left(\frac{c_E}{1+ \mu_k t^\alpha}\right)^2 \big| (f|_{\partial \Omega}, \psi_k)_{\partial \Omega}\, \mu_k \big|^2 \leq c_E^2\, t^{-2\alpha}\sum_{k > n} \big| (f|_{\partial \Omega}, \psi_k)_{\partial \Omega}\big|^2.
\end{align*}
Thus, we also get that
\begin{align}
\int_{\bar{\Omega}}| \Delta (\mathsf{w}_{f|_{\partial \Omega}} - \mathsf{w}_{n}) |^2 m(dx) 
\leq & \sum_{k > n} \left(\frac{c_E}{1+ \mu_k t^\alpha}\right)^2 \big| (f|_{\partial \Omega}, \psi_k)_{\partial \Omega}\, \mu_k \big|^2 \notag\\
\leq & c^2_E\, t^{-2\alpha} \sum_{k > n}  \big| (f|_{\partial \Omega}, \psi_k)_{\partial \Omega} \big|^2 \to 0 \quad \textrm{as } n \to \infty. \label{boundMU}
\end{align}  
The fact that \eqref{uf2} solves 
\begin{align*}
D^\alpha_t \mathsf{w} = G \mathsf{w} \;\; \textrm{in } \bar{\Omega} \quad \textrm{with} \quad \mathsf{w}_0 = f \in \mathcal{F}(\partial \Omega)
\end{align*}
on $L^2(\bar{\Omega})$ comes from \eqref{eigML} and the fact that $\{\psi_k\}_k \in D(G)$, together with the continuity w.r. to the initial datum. Indeed, the equality
\begin{align}
 (f|_{\partial \Omega}, \psi_k)_{\partial \Omega}\, \psi_k(x) \, D^\alpha_t E_\alpha(-\mu_k t^\alpha)  = E_\alpha(-\mu_k t^\alpha) (f|_{\partial \Omega}, \psi_k)_{\partial \Omega}\, \Delta \psi_k(x)
 \label{EQpointWiseID}
\end{align}
holds term by term. For the convergence of the series involving the terms in \eqref{EQpointWiseID}, it suffices to consider \eqref{boundMU}. This concludes the proof of \textit{iii)}.\\

\textit{iv)} This follows from \eqref{repS}.\\

The proof is completed.
\end{proof}

We now examine $\mathsf{w}_{f|_{\partial \Omega}}$ as a function of $\hat{D}_L$. From the proof of Theorem \ref{lemmaWrond}, for the series \eqref{uf2} we have that
\begin{align*}
& \textrm{$\mathsf{w}_{f|_{\partial \Omega}}(t, \cdot)$  is continuous as a function from $[0, \infty)$ to $L^2(\bar{\Omega})$},\\
& \textrm{$D^\alpha_t \mathsf{w}_{f|_{\partial \Omega}}(t, \cdot)$ and $\Delta \mathsf{w}_{f|_{\partial \Omega}}(t, \cdot)$ are continuous as functions from $(0, \infty)$ to $L^2(\bar{\Omega})$}.
\end{align*}
We introduce 
\begin{align*}
\hat{D}^2_L = \bigg\{ & \varphi: (0, \infty) \times \bar{\Omega} \to \mathbb{R} \textrm{ with $\phi = T\varphi$ such that: } \\
 & \textit{ i) $\frac{\partial \varphi}{\partial t}$ is continuous as a functions from $(0, \infty)$ to $L^2(\bar{\Omega})$} \\
 &  \textrm{ ii) } \Big| \frac{\partial \phi}{\partial t}(t,x) \Big| \leq \mathfrak{g}(x)\, t^{\alpha-1}, \,  t>0, \, x \in \partial \Omega,\;   \mathfrak{g} \in L^2(\partial \Omega) \bigg\}
\end{align*}
and provide the following result.
\begin{theorem}
Let us consider \eqref{uf2}. Then, for $f \in \mathcal{F}(\partial \Omega)$, uniformly on $\bar{\Omega}$,
\begin{align}
\Big| \frac{\partial}{\partial t} T \mathsf{w}_{f|_{\partial \Omega}}(t,x) \Big| \leq \mathfrak{g}(x)\, t^{\alpha-1}, \quad t>0,\, x \in \partial \Omega, \quad \mathfrak{g} \in L^2(\partial \Omega) \label{Ip2Dhat2}
\end{align}
with
\begin{align*}
\mathfrak{g}(x) = c_E\, \Big| \sum_k  (g, \psi_k)_{\partial \Omega}\, \big(\Delta \psi_k\big)(x) \Big|.
\end{align*}
\label{ThmDhatL2}
\end{theorem}
\begin{proof}
according to \eqref{MLbound} and \cite[equation (17)]{Kra03} we get
\begin{align*}
\Big| \frac{\partial}{\partial t} \sum_k E_\alpha(-\mu_k t^\alpha) (f|_{\partial \Omega}, \psi_k)_{\partial \Omega}\, \psi_k \Big| \leq & \, c_E\,  t^{\alpha-1} \Big| \sum_k  (f|_{\partial \Omega}, \psi_k)_{\partial \Omega}\, \psi_k \, \mu_k \Big| = c_E\, t^{\alpha-1} |y|
\end{align*}
where
\begin{align*}
y(x) := \sum_k  (f|_{\partial \Omega}, \psi_k)_{\partial \Omega}\, \mu_k\, \psi_k(x) = \sum_k  (g, \psi_k)_{\partial \Omega}\, \big(\Delta \psi_k\big)(x) =:\Delta y^*(x), \quad x \in \bar{\Omega}.
\end{align*}
From \eqref{fFourier}, 
\begin{align}
\|\Delta y^* \|_{L^2(\bar{\Omega})} = \sum_k  |(g, \psi_k)_{\partial \Omega}|^2\, \mu^2_k \leq \sum_k  |(g, \psi_k)|^2\, \mu^2_k = \|\Delta g \|_{L^2(\bar{\Omega})}.
\label{normL2y}  
\end{align}
Since $\Delta y^* \in L^2(\bar{\Omega})$, then \eqref{Ip2Dhat2} is satisfied. \\

This concludes the proof.
\end{proof}

Observe that for $f \in \mathcal{F}(\partial \Omega)$ we have
\begin{align*}
\Big\| \int_0^\infty e^{-\lambda t}  \frac{\partial}{\partial t}  \mathsf{w}_{f|_{\partial \Omega}} dt \Big\|_{L^\infty(\bar{\Omega})} = \| \lambda R^\alpha_\lambda g \mathbf{1}_{\partial \Omega} - g \mathbf{1}_{\partial \Omega} \|_{L^\infty(\bar{\Omega})} \leq 2\| g \|_{L^\infty(\partial \Omega)}
\end{align*}
and recall that 
\begin{align*}
\frac{\partial}{\partial t} T \mathsf{w}_{f|_{\partial \Omega}} = T \frac{\partial}{\partial t}  \mathsf{w}_{f|_{\partial \Omega}}.
\end{align*}
With \eqref{normL2y} in mind, from \eqref{Ip2Dhat2}, we get
\begin{align}
\Big\| \frac{\partial}{\partial t} \mathsf{w}_{f|_{\partial \Omega}} \Big\|_{L^1(\partial \Omega)} \leq \int_{\bar{\Omega}} \Big| \frac{\partial}{\partial t} \mathsf{w}_{f|_{\partial \Omega}}(t, x) \Big| m(dx) \leq t^{\alpha -1}  \, c_E\,  m(\bar{\Omega}) \| \Delta g \|_{C(\bar{\Omega})}
\label{boundL1timeDer}
\end{align}
where the second inequality follows by observing that $\|\mathfrak{g} \|_{L^1(\bar{\Omega})} \leq c_E\, m(\bar{\Omega}) \| \Delta g \|_{L^\infty (\bar{\Omega})}$ whereas the first one is trivially verified and has been considered only in view of the next result. That is, for $f$ in a subset of $\mathcal{F}(\partial \Omega)$, we show that
\begin{align}
\Big\| \int_0^\infty  e^{-\lambda t} \frac{\partial}{\partial t} \mathsf{w}_{f|_{\partial \Omega}} \, dt \Big\|_{L^1(\partial \Omega)} \leq & m(\bar{\Omega}) \left( \| g\|_{L^\infty(\bar{\Omega})} + \lambda^{-\alpha}\,  \| \Delta g \|_{L^\infty (\bar{\Omega})} \right), \quad \lambda>0
\label{IpDhat1}
\end{align}
and a quick reading is obtained by recalling the Euler integral
\begin{align*}
\lambda^{-\alpha} = \int_0^\infty e^{-\lambda t} \frac{t^{\alpha-1}}{\Gamma(\alpha)} dt.
\end{align*}
For the sake of simplicity we provide this result only dealing with the simple cases $g>0$ or $g<0$, that is $g$ never change its sign on $\bar{\Omega}$. In particular, we assume that
\begin{align}
g \in D(G) \quad \textrm{such that} \quad \{g\geq f\} = \bar{\Omega} \quad \textrm{or} \quad \{g \leq f\} = \bar{\Omega}.
\label{ipotesysDhat}
\end{align}
\begin{theorem}
Assume $f \in \mathcal{F}(\partial \Omega)$ under \eqref{ipotesysDhat}. Then, \eqref{IpDhat1} holds true.
\end{theorem}
\begin{proof}


The fact that $g \in C(\bar{\Omega})$ implies that $g\leq 0$ if $\{g \leq f\} = \bar{\Omega}$ and $g\geq 0$ if $\{g\geq f\}=\bar{\Omega}$. Observe that $\{g < f\} \subset \{g < 0 \}$ and $\{g > f \} \subset \{g > 0 \}$ with $\lambda^\alpha R_{\lambda^\alpha} f - f = \lambda^\alpha R_{\lambda^\alpha} g - g$ on $\{f=g\} = \partial \Omega$, that is $f=g$ pointwise and
\begin{align*}
\int_{\partial \Omega} | \lambda^\alpha R_{\lambda^\alpha} f - f | m(dx) = \int_{\partial \Omega} | \lambda^\alpha R_{\lambda^\alpha} g \mathbf{1}_{\partial \Omega} - g \mathbf{1}_{\partial \Omega} | m(dx).
\end{align*}
We also notice that $g=0$ on $\{g=f\}$ is the trivial case $f=0$ on $\bar{\Omega}$. 

If $g\geq f \geq 0$, then
\begin{align*}
0 \leq S_t (g-f) = S_t g - S_t f \quad \textrm{implies} \quad R^\alpha_\lambda f \leq R^\alpha_\lambda g \quad \textrm{on } \bar{\Omega}
\end{align*}
from the positivity of $S_t$. Thus, $\lambda^\alpha R_{\lambda^\alpha} f - g$ on $\bar{\Omega}$ satisfies 
\begin{align*}
- |g| \leq -g \leq \lambda^\alpha R_{\lambda^\alpha} f - g \leq \lambda^\alpha R_{\lambda^\alpha} g - g \leq | \lambda^\alpha R_{\lambda^\alpha} g - g |.
\end{align*}
Moreover, we have
\begin{align*}
- | \lambda^\alpha R_{\lambda^\alpha} g - g | - |g| \leq \lambda^\alpha R_{\lambda^\alpha} f - g \leq |g| + | \lambda^\alpha R_{\lambda^\alpha} g - g |  \quad \textrm{on } \bar{\Omega}. 
\end{align*}

If $g \leq f \leq 0$, then $\lambda^\alpha R_{\lambda^\alpha} f - g$ on $\bar{\Omega}$ satisfies
\begin{align*}
-  | \lambda^\alpha R_{\lambda^\alpha} g - g |  \leq \lambda^\alpha R_{\lambda^\alpha} g - g \leq  \lambda^\alpha R_{\lambda^\alpha} f - g \leq -g \leq |g|
\end{align*}
and we write
\begin{align*}
- | \lambda^\alpha R_{\lambda^\alpha} g - g | - |g|  \leq \lambda^\alpha R_{\lambda^\alpha} f - g \leq |g| + | \lambda^\alpha R_{\lambda^\alpha} g - g |  \quad \textrm{on } \bar{\Omega}.
\end{align*}
This means that $| \lambda^\alpha R_{\lambda^\alpha} f - f | = | \lambda^\alpha R_{\lambda^\alpha} f - g |$ on $\partial \Omega$ and  
\begin{align*}
| \lambda^\alpha R_{\lambda^\alpha} f - g | \leq |g| + | \lambda^\alpha R_{\lambda^\alpha} g - g | \quad \textrm{on } \bar{\Omega}, 
\end{align*}
that is, under \eqref{ipotesysDhat},
\begin{align}
\int_{\partial \Omega}  |\lambda^\alpha R_{\lambda^\alpha} f - g|  m(dx) \leq \int_{\bar{\Omega}} \left( |g| + |\lambda^\alpha R_{\lambda^\alpha} g - g|\right)  m(dx).
\label{boundL1gRg}
\end{align}

We now consider $G^*$ on $L^2(\Omega) \oplus L^2(\partial \Omega)$ for the semigroup $S_t$ on $L^2(\Omega) \oplus L^2(\partial \Omega)$, 
\begin{align*}
& D(G^*)=\{(\varphi, T\varphi) \in H^1(\Omega) \oplus L^2(\partial \Omega) :\, \Delta \varphi \in L^2(\Omega),\, \partial_{\bf n}\varphi \in L^2(\partial \Omega)\}\\
& G^* (\varphi, T\varphi) = (\Delta \varphi, \, -\partial_{\bf n} \varphi - c T\varphi)
\end{align*}
and $G$ on $C(\bar{\Omega})$ for the semigroup $S_t$ on $C(\bar{\Omega})$ with the inequality
\begin{align}
\| R_{\lambda} g \| _{L^\infty (\bar{\Omega})} \leq c_R\, \lambda^{-1} \big( \| g\|_{L^p(\Omega)} + \| g\|_{L^q(\partial \Omega)} \big), \quad p>d/2, \; q>d-1.
\label{contractionR}
\end{align}
We mainly refer to \cite{AMPR03, ArendtSauter2022} and Section \ref{sec:secX}. The semigroup $S_t$ on $L^2(\bar{\Omega})$ is contractive in both $L^2$ and $L^\infty$. Since $S_t$ is differentiable and strongly continuous in $L^2$, then for $g^* \in D(G^*)$,
\begin{align}
\Big\| \frac{d}{d t} S_t g^* \Big\|_{L^2(\bar{\Omega})}  = \| S_t \Delta g^* \|_{L^2(\bar{\Omega})} \leq \| \Delta g^* \|_{L^2(\bar{\Omega})}.
\label{SGcommute}
\end{align}
Thus, for $S_t$ on $L^2(\bar{\Omega})$ and $g^* \in D(G^*)$,  
\begin{align*}
\| \lambda R^\alpha_\lambda g^* - g^* \|_{L^2(\bar{\Omega})} 
= & \Big\| \lambda^\alpha \int_0^\infty e^{-\lambda^\alpha s} S_s g^* \,ds - g^* \Big\|_{L^2(\bar{\Omega})}\\
= & \Big\| \int_0^\infty e^{-\lambda^\alpha s} \frac{d}{ds} S_s g^* \,ds \Big\|_{L^2(\bar{\Omega})} \\
\leq & \lambda^{-\alpha} \Big\| \frac{d}{ds} S_s g^* \Big\|_{L^2(\bar{\Omega})} \\
\leq & \lambda^{-\alpha} \|\Delta g^* \|_{L^2(\bar{\Omega})}.
\end{align*}
Consider $g \in D(G) \subset D(G^*)$, then
\begin{align}
\| \lambda R^\alpha_\lambda g - g \|_{L^2(\bar{\Omega})} \leq  \lambda^{-\alpha}\, \| \Delta g \|_{L^2(\bar{\Omega})} \leq \lambda^{-\alpha}\, \sqrt{m(\bar{\Omega})}\, \|\Delta g \|_{L^\infty(\bar{\Omega})}
\label{boundCOMMUTEg}
\end{align}
where the last step is obtained from the fact that $L^\infty \subset L^2$. For $f \in \mathcal{F}(\partial \Omega)$, 
\begin{align*}
\| f \|_{L^2(\bar{\Omega})} = \| g \mathbf{1}_{\partial \Omega} \|_{L^2(\bar{\Omega})} = \| Tg \|_{L^2(\partial \Omega)} \leq C_H\, \| g \|_{H^1(\Omega)}, \quad C_H>0
\end{align*}
where $T$ is a continuous operator for $H^1(\Omega)$ into $L^2(\partial \Omega)$. From \eqref{boundL1gRg} and \eqref{boundCOMMUTEg} we get
\begin{align*}
\| \lambda^\alpha R_{\lambda^\alpha} f - f \|_{L^1(\partial \Omega)} 
\leq & \| g\|_{L^1(\bar{\Omega})} + \sqrt{m(\bar{\Omega})} \|\lambda^\alpha R_{\lambda^\alpha} g - g\|_{L^2(\bar{\Omega})}\\ 
\leq & m(\bar{\Omega}) \left( \| g\|_{L^\infty(\bar{\Omega})} + \lambda^{-\alpha}\,  \| \Delta g \|_{L^\infty (\bar{\Omega})} \right)
\end{align*}
which is \eqref{IpDhat1}. With the help of \eqref{contractionR}, as in \cite{ArendtSauter2022}, we realize the Wentzell-Robin boundary condition  on $C(\bar{\Omega})$. In particular, for $S_t$ on $C(\bar{\Omega})$, given $g \in D(G)$, $\Delta S_t g = S_t \Delta g$ holds pointwise in $\bar{\Omega}$. 
\end{proof}

Concerning the compact representation in \textit{iii)} of Theorem \ref{lemmaWrond}, we now examine the convergence of the series \eqref{uf2} for $\alpha \in (0,1]$ and $d>1$.\\

({\it $\alpha=1$ and $d>1$}) Formula \eqref{boundMLfunction} played a role in the proofs above. Here we discuss the case $\Omega \subset \mathbb{R}^d$ and $\alpha=1$ for which the Mittag-Leffler function becomes an exponential function. Thus, we deal with \eqref{uf0}. First we observe that
\begin{align*}
S_t f = \sum_{k \geq 0} (S_t \psi_k)\, (f, \psi_k) \quad \textrm{where} \quad S_t \psi_k = e^{-t \mu_k} \psi_k, \quad  S_t \Delta \psi_k =  - \mu_k\, e^{-t \mu_k} \psi_k = \frac{d}{dt} S_t \psi_k.
\end{align*}
Now, for $t>0$,
\begin{align*}
\| S_t \psi_k \|_{L^2(\bar{\Omega})} 
= & \Big\| \int_{\bar{\Omega}} p(t,x,y) \psi_k(y) m(dy) \Big\|_{L^2(\bar{\Omega})}\\
\leq & \int_{\bar{\Omega}}  \left( \int_{\bar{\Omega}}  \big| p(t,x,y) \psi_k(y) \big| m(dy) \right)^2 m(dx)\\
\leq & \int_{\bar{\Omega}} \| p(t,x, \cdot) \|_{L^2(\bar{\Omega})} \|\psi_k \|_{L^2(\bar{\Omega})} \, m(dx)\\
\leq & C_{\bar{\Omega}}\, \| p(t, \cdot, \cdot) \|_{L^\infty(\Omega \times \Omega)}, \quad C_{\bar{\Omega}} >0
\end{align*}
and (H\"{o}lder's inequalites and Parseval's idenity)
\begin{align*}
\sum_{k \geq 0} |S_t \psi_k\, (f, \psi_k)| 
\leq & \sqrt{\sum_{k \geq 0} |S_t \psi_k|^2}\, \sqrt{\sum_{k \geq 0} |(f,\psi_k) |^2 } = \| S_t \psi_k \|_{L^2(\bar{\Omega})} \, \| f\|_{L^2(\bar{\Omega})}
\end{align*}
leads to
\begin{align*}
\sum_{k \geq 0} |S_t \psi_k\, (f, \psi_k)| \leq C_{\bar{\Omega}}\, \| p(t, \cdot, \cdot) \|_{L^\infty(\Omega \times \Omega)}\, \| f\|_{L^2(\bar{\Omega})}.
\end{align*}
In particular, for $n \in \mathbb{N}$, for $t>0$ and $x \in \bar{\Omega}$,
\begin{align}
\sum_{k \geq n} |S_t \psi_k\, (f, \psi_k)| \leq C_{\bar{\Omega}}\, \| p(t, \cdot, \cdot) \|_{L^\infty(\Omega \times \Omega)}\, \sqrt{\sum_{k \geq n} |(f, \psi_k)|^2} \to 0
\label{UnifNormS}
\end{align}
as $n\to \infty$ uniformly. Since \eqref{uf0} holds in $L^2(\bar{\Omega})$, then the series converges absolutely and uniformly for every $t>0$ and $x \in \bar{\Omega}$ (see also Section \ref{sec:secX}). This can be regarded as a spectral consequence of ultracontractivity. Indeed (\cite[Proposition 2.6]{AMPR03}), for $t>0$, 
\begin{align}
\| S_t g \|_{L^\infty(\bar{\Omega})} \leq t^{-\theta}\, \| g \|_{L^2(\bar{\Omega})}, \quad \textrm{with $\theta>1/2$ and depending on $d$} .
\label{UltraContrctS}
\end{align}
These arguments seem to be useful to prove the convergence of \eqref{uf2}. However, an exhaustive discussion on the convergence of the series $\mathsf{w}_{f|_{\partial \Omega}}$ (under weak regularity conditions for the initial datum $f$) is beyond the scope of the present work. Thus, we mainly focus on the case $d=2$ with $\alpha \in (0,1)$.\\

({\it $\alpha \neq 1$ and $d>1$}) Consider the series $\mathsf{w}_{f|_{\partial \Omega}}$ and the sequence $\{\mathsf{w}_n\}_n$ respectively defined in \eqref{uf2} and \eqref{defSERIESwn}. First we observe that, for $d > 1$, 
\begin{align*}
\| \mathsf{w}_n - \mathsf{w}_{f|_{\partial \Omega}}\|_{L^p(\bar{\Omega})} \to 0 \quad \textrm{as $n \to \infty$ for $p \in [1,2]$}.
\end{align*}
This follows from \textit{iii)} of Theorem \ref{lemmaWrond} and the fact that $L^p \subset L^2$. We provide the following unrefined result for planar domains of interest in the present work. Nonetheless, we deal with the worst-case scenario where $c=0$. 

\begin{theorem}
Assume $c=0$ and $d=2$.  Then, $\mathsf{w}_{f|_{\partial \Omega}}$ converges uniformly in $t>0$ and absolutely in $L^1(\bar{\Omega})$, that is  
\begin{align}
\sum_{k} \| E_\alpha(-\mu_k t^\alpha)\, (f|_{\partial \Omega}, \psi_k)\, \psi_k \|_{L^1(\bar{\Omega})} < \infty
\label{absL1}
\end{align}
and the series $\mathsf{w}_{f|_{\partial \Omega}}$ converges (absolutely) almost everywhere in $\bar{\Omega}$ .
\end{theorem}
\begin{proof}
We write here $\mu_k = \mu_k^0$, $k \in \mathbb{N}_0$ to streamline the notation. According to the proof of Theorem \ref{lemmaWrond}, the eigenvalue $\mu_k^0$ corresponds to $c=0$. From standard arguments we get the rough estimate, for $k>0$,
\begin{align}
|(g , \psi_k)| = \frac{1}{\mu_k} |(g, \Delta \psi_k )| = \frac{1}{\mu_k} |(\Delta g,\psi_k )| \leq \frac{1}{\mu_k}\, \sqrt{m(\bar{\Omega})}\, \| \Delta g \|_{L^2(\bar{\Omega})}.
\label{boundFourier}
\end{align}
Indeed, $-\mu_k \psi_k = \Delta \psi_k$ and
\begin{align*}
\int_{\bar{\Omega}} \psi_k \Delta g = \int_{\bar{\Omega}} g \, \Delta \psi_k, \quad g \in D(G).
\end{align*}
Now we use the fact that $\mu_k$ grows like $k^{1/(d-1)}$ (see \cite[Section 4]{vBelFra}) and
\begin{align}
\sum_{k \geq 1} \frac{1}{\mu_k^2} < \infty \quad \textrm{for} \quad d < 3
\label{condConvEig}
\end{align}
together with
\begin{align*}
\| \psi_k \|_{L^1(\bar{\Omega})} \leq \sqrt{m(\bar{ \Omega})} \|\psi_k\|_{L^2(\bar{\Omega})}.
\end{align*}
We also consider that $\| E_\alpha(-\mu_0 t^\alpha)\, (g, \psi_0)\, \psi_0 \|_{L^1(\bar{\Omega})} = m(\bar{\Omega}) |(g, \psi_0)\, \psi_0|$ and write
\begin{align*}
\sum_{k \geq 1} = \sum_k.
\end{align*} 
From \eqref{MLbound} and \eqref{boundMLfunction} we get
\begin{align*}
\sum_{k} \| E_\alpha(-\mu_k t^\alpha)\, (g, \psi_k)\, \psi_k \|_{L^1(\bar{\Omega})} \leq & \sqrt{m(\bar{ \Omega})} \sum_{k} \frac{c_E}{1 + \mu_k t^\alpha} |(g, \psi_k)|\\ 
\leq & \sqrt{m(\bar{ \Omega})}\, c_E\, t^{-\alpha} \sum_k \frac{|(g, \psi_k)|}{\mu_k}
\end{align*}
and, from \eqref{boundFourier},
\begin{align}
\sum_k \| E_\alpha(-\mu_k t^\alpha)\, (g, \psi_k)\, \psi_k \|_{L^1(\bar{\Omega})} \leq &  c_E\, t^{-\alpha}\, \| \Delta g \|_{L^2(\bar{\Omega})} \, m(\bar{ \Omega})\,\sum_k \frac{1}{\mu^2_k}
\label{ConvSg}
\end{align}
proves convergence for $d<3$. 

Let us consider $f \in \mathcal{F}(\partial \Omega)$, we obtain from \eqref{MLbound}
\begin{align*}
& \sum_k \| E_\alpha(-\mu_k t^\alpha)\, (f|_{\partial \Omega}, \psi_k)\, \psi_k \|_{L^1(\bar{\Omega})}\\
\leq & \sqrt{m(\bar{ \Omega})} \sum_{k} \frac{c_E}{1 + \mu_k t^\alpha} |(f|_{\partial \Omega}, \psi_k)|\\
\leq &  \sqrt{m(\bar{ \Omega}) \sum_{k} \left(\frac{c_E}{1 + \mu_k t^\alpha}\right)^2  \sum_k |(f|_{\partial \Omega}, \psi_k)|^2}\\
\leq & \sqrt{m(\bar{ \Omega}) \sum_{k} \left(\frac{c_E}{1 + \mu_k t^\alpha}\right)^2 \, \| g \|^2_{L^2(\bar{\Omega})}}
\end{align*}
where the last inequality comes from \eqref{fFourier} and, from \eqref{boundMLfunction}, 
\begin{align*}
\sum_{k} \left(\frac{c_E}{1 + \mu_k t^\alpha} \right)^2 \leq c_e^2\, t^{-2\alpha} \sum_k \frac{1}{\mu_k^2}.
\end{align*}
Thus, 
\begin{align}
\sum_k \| E_\alpha(-\mu_k t^\alpha)\, (f|_{\partial \Omega}, \psi_k)\, \psi_k \|_{L^1(\bar{\Omega})} \leq  c_E\, t^{-\alpha} \|g \|_{L^2(\bar{\Omega})}\, \sqrt{m(\bar{\Omega}) \, \sum_k \frac{1}{\mu_k^2}}
\label{ConvSf}
\end{align}
and, from \eqref{condConvEig}, the series converges for $d<3$. Formulae \eqref{ConvSg} and \eqref{ConvSf} both imply convergence in $L^1(\bar{\Omega})$ from which we obtain, uniformly in $t>0$, absolute convergence a.e. in $\bar{\Omega}$. For example, from \eqref{ConvSf} we write
\begin{align*}
\sum_k \int_{\bar{\Omega}} | E_\alpha(-\mu_k t^\alpha)\, (f|_{\partial \Omega}, \psi_k)\, \psi_k | m(dx) = \int_{\bar{\Omega}} \left( \sum_k  | E_\alpha(-\mu_k t^\alpha)\, (f|_{\partial \Omega}, \psi_k)\, \psi_k | \right)\, m(dx)
\end{align*}
and conclude that the series converges absolutely a.e. in $\bar{\Omega}$. 
\end{proof}

The previous result is obtained under the assumption of weak regularity, namely $f \in H^1(\Omega)$. This does not contradict the convergence in $L^2(\bar{\Omega})$ which is valid for $d\geq 2$ due to \eqref{L2fourierf}.\\

The dynamic boundary condition implies an asymptotic behaviour (with growth order $1/q$ and $q=d-1$) which differs from the behaviour in case of classical (non-dynamic) boundary condition (with growth order $1/p$ and $p=d/2$). The spectral growth order decreases with respect to the classical cases (see \cite{vBelFra}). \\

We provide some comment based on the previous results and conclude the discussion about the process $\hat{X}$.\\

({\it Switching behaviour}) Recall formulas \eqref{repf12} and \eqref{wXhatLaw}. Since $\hat{\gamma}_t$ is absolutely continuous w.r. to the Lebesgue measure, then we can write \eqref{Rhat} as
\begin{align*}
\hat{R}_\lambda f(x) = \mathbf{E}_x\left[\int_0^\infty e^{-\lambda t} f(\hat{X}_t) d(\hat{\Gamma}_t + \hat{\gamma}_t) \right], \quad \lambda>0, \quad x \in \bar{\Omega}, \quad f \in C(\bar{\Omega}).
\end{align*}
We also write $\mathsf{w}= \mathsf{w}_{f|_\Omega} + \mathsf{w}_{f|_{\partial \Omega}}$ meaning that $\mathsf{w}_{f|_\Omega}: [0, \infty)\times \bar{\Omega} \to \mathbb{R}$ solves
\begin{equation}
\label{probInterior}
\left\lbrace
\begin{array}{ll}
\displaystyle \frac{\partial}{\partial t} \mathsf{w} = \Delta \mathsf{w}, \quad (0, \infty) \times \Omega\\
\\
\displaystyle \eta T \frac{\partial}{\partial t} \mathsf{w} = - \sigma \partial_{\bf n} \mathsf{w} - c \mathsf{w},  \quad (0, \infty) \times \partial \Omega,\\
\\
\displaystyle \mathsf{w} (0,x)= f_1(x), \quad x \in \bar{\Omega}, \quad f_1 \in \mathcal{F}(\Omega)
\end{array}
\right. 
\end{equation}
whereas $\mathsf{w}_{f|_{\partial \Omega}} : [0, \infty)\times \bar{\Omega} \to \mathbb{R}$ solves
\begin{equation}
\label{probBoundary}
\left\lbrace
\begin{array}{ll}
\displaystyle D^\alpha_t \mathsf{w} = \Delta \mathsf{w},  \quad (0, \infty) \times \Omega,\\
\\
\displaystyle \eta T D^\alpha_t \mathsf{w} = - \sigma \partial_{\bf n} \mathsf{w} - c \mathsf{w},  \quad (0, \infty) \times \partial \Omega,\\
\\
\displaystyle \mathsf{w} (0,x)= f_2(x), \quad x \in \bar{\Omega}, \quad f_2 \in \mathcal{F}(\partial \Omega)
\end{array}
\right.
\end{equation}
with $f=f_1 + f_2$ according to \eqref{repf12}. By following the same arguments as in \eqref{SameArgs}, the conditions above take respectively the forms
\begin{align*}
\eta \frac{\partial}{\partial t} T \mathsf{w}_{f|_{\Omega}} = -\sigma \partial_{\bf n} \mathsf{w}_{f|_{\Omega}} - c \, \mathsf{w}_{f|_{\Omega}} \quad \textrm{on } \partial \Omega 
\end{align*}
and
\begin{align*}
\eta D^\alpha_t T \mathsf{w}_{f|_{\partial\Omega}} = -\sigma \partial_{\bf n} \mathsf{w}_{f|_{\partial\Omega}} - c \, \mathsf{w}_{f|_{\partial\Omega}} \quad \textrm{on } \partial \Omega.
\end{align*}

\vspace{.5cm}

({\it Existence and uniqueness}) We recall that $(G,D(G))$ generates a semigroup on $L^2(\bar{\Omega}, m)$ and the functions $\mathsf{w}_{f|_{\Omega}}$, $\mathsf{w}_{f|_{\partial \Omega}}$ are well-defined. The initial data $f_1$,$f_2$ are not in $D(G)$ but we still have a classical solution of the Cauchy problem \eqref{probInterior} equivalent to \eqref{probX} and the fractional Cauchy problem \eqref{probBoundary} equivalent to \eqref{probXL1}. Indeed, $D(G)\subset L^2(\bar{\Omega})$ and $S_t$ is associated with a maximal monotone (it is a contraction of $C_0$-semigroup) and self-adjoint (\cite{GrVo17}) operator. Thus, $f_1,f_2 \in L^2(\bar{\Omega})$ ensures existence and uniqueness of both solutions (\cite[Theorem VII.7]{Brezis}). Moreover, in both cases we have a compact representation written in terms of the semigroup $S_t$ of $X$.

\begin{remark}
($1$-dimensional case) As an example, we consider the problem \eqref{probXL} on $\Omega=(0,1)$. For the sake of simplicity we fix $c=0$, $\eta=\sigma=1$. Then, we get
\begin{align*}
w(t,x) = \sum_{k} E_\alpha(-\mu_k t) \, \psi_k(x) \, f_k, \quad t\geq 0, \; x \in (0,1)
\end{align*}
where $\mu_k = (\pi k)^2$, $\psi_k(x) = \cos(x \sqrt{\mu_k}) - \sqrt{\mu_k} \sin (x \sqrt{\mu_k})$ and
\begin{align*}
f_k = \int_{[0,1]} f(x)\, \psi_k(x)\, m(dx), \quad m(dx) = dx + \big(\delta_0(dx) + \delta_1(dx)\big).
\end{align*}
\end{remark}

\subsection{Main results on the NLBVP: the process $\bar{X}$}
\label{Sec:MainXbar}

Here we obtain the probabilistic characterization of the NLBVP previously introduced.

\begin{theorem}
\label{thm:MAINI}
For the solution $u \in C((0, \infty)\times \bar{\Omega}) \cap \bar{D}_L$ to the NLBVP \eqref{probXbar} with $f \in C(\bar{\Omega})$, the following representation holds true
\begin{align*}
u(t,x) = \mathbf{E}_x\left[ f(\bar{X}_t) \right], \quad t>0,\; x \in \bar{\Omega}
\end{align*}
where the process $\bar{X} = \{\bar{X}_t\}_{t\geq 0}$ can be obtained via time change of an elastic Brownian motion. In particular, 
\begin{align}
\label{MINrepXbar}
\mathbf{E}_x\left[ f(\bar{X}_t) \right] = & \mathbf{E}_x\left[f(X^+ \circ \bar{V}^{-1}_t) \exp \left( - (c/\sigma)\, \gamma^+ \circ \bar{V}^{-1}_t \right) \right], \quad t>0,\; x \in \bar{\Omega}
\end{align}
where $\bar{V}^{-1}_t$ is the inverse to $\bar{V}_t = t + H \circ (\eta/\sigma) \gamma^+_t$.
\end{theorem}

\begin{proof}
Our time-change via $\bar{V}_t$ is obtained by considering the elastic Brownian motion and the pair $(X^+, \gamma^+)$.

Assume that \eqref{MINrepXbar} holds true. Then,
\begin{align*}
\mathbf{E}_x\left[ \int_0^\infty e^{-\lambda t} f(\bar{X}_t) dt \right]
= & \mathbf{E}_x\left[ \int_0^\infty e^{-\lambda t - (c/\sigma) \gamma^+ \circ \bar{V}^{-1}_t } f(X^+ \circ \bar{V}^{-1}_t) dt \right]\\
= & \mathbf{E}_x\left[ \int_0^\infty e^{- \lambda \bar{V}_t - (c/\sigma) \gamma^+_t } f(X^+_t) d\bar{V}_t \right]\\
= & \mathbf{E}_x\left[ \int_0^\infty e^{- \lambda t - \lambda H \circ (\eta/\sigma) \gamma^+_t - (c/\sigma) \gamma^+_t } f(X^+_t) d(t + H \circ (\eta/\sigma) \gamma^+_t) \right]\\
= & I_1(x) + I_2(x)
\end{align*}
where
\begin{align*}
I_1(x) 
= & \mathbf{E}_x\left[ \int_0^\infty e^{- \lambda t - \lambda H \circ (\eta/\sigma) \gamma^+_t - (c/\sigma) \gamma^+_t } f(X^+_t) dt \right] \\
= & \mathbf{E}_x\left[ \int_0^\infty e^{- \lambda t - \lambda^\alpha (\eta/\sigma) \gamma^+_t - (c/\sigma) \gamma^+_t } f(X^+_t) dt \right]
\end{align*}
and
\begin{align*}
I_2(x) = \mathbf{E}_x\left[ \int_0^\infty e^{- \lambda t - \lambda H \circ (\eta/\sigma) \gamma^+_t - (c/\sigma) \gamma^+_t } f(X^+_t) d(H \circ (\eta/\sigma) \gamma^+_t) \right]
\end{align*}
can be obtained as
\begin{align*}
I_2 (x) = - \frac{1}{\lambda} \mathbf{E}_x\left[ \int_0^\infty N^{el}_t\, d M^\alpha_t \right] 
\end{align*}
where
\begin{align*}
N^{el}_t = f(X^+_t) \exp \left(-\lambda t - (c/\sigma) \gamma^+_t\right)
\end{align*}
and  
\begin{align*}
M^\alpha_t = \exp(-\lambda H \circ (\eta/\sigma) \gamma^+_t)
\end{align*}
is a multiplicative functional such that, from the independence of $H$,
\begin{align*}
\mathbf{E}_x[M^\alpha_t | X^+_t] = \exp(-\lambda^\alpha (\eta/\sigma) \gamma^+_t).
\end{align*}
In particular,
\begin{align*}
I_2(x)
= & - \frac{1}{\lambda} \mathbf{E}_x\left[ \int_0^\infty e^{- \lambda t - (c/\sigma) \gamma^+_t } f(X^+_t) \, d( e^{-\lambda H \circ (\eta/\sigma) \gamma^+_t }) \right]\\
= & - \frac{1}{\lambda} \mathbf{E}_x\left[ \int_0^\infty e^{- \lambda t - (c/\sigma) \gamma^+_t } f(X^+_t) \, d( e^{-\lambda^\alpha (\eta/\sigma) \gamma^+_t }) \right]\\
= & (\eta/\sigma) \frac{\lambda^\alpha}{\lambda} \mathbf{E}_x\left[ \int_0^\infty e^{- \lambda t -\lambda^\alpha (\eta/\sigma) \gamma^+_t - (c/\sigma) \gamma^+_t } f(X^+_t) \, d\gamma^+_t \right]
\end{align*}
Now, let us consider the time change
\begin{align}
T_{\lambda, t} = t + \frac{\lambda^\alpha}{\lambda }(\eta /\sigma) \gamma^+_t
\label{TimeChangeLambda}
\end{align}
as a new clock for the elastic Brownian motion written in terms of $(X^+, \gamma^+)$ and write
\begin{align}
\label{RbarIunodue}
\bar{R}_\lambda f(x) = & I_1(x) + I_2(x), \quad \lambda>0
\end{align}
as follows
\begin{align}
\bar{R}_\lambda f(x) = & \mathbf{E}_x \left[ \int_0^\infty f(X^+_t) \exp \left(- \lambda T_{\lambda, t}  - (c / \sigma)\, \gamma^+_t \right) dT_{\lambda, t} \right] \label{potTimeChangeTlambda}\\
= & \mathbf{E}_x \left[ \int_0^\infty e^{-\lambda t} f(X^+ \circ T^{-1}_{\lambda, t}) \exp \left(-(c/\sigma)\, \gamma^+ \circ T^{-1}_{\lambda, t} \right) dt \right], \quad \lambda>0. \notag
\end{align}
We see that 
\begin{align}
T_{\lambda, t} = \int_{\bar{\Omega}} \gamma^+_{t,z} m_\lambda(dz) \quad \textrm{with} \quad m_\lambda(dz) = dz + \frac{\lambda^\alpha}{\lambda} (\eta/\sigma) m_\partial(dz)
\label{Tlambda}
\end{align}
where the local time $\gamma^+_{t,z}$ is the jointly continuous local time of $X^+$ at $z \in \bar{\Omega}$, $t\geq 0$.\\

(Apart from the time change $T_{\lambda, t}$) Let us consider
\begin{align}
\bar{R}_\lambda^{el} f(x) = \mathbf{E}_x \left[ \int_0^\infty e^{-\lambda t} f(X^{el}_{\beta,t})dt \right], \quad \lambda >0
\label{IDgenELASTIC}
\end{align}
which is the resolvent operator for the elastic Brownian motion $X^{el}_{\beta} = \{X^{el}_{\beta,t}\}_{t \geq 0}$ with generator
\begin{align*}
(G^{el}_\beta, D(G^{el}_\beta)), \quad \beta >0 
\end{align*}
where $G^{el}_\beta = \Delta$ and
\begin{align*}
D(G^{el}_\beta) = \{\varphi, \Delta \varphi \in C(\bar{\Omega}), \, \varphi \in H^1(\Omega)\,:\, 0 = \partial_{\bf n} \varphi  + \beta\, \varphi |_{\partial \Omega}\}.
\end{align*}
We can write
\begin{align*}
I_1 = \bar{R}^{el}_\lambda f \quad \textrm{and} \quad I_2 = \bar{R}^{\partial}_\lambda f \quad \textrm{so that} \quad  \bar{R}_\lambda f = \bar{R}^{el}_\lambda f +  \bar{R}^{\partial}_\lambda f
\end{align*}
for $\beta = \lambda^\alpha (\eta/\sigma) + (c/\sigma)$. For all positive $\beta$, we have
\begin{equation*}
\left\lbrace
\begin{array}{ll}
\displaystyle \Delta \bar{R}^{el}_\lambda f = \lambda \bar{R}^{el}_\lambda f - f & \textrm{on}\; \Omega,\\ 
\\
\displaystyle \partial_{\bf n} \bar{R}^{el}_\lambda f + \beta \bar{R}^{el}_\lambda f = 0 & \textrm{on}\; \partial \Omega.
\end{array}
\right .
\end{equation*}
Moreover, from Lemma \ref{lemmaPap} with $c_1=\lambda$ and $c_2 = \lambda^\alpha (\eta/\sigma) + (c/\sigma)$, we get
\begin{equation*}
\left\lbrace
\begin{array}{ll}
\displaystyle \Delta \bar{R}^{\partial}_\lambda f = \lambda \bar{R}^{\partial}_\lambda f & \textrm{on}\; \Omega,\\ 
\\
\displaystyle \partial_{\bf n} \bar{R}^{\partial}_\lambda f + c_2\, \bar{R}^{\partial}_\lambda f = (\eta/\sigma) \frac{\lambda^\alpha}{\lambda} f & \textrm{on}\; \partial \Omega.
\end{array}
\right .
\end{equation*}
We can immediately see that, for all $\beta$,
\begin{align*}
\Delta (\bar{R}^{el}_\lambda f + \bar{R}^{\partial}_\lambda f) = \lambda (\bar{R}^{el}_\lambda f + \bar{R}^{\partial}_\lambda f) - f \quad \textrm{on} \quad \Omega
\end{align*}
and, for $\beta=c_2$,
\begin{align*}
\partial_{\bf n} (\bar{R}^{el}_\lambda f + \bar{R}^{\partial}_\lambda f) + (\lambda^\alpha (\eta/\sigma) + (c/\sigma)) (\bar{R}^{el}_\lambda f + \bar{R}^{\partial}_\lambda f) = (\eta/\sigma) \frac{\lambda^\alpha}{\lambda} f \quad \textrm{on} \quad \partial \Omega.
\end{align*}
We get
\begin{equation*}
\left\lbrace
\begin{array}{ll}
\displaystyle \lambda \bar{R}_\lambda f - f = \Delta \bar{R}_\lambda f & \textrm{on}\; \Omega,\\ 
\\
\displaystyle  (\eta/\sigma) \frac{\lambda^\alpha}{\lambda} \big( \lambda \bar{R}_\lambda f - f\big) = - \partial_{\bf n} \bar{R}_\lambda f - (c/\sigma) \bar{R}_\lambda f & \textrm{on}\; \partial \Omega.
\end{array}
\right .
\end{equation*}
Notice that $\bar{R}_\lambda^{el}f$ and $\bar{R}_\lambda^{\partial} f$ are unique continuous solutions for the associated problems. Thus, $\bar{R}_\lambda f$ belongs to the domain of the generator of an elastic sticky Brownian motion obtained by considering $T_{\lambda, t}$ with $t>0$, $\lambda>0$. \\

(Back to the time-change) We can write, for $\lambda >0$,
\begin{align*}
\bar{R}_\lambda f(x) = \mathbf{E}_x \left[ \int_0^\infty e^{-\lambda t} f(X^{st}_{\beta, t}) dt \right], \quad  \beta= (\eta/\sigma) \frac{\lambda^\alpha}{\lambda}
\end{align*}
where $X^{st}_\beta = \{X^{st}_{\beta, t}\}_{t \geq 0}$, for all $\beta>0$, has generator $G_\beta =\Delta$ with
\begin{align*}
D(G_\beta) = \left\{ \varphi, \Delta \varphi \in C(\bar{\Omega}), \varphi \in H^1(\Omega)\,:\, \beta (\Delta \varphi) |_{\partial \Omega} = - \partial_{\bf n} \varphi - (c/\sigma) \varphi |_{\partial \Omega} \right\}.
\end{align*} 
Let us write, for $\lambda >0$,
\begin{align*}
\bar{R}^{st}_\lambda f(x) = \mathbf{E}_x \left[ \int_0^\infty e^{-\lambda t} f(X^{st}_{\beta, t}) dt \right], \quad \beta >0
\end{align*}
where $X^{st}_{\beta, t}$ can be realized from $X^+$ via time change by considering
\begin{align*}
T^{st}_{\beta, t} = t + \beta \gamma^+_t 
\end{align*} 
and $\bar{R}^{st}_\lambda f = \bar{R}_\lambda f$ only if $\beta= (\eta/\sigma) (\lambda^\alpha /\lambda)$, that is $T^{st}_{\beta, t} = T_{\lambda, t}$. For all $\beta > 0$, the sticky boundary condition introduced by $T^{st}_{\beta, t}$ can be therefore written as (recall \eqref{SameArgs})
\begin{align}
\label{Condsurface-RH}
\beta \Delta \bar{R}^{st}_\lambda f = - \partial_{\bf n} \bar{R}^{st}_\lambda f - (c/\sigma)\, \bar{R}^{st}_\lambda f, \quad \textrm{on } \partial \Omega
\end{align}
or
\begin{align}
\label{concludeTu}
\beta  \big( \lambda \bar{R}^{st}_\lambda f - f \big) = - \partial_{\bf n} \bar{R}^{st}_\lambda f - (c/\sigma)\, \bar{R}^{st}_\lambda f , \quad \textrm{on } \partial \Omega
\end{align}
with $\bar{R}^{st}_\lambda f \in D(G_\beta)$. From \eqref{CD-LT} we conclude that \eqref{concludeTu} with $\beta = (\eta/\sigma) (\lambda^\alpha / \lambda)$, leads to
\begin{align}
\label{uTMPproofBC}
(\eta/\sigma) D^\alpha_t Tu(t,x) = - \partial_{\bf n} u(t,x) - (c/\sigma) u(t,x), \quad t>0,\; x \in \partial \Omega
\end{align}
as claimed. The condition \eqref{uTMPproofBC} gives a unique solution for the heat equation. In particular, the pair $(u, Tu)$ is the unique solution of \eqref{probXbar} by uniqueness of the Laplace inverse under continuity for $u$ on $\bar{\Omega}$. Indeed, \eqref{potTimeChangeTlambda}, with given positive constants $\eta,\sigma, c, \lambda$, gives the resolvent for a Feller-Wentzell semigroup.
\end{proof}

\begin{remark}
Focus on \eqref{TimeChangeLambda}. We notice that, for $\alpha=1$, for all $\lambda >0$, 
$$T_{\lambda,t}= t + (\eta/\sigma) \gamma^+_t =:V_t.$$
\end{remark}

\begin{remark}
Focus on \eqref{Tlambda}. We observe that
\begin{align*}
m_\lambda(dz) = \int_0^\infty e^{-\lambda t} \bigg( \delta_0(t) dz + \bar{\Pi}(t) (\eta/\sigma) m_\partial(dz) \bigg) dt, \quad \lambda>0
\end{align*}
where $\bar{\Pi}(t):=\Pi(t, \infty) = \alpha \int_t^\infty y^{-1-\alpha}/\Gamma(\-\alpha) dy$, $t>0$ is the tail of the L\'{e}vy measure of $H$. The subordinator $H$ has no role as $X^+ \notin \partial \Omega$. This suggests that the positive continuous functional
\begin{align*}
\bar{A}_t = (\eta/\sigma) \int_0^t \bar{\Pi}(t-s) d\gamma^+_s = \frac{(\eta/\sigma)}{\Gamma(1-\alpha)} \int_0^t (t-s)^{-\alpha} d\gamma^+_s 
\end{align*}
can be considered in order to represent $\bar{X}$ via time change (notice that we have $D^\alpha_t \gamma^+_t$). Due to the profound implication of this fact, a detailed analysis will be presented in a separate work
\end{remark}

\begin{remark}
Focus on the path integral
\begin{align*}
\mathbf{E}_x \left[ \int_0^\infty e^{-\lambda t} f(\bar{X}_t) d(\bar{V}_t - t) \right] 
= & (\eta/\sigma) \frac{\lambda^\alpha}{\lambda} \mathbf{E}_x \left[ \int_0^\infty e^{-\lambda t - \lambda^\alpha (\eta/\sigma) \gamma^+_t -(c/\sigma) \gamma^+_t} f(X^+_t) d\gamma^+_t \right]
\end{align*}
which has been previously denoted by $I_2$ with $\bar{V}_t - t = H \circ (\eta/\sigma) \gamma^+_t$. We discuss below in Theorem  \ref{thm:holdingHatBar} the process $V_t - t$ in terms of holding times. We underline that, with
\begin{align*}
c_1 = \lambda, \quad c_2 = \lambda^\alpha (\eta/\sigma) + (c/\sigma),
\end{align*}
we can apply Lemma \ref{lemmaPap} where $(\eta /\sigma) (\lambda^\alpha/\lambda) f$ must be considered in place of $f$. 
\end{remark}

\begin{theorem}
For all  $t>0$, there exists a continuous kernel $\bar{p}(t, \cdot, \cdot): \bar{\Omega} \times \bar{\Omega} \to \mathbb{R}$ such that
\begin{align*}
\mathbf{E}_x[f(\bar{X}_t)] = \int_{\bar{\Omega}} f(y)\, \bar{p}(t,x,y) m(dy), \quad f \in L^2(\bar{\Omega}, m).
\end{align*} 
\label{thm:pBarContinuity}
\end{theorem}
\begin{proof}
For $\eta,\sigma, c, \beta>0$, let us consider $G_\beta = \Delta$ with
\begin{align*}
D(G_\beta) = \left\{ \varphi, \Delta \varphi \in C(\bar{\Omega}), \varphi \in H^1(\Omega)\,:\, \beta (\Delta \varphi) |_{\partial \Omega} = - \sigma \partial_{\bf n} \varphi - c \varphi |_{\partial \Omega} \right\}
\end{align*}
Then $(G_\beta, D(G_\beta))$ generates an elastic sticky Brownian motion, say $X_\beta = \{X_{\beta, t}\}_{t\geq 0}$ which is a strong Markov process on $\Omega$ with $C_0$-semigroup on $C(\bar{\Omega})$, say $S_{\beta, t}$, with probabilistic representation
\begin{align*}
S_{\beta, t} f(x) = \mathbf{E}_x [ f(X_{\beta, t})].
\end{align*}
For $\beta=\eta \frac{\lambda^\alpha}{\lambda}$ and $\lambda>0$ we write
\begin{align}
\bar{R}_\lambda f(x) := \mathbf{E}_x \left[ \int_0^\infty e^{-\lambda t} f(\bar{X}_t) dt \right] = \mathbf{E}_x \left[ \int_0^\infty e^{-\lambda t} f(X_{\beta, t}) dt \right] .
\label{SlambdaPot1}
\end{align}
Recall {\bf S1} and {\bf S2} in Section \ref{sec:secX}. For the $C_0$-semigroup $S_{\beta, t}$ on $C(\bar{\Omega})$ we know that $S_{\beta, t} L^2(\Omega) \subset C(\bar{\Omega})$ and $S_{\beta, t} : L^2(\Omega) \to C(\bar{\Omega})$ is compact. In particular, there exists a continuous kernel $\varphi_\lambda$ on $\bar{\Omega}$ such that
\begin{align}
\int_0^\infty e^{-\lambda t} S_{\beta, t} f(x)\, dt = \int_{\bar{\Omega}} f(y) \varphi_\lambda(x,y) m(dy) \quad \forall\, \beta>0.
\label{SlambdaPot2}
\end{align} 
From \eqref{SlambdaPot1} and \eqref{SlambdaPot2}, for $\beta=\eta \frac{\lambda^\alpha}{\lambda}$ and $\lambda>0$, we get that $\varphi_\lambda,\, \bar{R}_\lambda f \in D(G_\beta)$ and
\begin{align*}
\bar{R}_\lambda f(x) = \int_{\bar{\Omega}} f(y) \varphi_\lambda(x,y) m(dy)
\end{align*} 
is continuous and bounded. That is, $\varphi_\lambda$ is a continuous kernel on $\bar{\Omega}$ for $\bar{R}_\lambda$. The claim follows.

\end{proof}

\section{The process $\bar{X}$}
\label{sec:barX}
We say that $\bar{X}$ is an elastic process meaning that 
$$\mathbf{E}_x[f(\bar{X}_t)] = \mathbf{E}_x[f(\bar{X}_t), t < \bar{\zeta}]$$ 
is written in terms of the lifetime $\bar{\zeta}$ and the multiplicative functional $\bar{M}_t$ for which
\begin{align}
\label{lawExpLTbar} 
\mathbf{P}(\bar{\zeta} > t | \bar{X}_t) = e^{-(c/\eta) \bar{\gamma}_t} =: \bar{M}_t \quad \textrm{with} \quad \bar{M}_t = M^+ \circ \bar{V}^{-1}_t.
\end{align}
We say that $\bar{X}$ is a sticky process meaning that $\{t\,:\, \bar{X}_t \in \partial \Omega\}$ is a set of positive Lebesgue measure related to the sticky holding times $\{\bar{e}_i\}_i$ discussed in Section \ref{sec:comparisonXhatX} and the holding times discussed in Remark \ref{rmk:HTbarX}. Indeed, due to the jumps of $\bar{V}$,  
\begin{align*}
\mathbf{P}_x(\bar{X}_s = x, s \in [0,\mathfrak{T}))>0 \quad \textrm{if} \quad x \in \partial \Omega
\end{align*}
for some random time $\mathfrak{T}$. Thus, $\mathfrak{T}$ is a holding time.
 
From the previous section we know that
\begin{align*}
\mathbf{E}_x[f(\bar{X}_t)] = \int_{\bar{\Omega}} f(y) \bar{p}(t,x,y) m(dy), \quad f \in C(\bar{\Omega})
\end{align*}
writes
\begin{align*}
\mathbf{E}_x[f(\bar{X}_t)]= \mathbf{E}_x[f(\bar{X}_t), t < \bar{\zeta}] 
\end{align*}
where $\bar{\zeta}$ is written in terms of the local time $\bar{\gamma}$ according to \eqref{lawExpLTbar}. Moreover, a functional of $\bar{X}$ can be given as time changed functional of $X^+$, that is
\begin{align}
\label{equivgammaBAR} 
\tilde{\gamma}_t = (\eta/\sigma) \gamma^+ \circ \bar{V}^{-1}_t.
\end{align}
The process \eqref{equivgammaBAR} slightly differs from
\begin{align*}
\bar{\gamma}_t := \int_0^t \mathbf{1}_{\partial \Omega} (\bar{X}_s) ds 
\end{align*} 
which is trivially non-zero for $t>0$ since $\bar{X}$ is a sticky process.\\

\begin{remark}
($1$-dimensional case) Let us consider the cases:
\begin{itemize}
\item ($\alpha=1$ and $d=1$) The representation \eqref{repXelSticky} for the sticky Brownian motion on the half line agrees with the results obtained in \cite[Section 10]{ItoMcK}. Bass (\cite{BasSticky}) proved weak uniqueness (uniqueness in law) for the stochastic equation
\begin{align}
X_t = x + \int_0^t \mathbf{1}_{(X_s \neq 0)} dW_s, \quad t\geq 0
\label{SDEstickyBass}
\end{align}
among continuous local martingales with speed measure
\begin{align}
m(dx) = dx + (\eta/\sigma) \delta_0(dx).
\label{mMeasure1dim}
\end{align}
From \eqref{mMeasure1dim} and the jointly continuous local time $\gamma^+_{t,z}$ (at level $\{z\}$ up to time $t$) of $X^+$ on $\mathbb{R}$ for $z \in \mathbb{R}$ and $t>0$, we have the well-known fact
\begin{align}
\int_{\mathbb{R}} \gamma^+_{t,z} m(dz) = t + (\eta/\sigma) \gamma^+_t=: V_t
\label{speedChangeV}
\end{align}
from which we get $X$ on $\mathbb{R}$ with sticky point $\{0\}$ via speed measure change for $X^+$. 
\item ($\alpha \neq 1$ and $d=1$) The NLBVP in dimension $d=1$ has been investigated in \cite{NLBVP-ItheHline}. In case of non-local dynamic boundary condition the process behaves like a Brownian motion far from the boundary according to \eqref{SDEstickyBass}. We still have a continuous solution associated with \eqref{mMeasure1dim}. However, $\bar{X}$ is not a martingale. Thus, according to \eqref{SDEstickyBass} and \eqref{mMeasure1dim} the result by Bass says that $X$ has exponential holding times ensuring the Markov nature of $X$. For general holding times we should therefore impose further conditions.
\end{itemize}
\end{remark}

Let us consider the $\lambda$-potential
\begin{align*}
\bar{R}_\lambda f(x) := \mathbf{E}_x \left[ \int_0^\infty e^{-\lambda t}f(\bar{X}_t) dt \right], \quad \lambda>0, \quad x \in \bar{\Omega}
\end{align*}
and the first hitting time $\bar{\tau}_{\partial \Omega} := \inf\{t \geq 0\,:\, \bar{X}_t \in \partial \Omega\}$. We underline that
\begin{align*}
\bar{R}_\lambda f(x) 
= & \mathbf{E}_x \left[ \int_0^\infty e^{-\lambda t}f(X^\dagger_t) dt \right] + \mathbf{E}_x \left[ e^{-\lambda \bar{\tau}_{\partial \Omega}} \mathbf{E}_{\bar{X}_{\bar{\tau}_{\partial \Omega}}} \left[ \int_0^\infty e^{-\lambda t} f(\bar{X}_t) dt \right] \right],
\end{align*}
that is $\bar{R}_\lambda f(x) = R^\dagger_\lambda f(x) + \bar{R}^\partial_\lambda f(x)$ with $T \bar{R}_\lambda f = T \bar{R}_\lambda^\partial f$ and 
\begin{align*}
\bar{R}_\lambda^\partial f(x)  
= & \int_0^\infty  e^{-\lambda t}  \int_{\partial \Omega} \bar{R}_\lambda f(y) \, \mathbf{P}_x(\bar{X}_{\bar{\tau}_{\partial \Omega}} \in dy, \bar{\tau}_{\partial \Omega} \in dt)
\end{align*}
for which we have
\begin{align*}
\Delta \bar{R}_\lambda f = \Delta \big( R^\dagger_\lambda f + \bar{R}^\partial_\lambda f \big), \quad \textrm{and} \quad \lambda \bar{R}_\lambda f - f = \lambda \big( R^\dagger_\lambda f + \bar{R}^\partial_\lambda f \big) - f, \quad \textrm{in } \bar{\Omega}.
\end{align*}
In particular,
\begin{align}
\lambda \bar{R}_\lambda f - f = \big( \lambda  R^\dagger_\lambda f -f \big) + \lambda \bar{R}^\partial_\lambda f \quad \textrm{on } \Omega
\end{align}
and
\begin{align}
\lambda \bar{R}_\lambda f - f = \big( \lambda  R^\partial_\lambda f -f \big) \quad \textrm{on } \partial \Omega
\end{align}
provide an interesting reading of the dynamic (non-local) boundary condition.\\

\begin{remark}
(Occupation measures)   For $\Lambda \subseteq \bar{\Omega}$, consider the occupation measure
\begin{align*}
\mu_{\bar{X}} (x, \Lambda^\prime) = \mathbf{E}_x \left[ \int_0^{\bar{\tau}(\Lambda)} \mathbf{1}_{\Lambda^\prime}(\bar{X}_s) ds \right], \quad x \in \Lambda, \quad \Lambda^\prime \subset \Lambda
\end{align*}
where $\bar{\tau}(\Lambda) = \inf\{t\,:\, \bar{X}_t \notin \Lambda\}$. Here we only underline the interesting role of the "trapping" boundary for $X^+$ under the time change $\bar{V}$. Thus, we only focus on 
\begin{align*}
\mu_{\bar{X}}(x, \Lambda) = \mathbf{E}_x [\bar{\tau}(\Lambda)] \leq \bar{R}_0 \mathbf{1}_\Lambda (x) := \mathbf{E}_x \left[ \int_0^\infty \mathbf{1}_\Lambda(\bar{X}_s) ds \right], \quad x \in \Lambda
\end{align*}
and the fact that
\begin{align*}
\bar{R}_0 \mathbf{1}_\Lambda (x) = \lim_{\lambda \to 0} \bar{R}_\lambda \mathbf{1}_{\Lambda}(x), \quad x \in \Lambda.
\end{align*} 
First we observe that
\begin{align*}
\bar{R}_\lambda \mathbf{1}_{\Lambda}(x) = \bar{R}_\lambda \mathbf{1}_{\Lambda \cap \Omega}(x) + \bar{R}_\lambda \mathbf{1}_{\Lambda \cap \partial \Omega}(x)
\end{align*}
where
\begin{align*}
\bar{R}_\lambda \mathbf{1}_{\Lambda \cap \Omega}(x) = \mathbf{E}_x\left[ \int_0^\infty e^{- \lambda t - \lambda^\alpha (\eta/\sigma) \gamma^+_t - (c/\sigma) \gamma^+_t } \mathbf{1}_{\Lambda \cap \Omega}(X^+_t) dt \right]
\end{align*}
and
\begin{align*}
\bar{R}_\lambda \mathbf{1}_{\Lambda \cap \partial \Omega}(x) = (\eta/\sigma) \frac{\lambda^\alpha}{\lambda} \mathbf{E}_x\left[ \int_0^\infty e^{- \lambda t -\lambda^\alpha (\eta/\sigma) \gamma^+_t - (c/\sigma) \gamma^+_t } \mathbf{1}_{\Lambda \cap \partial \Omega}(X^+_t) \, d\gamma^+_t \right]
\end{align*}
can be respectively obtained from $I_1$ and $I_2$ in \eqref{RbarIunodue}. We have that
\begin{align*}
\bar{R}_\lambda \mathbf{1}_{\Lambda \cap \Omega}(x) = \bar{g}_{1,\lambda}(x), \quad \bar{R}_\lambda \mathbf{1}_{\Lambda \cap \partial \Omega}(x) = (\eta/\sigma) \frac{\lambda^\alpha}{\lambda} \bar{g}_{2,\lambda}(x)
\end{align*}
where $\bar{g}_{1,\lambda}, \bar{g}_{2,\lambda}$ for $\lambda>0$ are continuous and bounded. Indeed, 
\begin{align*}
\tilde{g}_{1,\lambda}(x) \leq \mathbf{E}_x\left[ \int_0^\infty e^{-\lambda t - \lambda^\alpha (\eta/\sigma)\gamma^+_t} \mathbf{1}_{\bar{\Omega}}(X^{el}_t) dt \right] = \mathbf{E}_x\left[ \int_0^{\zeta^{el}} e^{-\lambda t - \lambda^\alpha (\eta/\sigma)\gamma^+_t} dt \right] < \infty
\end{align*}
where $\zeta^{el}$ is the lifetime of the elastic Brownian motion $X^{el}$. It is well known that $\zeta^{el} = \inf\{t\,:\, \gamma^+_t \geq \chi\}$ where $\chi$ is an exponential random variable. In our case, 
\begin{align*}
\mathbf{P}(\chi >t) = e^{-(c/\sigma) t}
\end{align*}
according to \eqref{lawExpLT} and \eqref{equivgammaBAR}. Moreover,
\begin{align*}
\tilde{g}_{2,\lambda}(x) \leq \mathbf{E}_x\left[ \int_0^\infty e^{-\lambda t - \lambda^\alpha (\eta/\sigma)\gamma^+_t - (c/\sigma) \gamma^+_t} \mathbf{1}_{\partial \Omega}(X^+_t) d\gamma^+_t \right]
\end{align*}
which is $g$ in point $iii)$ of Lemma \ref{lemmaPap} with $c_1=\lambda>0$, $c_2 =\lambda^\alpha (\eta/\sigma) + (c/\sigma)>0$ and $f=\mathbf{1}$. Since $\bar{\Omega}$ is compact and connected, then there exists a complete orthonormal basis consisting of Robin eigenfunctions of $\Delta$ with an increasing sequence of strictly positive eigenvalues (\cite[page 8]{Chav84}). This implies that $\mathfrak{g}$ (in Lemma \ref{lemmaPap}) is finite as stated in \cite[Theorem 3.7]{Pap90}. Thus, we get $\tilde{g}_{2,\lambda} < \infty$ for $\lambda>0$. \\

As $\lambda \to 0$, for $x \in \Lambda$ we get that
\begin{align*}
\tilde{g}_{1, 0}(x) \leq \mathbf{E}_x \left[ \int_0^{\zeta^{el}} dt \right]
\end{align*}
and
\begin{align*}
\tilde{g}_{2, 0}(x) \leq \mathbf{E}_x\left[ \int_0^\infty e^{-(c/\sigma)\gamma^+_t} \mathbf{1}_{\partial \Omega} (X^+_t) d\gamma^+_t \right] = \mathbf{E}_x\left[ \int_0^\infty e^{-(c/\sigma)\gamma^+_t} d\gamma^+_t \right]  = \mathbf{E}_x\left[ \int_0^{\zeta^{el}} d\gamma^+_t \right].
\end{align*}
The last equality is obtained by taking into account the multiplicative functional 
\begin{align*}
M^+_t= e^{-(c/\sigma)\gamma^+_t} = \mathbf{E}[\, \mathbf{1}_{(\,t\, <\, \zeta^{el}\,)} \,| \, X^+_t].
\end{align*}
By Stieltjes integration we have
\begin{align*}
M^+_\tau = M^+_0 - (c/\sigma) \int_0^\tau M^+_t d\gamma^+_t, \quad \tau >0
\end{align*} 
and
\begin{align*}
\mathbf{E}_x\left[ \int_0^\infty e^{-(c/\sigma)\gamma^+_t} d\gamma^+_t \right] = (\sigma / c).
\end{align*} 
By definition of $\zeta^{el}$, we know that $\gamma^+ \circ \zeta^{el} = \chi$ and 
\begin{align*}
\mathbf{E}_x\left[ \int_0^{\zeta^{el}} d\gamma^+_t \right] = \mathbf{E}[\chi] = (\sigma/c).
\end{align*}
Since
\begin{align}
\mathbf{E}_x[\zeta^{el}] < \infty, \quad x \in \Lambda  \quad \textrm{and} \quad \mathbf{E}_x[\gamma^+ \circ \zeta^{el}] < \infty \quad x \in \Lambda,
\label{OccMeasA}
\end{align}
then 
\begin{align*}
\tilde{g}_{1, 0}(x) \leq \mathbf{E}_x[\zeta^{el}] \quad \textrm{implies} \quad \bar{R}_0 \mathbf{1}_{\Lambda \cap \Omega}(x) < \infty
\end{align*}
and
\begin{align*}
\tilde{g}_{2, 0}(x) \leq \mathbf{E}_x[\gamma^+ \circ \zeta^{el}] \quad \textrm{implies} \quad \bar{R}_0 \mathbf{1}_{\Lambda \cap \partial \Omega}(x) < \infty
\end{align*}
Under \eqref{OccMeasA}, 
\begin{align*}
\{X^{el} \circ V^{-1}_t,\, V^{-1}_t < \zeta^{el}\} \quad \textrm{which equals in law} \quad \{X_t,\, t < \zeta\}
\end{align*}
started at $X^{el}_0 \in \Lambda \subseteq \bar{\Omega}$ spends on $\bar{\Omega}$ a finite (mean) time with
\begin{align*}
\mathbf{E}_x[\zeta^{el}] + (\eta/\sigma) \mathbf{E}_x[\gamma^+ \circ \zeta^{el}] < \infty \quad x \in \Lambda.
\end{align*}
This corresponds to the case $\alpha=1$ and coincides with $\mathbf{E}_x[V \circ \zeta^{el}]$. That is,
\begin{align*}
\mathbf{E}_x \left[\int_0^\infty \mathbf{1}_\Lambda (X_t) dt \right] \leq \mathbf{E}_x \left[\int_0^\infty \mathbf{1}_{\bar{\Omega}} (X_t) dt \right] = \int_0^\infty \mathbf{P}_x(t< V \circ \zeta^{el}) dt = \mathbf{E}_x[V \circ \zeta^{el}].
\end{align*}
We have that, under \eqref{ASShtINTRO}, 
\begin{align*}
\mathbf{E}_x[\bar{\tau}(\Lambda \cap \Omega)] \leq \bar{R}_0 \mathbf{1}_{\Lambda \cap \Omega}(x) < \infty \quad \textrm{for all} \quad \alpha \in (0,1]
\end{align*}
whereas, due to the ratio $\lambda^\alpha / \lambda$, 
\begin{align}
\mathbf{E}_x[\bar{\tau}(\Lambda \cap \partial \Omega)] \leq \bar{R}_0 \mathbf{1}_{\Lambda \cap \partial \Omega}(x)  \left\lbrace
\begin{array}{ll}
<\infty, & \alpha =1,\\
= \infty, & \alpha \in (0,1).
\end{array} 
\right.
\label{occupMeasXbar}
\end{align}
Thus, for $\alpha \in (0,1)$, only in case $\Lambda \cap \partial \Omega = \emptyset$ do we have $\mathbf{E}_x[\bar{\tau}(\Lambda)] < \infty$ for all $x \in \Lambda$. This suggests a trap effect on the boundary for the Brownian motion.  
\end{remark}

\begin{remark}
($1$-dimensional case) Here we consider the problem 
\begin{align}
\label{FBVP-realLine}
\left\lbrace
\begin{array}{ll}
\displaystyle \frac{\partial u}{\partial t}(t,x) = u^{\prime \prime}(t,x), \quad t>0, \, x >0,\\
\\
\displaystyle \eta D^\alpha_t u(t,0) = \sigma u^\prime(t,0) - c\, u(t, 0), \quad t>0,\\
\\
\displaystyle u(0,x) = f(x), \quad x \geq 0, \quad f \in C_b[0, \infty),
\end{array}
\right.
\end{align}
for which we maintain the notation $\bar{X}$ and $X$ for the involved processes, now on $[0, \infty)$.  In particular,
\begin{align*}
u(t,x) = \mathbf{E}_x[f(\bar{X}_t)] \quad \textrm{and} \quad \bar{X}=X \quad \textrm{if} \quad \alpha=1.
\end{align*}
For $\Lambda \subseteq [0, \infty)$, we write
\begin{align*}
\bar{\tau}(\Lambda) = \inf\{t\,:\, \bar{X}_t \notin \Lambda \}, \quad \tau(\Lambda) = \inf\{t\,:\, X_t \notin \Lambda \}.
\end{align*}
Let $\Lambda$ be an interval, $m(\Lambda)<\infty$ where $m$ is that in \eqref{mMeasure1dim}. We recall the following result (\cite[Lemma 3]{NLBVP-ItheHline}):
\begin{itemize}
\item [i)] If $\Lambda \subset (0, \infty)$, then 
\begin{align*}
\forall\, \alpha \in (0,1], \quad \forall\, x \in \Lambda, \quad \mathbf{E}_x[\bar{\tau}(\Lambda) ]< \infty.
\end{align*} 
\item [ii)] Otherwise, 
\begin{align*}
\forall\, \alpha \in (0,1), \quad \forall\, x \in \Lambda, \quad \mathbf{E}_x[\bar{\tau}(\Lambda) ] = \infty
\end{align*}
and
\begin{align*}
\quad \forall\, x \in \Lambda, \quad \mathbf{E}_x[\tau(\Lambda) ] < \infty.
\end{align*}
\end{itemize} 
Thus, $\Lambda$ plays the role of a "trap set" only if $\Lambda$ includes the sticky point $\{0\}$. On $(0, \infty)$, the process $\bar{X}$ behaves like a Brownian motion. For $d=1$, we have the explicit formula
\begin{align*}
\bar{R}_\lambda f(x) =  R^\dagger_\lambda f(x) + \frac{e^{-x \sqrt{\lambda}}}{c + \eta \lambda^\alpha + \sigma \sqrt{\lambda}} \left( \sigma \int_0^\infty e^{-y \sqrt{\lambda}} f(y) dy + \eta \frac{\lambda^\alpha}{\lambda}  f(0) \right), \quad \lambda>0
\end{align*}
where $R^\dagger_\lambda$ is the resolvent associated with the killed process $X^\dagger$ on $(0, \infty)$ as defined in Section \ref{sec:intro} in case of bounded domains. In particular, $\{X^\dagger_t,\, t < \tau_0\}$ where $\tau_0 = \inf\{t\,:\, X^\dagger_t=0\}$ is the lifetime of the process. The mean value $\mathbf{E}_x[\bar{\tau}(\Lambda)]$ can be explicitly evaluated as well as the limit $\bar{R}_0 \mathbf{1}_\Lambda (x)$. We only underline that, for $x \in \Lambda =[0, \epsilon]$ with $\epsilon>0$, that is $\Lambda$ includes $\{0\}$, we have 
\begin{align*}
\bar{R}_0 \mathbf{1}_{[0,\epsilon]}(x) = R^\dagger_0 \mathbf{1}_{[0, \epsilon]}(x)  + (\sigma/c) \epsilon + (\eta/c) \lim_{\lambda \to 0} \frac{\lambda^\alpha}{\lambda}
\end{align*}
with $R^\dagger_0 \mathbf{1}_{[0, \epsilon]}(x) < \infty$. For $x \in \Lambda =[a, \epsilon]$ with $\epsilon > a >0$, we have
\begin{align*}
\bar{R}_0 \mathbf{1}_{[a,\epsilon]}(x) = R^\dagger_0 \mathbf{1}_{[a, \epsilon]}(x)  + (\sigma/c) (\epsilon - a)
\end{align*}
with $R^\dagger_0 \mathbf{1}_{[a, \epsilon]}(x) < \infty$. As we can see, in case $\alpha=1$, $\bar{R}_0 \mathbf{1}_{[a,\epsilon]}(x) < \infty$ for $a\geq 0$. Let us consider now the reflected Brownian motion $X^+$ on $[0, \infty)$ and the killed reflected Brownian motion on $[0, \epsilon)$. We recall that, for $\tau^+_\epsilon = \inf\{t\,:\, X^+_t = \epsilon\}$ with $X^+_0=0$, the composition $\gamma^+ \circ \tau^+_\epsilon$ is an exponential r.v. with mean $\epsilon >0$. For the killed process $\{X^\dagger_t\}_{t\geq 0} := \{X^+ \circ V^{-1}_t,\, V^{-1}_t < \tau^+_\epsilon \}_{t \geq 0}$, the expected lifetime on $[0, \epsilon)$ is given by
\begin{align*}
\mathbf{E}_x[\tau_\epsilon] + (\eta/\sigma) \mathbf{E}_x[\gamma^+ \circ \tau^+_\epsilon].
\end{align*}
For details on this remark we refer to \cite{NLBVP-ItheHline}.
\end{remark}

\begin{theorem}
($1$-dimensional case) For $\alpha = 1$, we have that 
\begin{align*}
\mathbf{P}_0(\bar{X}_t = 0) = \mathbf{P}_0(\chi > \gamma^+_t) = \mathbf{E}_0 \left[ e^{-(\sigma / \eta) \gamma^+_t} \right], \quad t \geq 0
\end{align*}
where $\chi \indep X^+$ with $\mathbf{P}(\chi > t) = e^{-(\sigma/\eta) t}$ and $\gamma^+$ is the local time of the reflected Brownian motion $X^+$ on $[0, \infty)$. 
\label{thm:threshold} 
\end{theorem}
\begin{proof}
For $\alpha=1$ we write $X$ in place of $\bar{X}$. Observe that, from $\bar{R}_\lambda f$ above, 
\begin{align*}
\lim_{\epsilon \to 0} \int_0^\infty e^{-\lambda t} \mathbf{P}_0(X_t \in [0, \epsilon]) dt = \frac{\eta}{\eta \lambda + \sigma \sqrt{\lambda}} , \quad \lambda >0
\end{align*}
which coincides with 
\begin{align*}
\int_0^\infty e^{-\lambda t} \mathbf{P}_0(\chi > \gamma^+_t) dt, \quad \lambda >0
\end{align*}
where $\chi \indep \gamma^+$.
\end{proof}

The occupation time of the boundary point $\{0\}$ is regulated by an exponential threshold for the local time of the reflected Brownian motion. Recall that $\bar{X}$, as well as $X$ in the case $\alpha=1$, represents a motion along the path of $X^{+}$; given that path, the sticky time change forces the process to slow down at (or near) the boundary point $\{0\}$. Observe that $\mathbf{P}_0(\bar{X}_t=0)$ for $\alpha \in (0,1)$ does not distinguish between delay at $\{0\}$ (like holding times) and infinitely many crossings of $\{0\}$ up to time $t$ (like sticky holding times).

\section{Sticky Holding Times: Discussion on the processes $X$, $\hat{X}$, $\bar{X}$}
\label{sec:comparisonXhatX}
{\color{\magenta} Recall that, in the present work, a Brownian motion is termed sticky if it spends a non-zero Lebesgue measure of time on the boundary.} The sticky Brownian motion in a smooth domain can be interpreted through two equivalent paradigms: 
\begin{itemize}
\item[1)] From the trajectory point of view, the process constantly hits the boundary along the normal direction (if we assume no lateral diffusion) and these high-frequency interactions create a cumulative delay. For example, for  the 1-dimensional motion in Theorem \ref{thm:threshold} the sticky holding time (the occupation time of the sticky point) is referred to the time the process spends bouncing at the same point.
\item[2)] From a time-change perspective, the process enjoys a random delay of the physical time. The paths of the process $X$ are (geometrically) identical to the paths of a reflected Brownian motion $X^+$. The process leaves a point instantaneously and the sticky behaviour emerges from slowing down the temporal flow as the process accumulates local time on the boundary. 
\end{itemize}
Thus, the positive (Lebesgue) measure of the time on the boundary for sticky processes can be determined in terms of competitive mechanism between infinitely many crossings and instantaneous reflections. However, despite the intricate dynamic of $X$ on the boundary, the processes $\hat{X}$ and $\bar{X}$ include a further behaviour which appears clearly from the time change. Indeed, for both processes, the jumps of the (independent) subordinator $H$ may either delay the frenetic behaviour of $X^+$ or completely hold the process at a point of the boundary for a random time (as soon as a jump of $H$ occurs).

{\color{\magenta}
\begin{definition}
(Holding Time and Sticky Holding Time)
\begin{itemize}
\item[HT)] A holding time is the time a process spends at a point on the boundary.\\ The process permanently occupies the same point.
\item[SHT)] A sticky holding time is the time a process spends on $\partial \Omega \subset \Lambda_\epsilon$ during a visit to $\Lambda_\epsilon$ where $\Lambda_\epsilon := \{x \in \bar{\Omega}\,:\, \operatorname{dist}(x, \partial \Omega) < \epsilon\}$ with $\epsilon >0$.\\ The process bounces off the boundary infinitely many times.
\end{itemize} 
\label{defSHTandHT} 
\end{definition}

As well as holding times, sticky holding times can be associated with occupation times. However, we refer to SHT to mean the residence time on the boundary accumulated while bouncing off it; some of these bounces produce permanent occupation time, which constitutes an HT.\\

For any (time-homogeneous) Markov process, a holding time is exponentially distributed (\cite[Proposition 2.19]{BluGet68}) and the process can leave a holding point only by a jump (\cite[Proposition 3.13]{BluGet68}). Recall that $\hat{X}$ and $\bar{X}$ are non-Markovian with continuous paths. Thus the results in \cite{BluGet68} do not apply to our framework. Concerning the sticky holding time, the process bounces off the boundary infinitely many times, which generates a time spent on the boundary. In addition, it can hit some holding points, thereby accumulating holding times as well. The combination of both components constitutes the overall sticky holding time.

We introduce a wide class of sticky Brownian motions (not necessarily Markovian) with (not necessarily exponential) sticky holding times and (not necessarily exponential) holding times. In particular, based on these definitions, our analysis implies that:
\begin{itemize}
\item $X$, $\hat{X}$, and $\bar{X}$ are processes characterized by non-zero sticky holding times.
\item $\hat{X}$ and $\bar{X}$ exhibit non-zero holding times.
\end{itemize}
As will become clearer below, according to (1) and (2), the (standard) holding times may contribute to the determination of the sticky holding times. If we consider the local time on its own time scale, we observe a straight line where its fractal nature generates the subsequent segments. This is the case $\alpha=1$. For $\alpha \in (0,1)$ we change the scale and the straight line is replaced by a non-decreasing path (a staircase path). Each plateau is a holding time. The reader can consult \cite[formulas (4.11) and (4.12)]{OnTelegSign} for an example in $d=1$ of a persistent random walk characterizing sticky local times.
}

\subsection{On the process $X$}
\label{sec:OnX}
{\color{\magenta} The process $X$ can be obtained from $\bar{X}$ and serves as the base process for $\hat{X}$.} Recall the definition \eqref{CDderDef}. The special case $\alpha=1$ in our Theorem \ref{thm:MAINI} gives
\begin{align*}
u(t,x) = \mathbf{E}_x[f(X_t)]
\end{align*} 
where $X$ has generator $(G,D(G))$ and $u$ solves \eqref{probX}. Indeed, 
\begin{align*}
H_t = t \quad \textrm{implies} \quad \bar{V}_t = V_t:=t + (\eta/\sigma) \gamma^+_t
\end{align*} 
and we obtain the representation 
\begin{align}
\label{repXelSticky}
\mathbf{E}_x[f(X_t)] = \mathbf{E}_x[f(X^+ \circ V^{-1}_t)\, M^+ \circ V^{-1}_t], \quad t>0,\; x \in \bar{\Omega}
\end{align}
where $M^+ \circ V^{-1}_t$ is the multiplicative functional $M_t = \exp\left( - (c/\eta) \gamma_t\right)$  associated with the elastic condition introduced in \eqref{lawExpLT}. Indeed, the boundary local time $\gamma$ of $X$ can be written as the composition 
\begin{align}
\label{equivgamma}
\gamma_t = (\eta/\sigma)\gamma^+ \circ V^{-1}_t, \quad t>0.
\end{align}
Notice that the representation \eqref{repXelSticky} on a bounded domain of $\mathbb{R}^d$ is new as well as the representation \eqref{MINrepXbar} for $\alpha \in (0,1)$.

\begin{remark}
\label{remarkPAPalternative}
(Concerning the proof of Theorem \ref{thm:alternativPap}) It is worth highlighting that a nice proof can be also obtained from Papanicolau's work. We use the fact that (Lemma \ref{lemmaBluGet})
\begin{align*}
\int_{(0, \infty)} g(t) f(X^+ \circ V^{-1}_t) d\gamma^+ \circ V^{-1}_t = \int_{(0, \infty)} g(V_t) f(X^+_t) d\gamma^+_t
\end{align*}
for $g,f$ regular enough (bounded and measurable for instance). From \eqref{repXelSticky} and \eqref{equivgamma} we write
\begin{align*}
\mathbf{E}_x\left[ \int_0^\infty e^{-\lambda t} f(X_t) d\gamma_t \right]
= & (\eta /\sigma) \mathbf{E}_x\left[ \int_0^\infty e^{-\lambda t - \lambda (\eta /\sigma) \gamma^+_t - (c/\sigma) \gamma^+_t} f(X^+_t) d\gamma^+_t \right],
\end{align*}
and from Lemma \ref{lemmaequivPCAFhatNOhat} we get
\begin{align}
\label{TMPLemmavarpi}
\mathbf{E}_x\left[\int_0^\infty e^{-\lambda t} f(\hat{X}_t) d\hat{\gamma}_t \right] = (\eta /\sigma) \frac{\lambda^\alpha}{\lambda}  \mathbf{E}_x\left[ \int_0^\infty e^{-c_1 t - c_2 \gamma^+_t} f(X^+_t) d\gamma^+_t \right] = \tilde{\mathsf{w}}_2(\lambda, x)
\end{align}
where $c_1=\lambda^\alpha$ and $c_2=\lambda^\alpha (\eta /\sigma) + (c/\sigma)$. According to Lemma \ref{lemmaPap}, $\tilde{\mathsf{w}}_2$ solves \eqref{prob31} with boundary condition
\begin{align*}
\partial_{\bf n} \tilde{\mathsf{w}}_2(\lambda, x) + c_2 \tilde{\mathsf{w}}_2(\lambda, x) = (\eta/\sigma) \frac{\lambda^\alpha}{\lambda} f(x) \quad x\in \partial \Omega, \; \lambda>0, 
\end{align*}
that is
\begin{align}
\label{conLTvarpi}
\partial_{\bf n} \tilde{\mathsf{w}}_2(\lambda, x) + (c/\sigma) \tilde{\mathsf{w}}_2(\lambda, x) + (\eta/\sigma) \bigg( \lambda^\alpha \tilde{\mathsf{w}}_2(\lambda, x)  - \frac{\lambda^\alpha}{\lambda} f(x) \bigg) =0, \quad x \in \partial \Omega, \; \lambda>0. 
\end{align}
Thus, $\tilde{\mathsf{w}}_2(\lambda, x)$ solves \eqref{pr2wHatLaw}. 

Moreover, we underline the following fact. Focus on the proof of Theorem \ref{thm:MAINI}. The potential $\tilde{\mathsf{w}}_2(\lambda, x)$ can be treated as $\bar{R}^\partial_\lambda f(x)$ whereas, there is no connection between the potential $\tilde{\mathsf{w}}_1(\lambda, x)$ and the resolvent $\bar{R}^{el}_\lambda f(x)$.
\end{remark}

\vspace{.5cm}

The process $X$ moves along the path of $X^+$ as well as the processes $\hat{X}$ and $\bar{X}$. In the $1$-dimensional case, $\Omega=[0, a)$ for example, the zero set $\{0\leq t < \infty :\, X^+_t=0\}$ has no isolated points. In the $2$-dimensional case for example, every point is almost surely not visited but there exists a random (big and uncountable) set of points visited infinitely often. Here, we focus on smooth boundaries of bounded domains $\Omega \subset \mathbb{R}^d$. As in \cite{HSU86}, we may consider $r_t:=\sup\{s\geq 0\,:\, \gamma^+_s \leq t\}$ such that $\gamma^+ \circ r_t = t$. For a given $t$, $r_t$ is a stopping time for $X^+$. Then, we introduce the boundary process $\mathfrak{X}_t := X^+ \circ r_t$. That is the boundary trace process of $X^+$. The set of times $J^+=\{t \geq 0\,:\, \mathfrak{X}_{t-} \neq \mathfrak{X}_t\}$ identifies the instants at which $\mathfrak{X}_t$ has a jump. The local time $\gamma^+_t$ is constant only for the excursions of $X^+$ on the interior of the domain and $\mathfrak{X}_t$ is a pure jump process. We write $J^+_t = J^+ \cap (0,t]$ and recall that a.s. $J^+$ is a dense countable subset of $(0,\infty)$. This can be associated with the countable jumps of the (Cauchy) trace process. 

The process $X$ can be obtained via time change by considering $V_t = t + (\eta/\sigma) \gamma^+_t$ for which we have
\begin{align*}
V\circ r_t = r_t + (\eta/\sigma) t \quad \textrm{is a subordinator with drift $(\eta/\sigma)$}.
\end{align*}
The jumps of $r_t$ give the time the process $X^+$ spends on the interior $\Omega$ and $(\eta/\sigma) t$ is the time on the boundary (w.r. to the clock $\gamma^+$). Let us consider the $\epsilon$-neighbourhood 
\begin{align*}
\Lambda_\epsilon := \{x \in \bar{\Omega}\,:\, \operatorname{dist}(x, \partial \Omega) < \epsilon\}, \quad \epsilon>0
\end{align*}
introduced above and the exit time
\begin{align*}
\tau^+_\epsilon := \inf\{ t\,:\, X^+_t \notin \Lambda_\epsilon  \}, \quad \epsilon>0
\end{align*}
for $X^+_0 \in \partial \Omega$. It follows that      
\begin{align*}
V \circ \tau^+_\epsilon = \tau^+_\epsilon + (\eta/\sigma) \gamma^+ \circ \tau^+_\epsilon =: \tau^+_\epsilon + e
\end{align*}
where $e$ is the extra time accumulated on $\partial \Omega$ by $X$ on $\Lambda_\epsilon$ up to time $\tau^+_\epsilon$. {\color{\magenta} If we interpret exit times as occupation times:
\begin{itemize} \label{pageOcc12}
\item[O1)] $\tau^+_\epsilon$ is the occupation time of $\Lambda_\epsilon$ for $X^+$; 
\item[O2)] $V \circ \tau^+_\epsilon$ is the occupation time of $\Lambda_\epsilon$ for $X$.
\end{itemize}
This is due to the measure 
\begin{align*}
m(dx) = \mathbf{1}_\Omega\, dx + (\eta/\sigma) \mathbf{1}_{\partial \Omega} m_\partial(dx) 
\end{align*}
of the sticky process $X$. For $\eta =0$, $m$ reduces to the Lebesgue measure and the energy functional associated with $X^+$ is given by
\begin{align*}
\int_{\bar{\Omega}} |\nabla \varphi|^2 m(dx) = \int_\Omega |\nabla \varphi|^2 dx.
\end{align*}
Roughly speaking, although the reflection is still instantaneous, $X$ interacts with the boundary within the physical time scale by spending a positive amount of time on it. 

Since $V_t$ is continuous and strictly increasing, then $V^{-1} \circ V_t = t$ a.s. and, from \eqref{equivgamma}, the time accumulated on $\partial \Omega$ by $X$ up to $V_{\tau^+_\epsilon}$ can be obtained from 
\begin{align*}
\gamma \circ V_{\tau^+_\epsilon} = \gamma^+ \circ \tau^+_\epsilon.
\end{align*}}
\noindent From $e:= (\eta/\sigma) \gamma^+ \circ \tau^+_\epsilon$ we can describe the extra times $\hat{e}$ and $\bar{e}$ respectively for $\hat{X}$ and $\bar{X}$. In particular, analogously to the set $J^+$ for $X^+$, we can introduce the sets
\begin{align*}
J_t = J \cap (0, t], \quad \hat{J}_t = \hat{J} \cap (0, t], \quad \bar{J}_t = \bar{J} \cap (0, t]
\end{align*}
for the processes $X$, $\hat{X}$ and $\bar{X}$. Denote by $J^\delta_t$, $\hat{J}^\delta_t$ and $\bar{J}^\delta_t$ the corresponding sets of inter-times between jumps. Then, we are able to identify the sets $J^\sharp_t$, $\hat{J}^\sharp_t$ and $\bar{J}^\sharp_t$ such that
\begin{align}
\sum_{ e \in J^\delta_t} e=\sum_{ j \in J^\sharp_t} e_j, \quad \sum_{\hat{e} \in \hat{J}^\delta_t} \hat{e} = \sum_{j \in \hat{J}^\sharp_t} \hat{e}_j \quad \textrm{and} \quad \sum_{\bar{e} \in \bar{J}^\delta_t} \bar{e} = \sum_{j \in \bar{J}^\sharp_t} \bar{e}_j
\label{totalAmount}
\end{align}
provide the amount of time the processes $X$, $\hat{X}$ and $\bar{X}$ respectively spend on the boundary $\partial \Omega$ accounting for all visits to the $\epsilon$-neighbourhood $\Lambda_\epsilon$ up to $t>0$. Since $\mathbf{P}(e >0) = 1$ for $\epsilon>0$, the sets $J^\sharp$, $\hat{J}^\sharp$ and $\bar{J}^\sharp$ can be identified as the set $\mathbb{N}_0 := \mathbb{N} \cup \{0\}$.

{\color{\magenta}
\begin{definition}
The sequence $\{e_i\}_i = \{e_i,\, i \in \mathbb{N}_0\}$ of random variables is the sequence of sticky holding times for $X$. Specifically, for a given $i \in \mathbb{N}_0$, $e_i$ denotes the time the process $X$ (started from $\partial \Omega$) spends on the boundary $\partial \Omega$ during the $i$-th visit to the $\epsilon$-neighbourhood $\Lambda_\epsilon$. 
\label{def:HTX}
\end{definition}
}

For example, if $X_0 \in \partial \Omega$, then 
\begin{align*}
e_0 = meas\{t < \tau^+_\epsilon\,:\, X_t \in \partial \Omega\}
\end{align*}
and (time $V_{\tau^+_\epsilon}$)
\begin{align*}
\inf\{t\,:\, X_t \in \bar{\Omega} \setminus \Lambda_\epsilon  \} = \tau^+_\epsilon + e_0, \quad \epsilon > 0
\end{align*} 
is the first hitting time of $\bar{\Omega} \setminus \Lambda_\epsilon$, that is the sojourn time of $X$ on $\Lambda_\epsilon$:
\begin{align}
\int_0^{\tau^+_\epsilon} dV_t = \int_0^{\tau^+_\epsilon} dt + (\eta/\sigma) \int_0^{\tau^+_\epsilon} d\gamma^+_t, \quad \epsilon>0.
\label{occLTe}
\end{align}

\subsubsection{Reflected Brownian motion in the $\epsilon$-tubular neighbourhood $\Lambda_\epsilon$}

{\color{\magenta} Based on the previous section, in order to study the sequence of sticky holding times of $X$, we must consider the time $\tau^+_\epsilon$ that the process $X^+$ spends in the interior $\Omega$ and the time $\gamma^+ \circ \tau^+_\epsilon$ that the process $X^+$ spends on the boundary. Our focus is therefore on $X^+$. Observe that $\tau^+_\epsilon$ is related to the jumps of $r_t$ (excursions of $X^+$ in $\Omega$) and the extra time generally depends on the endpoint of the previous long excursion (with maximum distance from the boundary  greater than $\epsilon>0$; see the set $\mathfrak{E}$ and the discussion below on excursion processes). 

Since $X^+$ is the main object we deal with, we first define the exit times
\begin{align*}
\tau^+_i(\Lambda_\epsilon) := \inf\{t > \tau^+_i(\Omega)\, :\, X^+_t \notin  \Lambda_\epsilon \}, \quad i \in \mathbb{N}_0
\end{align*} 
and 
\begin{align*}
\tau^+_i (\Omega) := \inf\{t > \tau^+_{i-1}(\Lambda_\epsilon)\, :\, X^+_t \notin  \Omega \}, \quad i \in \mathbb{N} \quad \textrm{with} \quad \tau^+_0(\Omega)=0.
\end{align*}
Observe that, for the path of $X^+$ with $X^+_0 \in \partial \Omega$, 
\begin{align*}
\tau^+_0(\Lambda_\epsilon) = \tau^+_\epsilon \quad \textrm{and} \quad \tau^+_i(\Lambda_\epsilon) = \tau^+_i(\Omega) + \tau^+_\epsilon, \quad i \in \mathbb{N}.
\end{align*}

Now define the sequences
\begin{align*}
\tau_i(\Lambda_\epsilon) := \inf\{t > \tau_i(\Omega)\, :\, X_t \notin  \Lambda_\epsilon \}, \quad i \in \mathbb{N}_0
\end{align*} 
and 
\begin{align*}
\tau_i (\Omega) := \inf\{t > \tau_{i-1}(\Lambda_\epsilon)\, :\, X_t \notin  \Omega \}, \quad i \in \mathbb{N} \quad \textrm{with} \quad \tau_0(\Omega)=0
\end{align*}
for $X$ on $\bar{\Omega}$ with $X_0 \in \partial \Omega$. Observe that
\begin{align*}
\tau_0 (\Lambda_\epsilon) = \tau^+_\epsilon + e_0 \quad \textrm{and} \quad \tau_i (\Lambda_\epsilon) = \tau_i(\Omega) + \tau^+_\epsilon + e_i, \quad i \in \mathbb{N}
\end{align*} 
where $\tau^+_\epsilon = \tau^+_\epsilon (x_i)$ and $e_i=e_i(x_i)$ depend on $x_i=X \circ \tau_i(\Omega)$. In particular, we can write
\begin{align*}
e_0(x_0) = \big( (\eta/\sigma)\gamma^+\circ \tau^+_\epsilon \, | \, X^+_{\tau^+_0} = x_0 \in \partial \Omega \big)
\end{align*}
and, in general, $\{e_i(x_i)\,:\, x_i \in \partial \Omega,\, i \in \mathbb{N}_0\}$ is the sequence of random variables representing the sticky holding time for the $i$-th visit to $\Lambda_\epsilon$ started from $x_i$. 
\begin{itemize}
\item[] \begin{minipage}{0.85\textwidth}
{\it "The sequence $\{x_i\}_i$ is well-defined for $\epsilon >0$ as can be verified by considering the sequences of exit times $\tau_i(\Lambda_\epsilon)$ and $\tau_i(\Omega)$ for which $\tau_i(\Lambda_\epsilon) - \tau_i(\Omega)>0$ for $i \in \mathbb{N}_0$".}
\end{minipage}
\end{itemize}
Indeed, the hitting points for a reflecting Brownian motion are well-defined when restricting the analysis to those separated by a positive amount of contact time. Moreover, the sequence $\{x_i\}_i$ previously introduced as the hitting points $\{X\circ \tau_i(\Omega)\}_i$ coincides with the hitting points $\{X^+ \circ \tau^+_i(\Omega)\}_i$. Indeed, $X$ is a delayed process moving along the path of $X^+$, with the delay occurring only on the boundary. 

Recall we write $\mathbf{E}[\cdot |\, X^+_t]$ or equivalently $\mathbf{E}[\cdot | \, \mathcal{F}^+_t]$ where $\mathcal{F}^+_t = \sigma\{X^+_s,\, 0 \leq s < t\}$. Since $\tau^+_i(\Omega)$ is a stopping time, by strong Markov property, we get
\begin{align*}
\mathbf{P}_x(\gamma^+ \circ \tau^+_i(\Lambda_\epsilon) - \gamma^+ \circ \tau^+_i(\Omega) > t \, | \, X^+_{\tau^+_i (\Omega)}) = \mathbf{P}_{X^+_{ \tau^+_i(\Omega)}}(\gamma^+ \circ \tau^+_\epsilon >t)
\end{align*}
and the sequence of boundary hitting points is known by conditioning on the boundary trace process until $\tau^+_i(\Omega)$. Thus, given the sequence $\{x_j\}_{j \leq i}$ we have
\begin{align}
\mathbf{P}(e_i(x_i) > t) = \mathbf{P}_{x_i}\big((\eta/\sigma) \gamma^+ \circ \tau^+_\epsilon > t \big), \quad t \geq 0.
\label{SHTcharacterization}
\end{align} 
This concludes the preliminary characterization of the sequence $\{e_i=e_i(x_i)\}_i$. Concerning the sequence $\{\tau^+_\epsilon= \tau^+_\epsilon(x_i)\}_i$ we remark that
\begin{align}
\tau^+_\epsilon (x_i) = \big(\tau^+_i(\Lambda_\epsilon) - \tau^+_i(\Omega) \big) | X^+_{\tau^+_i(\Omega)} \quad \textrm{with} \quad x_i = X^+_{\tau^+_i(\Omega)}.
\label{EXITcharacterization}
\end{align}

We briefly discuss the sequence of sticky holding times. Let us consider  $x_0,x_1 \in \partial \Omega$ and $e_0=e_0(x_0), e_1=e_1(x_1)$,  the first two sticky holding times for $X$. Although they both depend on the associated hitting boundary points $x_0,x_1$ we have that $e_0 \indep e_1$. Therefore, we may face these three scenarios: 
\begin{itemize}
\item[i)] $e_i=e_i(x_i)$, $i \in \mathbb{N}_0$ is a sequence of independent random variables whose distributions depend on the specific hitting points $x_i \in \partial \Omega$;
\item[ii)] $e_{i_k}=e_{i_k}(x_i)$, $k \in \mathbb{N}_0$, given $x_i$,  is a subsequence of i.i.d. random variables, all sharing a common distribution that depends on the boundary hitting point $x_i \in \partial \Omega$;
\item[iii)] $e_{i}=e_{i}(x)$, $i \in \mathbb{N}_0$ is a sequence of i.i.d. random variables whose common distribution does not depend on the boundary hitting point $x \in \partial \Omega$.
\label{propertiesSeqSHT}
\end{itemize}
The independence of the sequences will be linked below to the independence of the excursions from $x_i \in \partial \Omega$, whereas the independence on $x \in \partial \Omega$ in point iii) will be related to the strong regularity of the domain $\Omega$. Notice that, in ii),  we are considering the fact that, given $x_i$,
\begin{align*}
\{e_{i_k} \}_k = \{e_{i_k}(x_i),\, k \in \mathbb{N}\}
\end{align*}
is the sequence of sticky holding times obtained during the visits to $\Lambda_\epsilon$ starting from $x_i$. Assuming $\epsilon>0$, the sequence $\{e_{i_k} \}_k$ is well-defined only if the Brownian motion has at least double points; it fails to exist if the Brownian motion is a self-avoiding path (for example, it is known that a.s. we have infinitely many double points and no triple points in $d=2$).\\

The problem of finding the law of the sticky holding times can be addressed by studying the harmonic function
\begin{align*}
u(x) = \mathbf{E}_x \left[ e^{-(\eta /\sigma) \gamma^+ \circ \tau^+_\epsilon} \right], \quad x \in \Lambda_\epsilon
\end{align*}
satisfying the Robin boundary condition on $\partial \Omega$ and such that 
\begin{align*}
u(x)=1, \quad x \in \partial \Lambda_\epsilon \setminus \partial \Omega \quad \textrm{that is} \quad x \in \{y \in \bar{\Omega}\,:\, \operatorname{dist}(y, \partial \Omega)=\epsilon\}.
\end{align*}
The function $u$ can be  obtained from the elastic Brownian motion $X^{el}$ with elastic coefficient $c=(\eta/\sigma)$. Moreover, $u$ can be considered as the equilibrium (or stationary) solution obtained in the long-time limit ($t\to \infty$) from the Feynman-Kac formula 
\begin{align*}
\mathbf{E}_x \left[ g(X^+_t) e^{-(\eta/\sigma) \gamma^+_t}\mathbf{1}_{(t < \tau^+_\epsilon)} \right], \quad t\geq 0,\; x \in \Lambda_\epsilon
\end{align*}
with $g(x)=1$ for $x$ such that $\operatorname{dist}(x, \partial \Omega) = \epsilon$. Indeed, $\Lambda_\epsilon$ is a bounded domain and the reflected Brownian motion exits the annulus almost surely, that is $\mathbf{P}_x(\tau^+_\epsilon < \infty) = 1$  for all $x \in \Lambda_\epsilon$. Thus, $u$ can be regarded as the stationary potential.

When $\Omega$ is a $C^3$-smooth bounded domain, the solution $u$ has been studied in \cite{MS90} and \cite{Pap90}. Recently, the case where $\Omega$ is a fractal domain has been investigated in \cite{BBC08}. The latter also deals with active surfaces which correspond here (via Jensen's inequality) to the case of finite mean sticky holding times, $\mathbf{E}_{x}[(\eta/\sigma)\gamma^+ \circ \tau^+_\epsilon] < \infty$. This fails for the sticky holding times of $\hat{X}$ and $\bar{X}$ (where $\mathbf{E}[H] = \infty$ or in general when $\Phi^\prime(0)=\infty$, see \eqref{InfinteMeanH} below) and provides a characterization of trap domains in line with the discussion in Section \ref{Sec:ConOP}.   

Under some regularity assumptions on $\Omega$ (for example constant curvature), the sequence of sticky holding times exhibits interesting properties. This is the case, for instance, when $\Omega$ is a ball (see below, Remark \ref{Rmk:regularityDomainPPP} for the disk and Remark \ref{Rmk:regularityDomainPPPanyd} for $d>2$) or a half-space of $\mathbb{R}^{d+1}$ (for which the discussion is omitted here, we refer to \cite{Watanabe79}). In Remark \ref{Rmk:regularityDomainPPP} we show that the local time accumulated until $\tau^+_\epsilon$ is exponentially distributed when $\Omega$ is a disk, as well as in dimension $d=1$, where this result is well known (see \cite[Theorem 7.7]{ChugWilliams}). In Remark \ref{Rmk:regularityDomainPPPanyd} we briefly discuss the case $d>2$.

\begin{remark}
\label{Rmk:regularityDomainPPP}
(The annulus $\Lambda_\epsilon$ of the unit disk $D \subset \mathbb{R}^2$) For 
$$\Lambda_\epsilon = \{x \in \bar{D}\,:\, \operatorname{dist}(x, \partial D)<\epsilon\}$$ 
we define the exit time $\tau^+_\epsilon = \inf\{t\,:\, X^+_t \in \bar{D} \setminus \Lambda_\epsilon \}$ for the reflected Brownian motion $X^+$ on $\bar{D}$. Then we study the sticky holding time for $X$ on $\bar{D}$. In particular, we show that:

\begin{itemize}
\item[] \begin{minipage}{0.85\textwidth}
{\it "Under the conditional law $\mathbf{P}_x$ with $x \in \partial D$, for the time $\gamma^+$ accumulated on the boundary $\partial D$ up to time $\tau^+_\epsilon$ we have, for $t \geq 0$,  
\begin{align*}
\mathbf{P}_x((\eta/\sigma)\gamma^+\circ \tau^+_\epsilon > t) = \exp\left( -\rho t \right) \quad \textrm{with} \quad \rho = (\sigma/\eta) / \ln (1/(1-\epsilon)).
\end{align*}
That is, $(\eta/\sigma)\gamma^+\circ \tau^+_\epsilon$ follows an exponential distribution with parameter $\rho$, independently of $x \in \partial D$".}
\end{minipage}
\end{itemize}
Our statement follows by considering the Laplace transform
\begin{align*}
u(x) = \mathbf{E}_x[e^{- (\eta/\sigma) \gamma^+\circ \tau^+_\epsilon }], \quad x \in \Lambda_\epsilon
\end{align*}
as the unique solution to
\begin{equation*}
\left\lbrace
\begin{array}{ll}
\displaystyle \Delta u(x) = 0 & x \in \Lambda_\epsilon\\
\displaystyle \partial_{\bf n} u(x) = - (\eta/\sigma) u(x) & x \in \partial D\\
\displaystyle u(x)=1 & x \in \partial \Lambda_\epsilon \setminus \partial D = \{y \in \bar{D}\,:\, \operatorname{dist}(y, \partial D)=\epsilon\}
\end{array}
\right .
\end{equation*}
Thus,  we address the problem of finding the radially harmonic component on $\Lambda_\epsilon$. Let us slightly abuse notation and write $u(x)= u(r)$ with $r=|x|$. Hence,
\begin{align*}
\Delta u(x) = 0 \quad \textrm{iff} \quad  u^{\prime \prime}(r) + \frac{1}{r} u^\prime(r) = 0.
\end{align*}
The solution to
\begin{equation*}
\left\lbrace
\begin{array}{ll}
\displaystyle u^{\prime \prime}(r) + (1/r) u^\prime(r) = 0, & 1-\epsilon < r \leq 1\\
\displaystyle u(1-\epsilon)=1 \\
\displaystyle u^\prime(1) = - (\eta/\sigma) u(1)
\end{array}
\right .
\end{equation*}
is the radial function
\begin{align*}
u(r) = 1 - \frac{(\eta/\sigma)}{ 1 + (\eta/\sigma) \ln (1/(1-\epsilon))} \ln \left( \frac{r}{1-\epsilon} \right), \quad r \in (1-\epsilon, 1)
\end{align*}
for which 
\begin{align*}
u(1) = \frac{1}{1 + (\eta/\sigma) \ln (1/(1-\epsilon))} = \frac{1}{1 + 1/\rho}
\end{align*}
implies
\begin{align*}
\mathbf{E}_x[e^{- (\eta/\sigma) \gamma^+\circ \tau^+_\epsilon }] = \mathbf{E}[e^{-\chi}], \quad \textrm{for all}\, x \in \partial D
\end{align*}
where $\mathbf{P}(\chi > t) = e^{-\rho t}$. This proves that
\begin{align*}
(\eta/\sigma) \gamma^+\circ \tau^+_\epsilon \textrm{ is exponentially distributed} 
\end{align*}
with
\begin{align*}
\mathbf{E}_x[(\eta/\sigma) \gamma^+\circ \tau^+_\epsilon] = (\eta/\sigma) \ln(1/(1-\epsilon)).
\end{align*}
Notice that, for $x \in \partial D$,
\begin{align*}
\mathbf{E}_x[(\eta/\sigma) \gamma^+\circ \tau^+_\epsilon]  \to 0 \quad \textrm{as} \quad \epsilon \to 0.
\end{align*}  
\end{remark}

\begin{remark}
\label{Rmk:regularityDomainPPPanyd}
For dimension $d>2$ we use the fact that 
\begin{align*}
\Delta u(x) = 0 \quad \textrm{iff} \quad  u^{\prime \prime}(r) + \frac{d-1}{r}  u^\prime(r) = 0
\end{align*}
and
\begin{align*}
u(r) = \left\lbrace
\begin{array}{ll}
\displaystyle b\ln r + c, & d=2\\
\\
\displaystyle b r^{2-d} + c, & d>2
\end{array}
\right .
\end{align*}
where $b$ and $c$ are constants to be determined from the boundary conditions (see for example \cite[Section 2.2.1]{EvansPDE}). We skip here the calculation and observe that
\begin{align*}
u(x) = \frac{|x|^{2-d} - 1 + (\sigma/\eta)(d-2)}{(1-\epsilon)^{2-d} - 1 + (\sigma/\eta) (d-2)}, \quad 1-\epsilon < |x| \leq 1
\end{align*}
for which
\begin{align*}
\displaystyle u(1) = \frac{1}{1 + (\eta/\sigma)\frac{(1-\epsilon)^{2-d} - 1}{(d-2)}} = \int_0^\infty e^{-z} \rho(d) e^{-\rho(d) z} dz
\end{align*}
with
\begin{align*}
\rho(d):= (\sigma/\eta) \frac{(d-2)}{(1-\epsilon)^{2-d} - 1} \to (\sigma/\eta) \frac{1}{\ln\left( \frac{1}{1-\epsilon} \right)} =\rho(2)=: \rho
\end{align*}
as $d \to 2$. We conclude that

\begin{align*}
(\eta/\sigma) \gamma^+\circ \tau^+_\epsilon \textrm{ is exponentially distributed} 
\end{align*}
with
\begin{align*}
\mathbf{E}_x[(\eta/\sigma) \gamma^+\circ \tau^+_\epsilon] = (\eta/\sigma) \frac{(1-\epsilon)^{2-d} - 1}{(d-2)}, \quad 0 < \epsilon < 1
\end{align*}
independently of $x \in \partial D$. 

\end{remark}

}

\subsubsection{Reflected Brownian motion in the subdomain $\bar{\Omega} \setminus \Lambda_\epsilon$}
\label{sec:EXCsubDomain}
In the context of Brownian excursions, we consider the set 
\begin{align*}
\mathfrak{E} = \{ \textit{excursions from $\partial \Omega$ that exit the $\epsilon$-neighbourhood $\Lambda_\epsilon$}  \}
\end{align*}
which contains all excursions that reach a distance greater than $\epsilon$ from the boundary $\partial \Omega$. For $0<\epsilon \ll 1$ we exclude only infinitesimal excursions of $X^+$ in the interior $\Omega$. Observe that, conditioned on the boundary process, the individual excursions are independent and depend only on the endpoints (\cite{Jacobs78}). Furthermore, the point process of excursion of $X^+$ on smooth domains is well-defined (\cite{HSU86}). 
{\color{\magenta} For a domain with $C^3$ boundary, by virtue of the strong Markov property of $X^+$, both the boundary local time and the associated point process of excursions are well-defined (\cite{HSU86}). In \cite{HSU86}, the author obtained a rigorous derivation of the excursion law and the compensating measure of the point process of excursions. In particular, on $(0,t)\times \mathfrak{E}$ we have (Theorem 5.1 in \cite{HSU86})
\begin{align*}
\mathsf{m}_t (\mathfrak{E}) := \int_0^t Q^{\mathfrak{X}_s}(\mathfrak{E} \cap \{\textrm{excursions started from } \mathfrak{X}_s \}) ds
\end{align*}
where $Q^x$ is the excursion law from $x \in \partial \Omega$, $\mathfrak{X}$ is the boundary process introduced above and the point process of excursions of $X^+$ is of class QL (quasi-left continuous). The point process taking values in the space of excursions is not a Poisson point process, as occurs instead in the one-dimensional case. Conditioning on $\{\mathfrak{X}_s,\, 0 \leq s < t\}$ (and therefore, given the sequence $\{x_i\}_i$ for $\{X^+_s,\, 0 \leq t < r_t\}$) we have
\begin{align*}
\mathbf{P}_x(\textrm{no excursions in $\mathfrak{E}$ in $(0, t]$ }\, | \, \mathfrak{X}_t) = e^{- \mathsf{m}_t (\mathfrak{E})}
\end{align*}
or equivalently
\begin{align*}
\mathbf{P}_x(\textrm{no excursions in $\mathfrak{E}$ in $(0, t]$ }\, | \, X^+_{r_t}) = e^{- \mathsf{m}_t (\mathfrak{E})}.
\end{align*}
Indeed, we can write
\begin{align*}
\mathsf{m}_t(\mathfrak{E}) = \int_0^{r_t} Q^{X^+_s}(\mathfrak{E} \cap \{\textrm{excursions started from } X^+_s \}) d\gamma^+_s
\end{align*}
where, by definition, $\gamma^+ \circ r_t=t$.  From the conditional expectation, it follows that
\begin{align*}
\mathbf{P}_x(\textrm{no excursions in $\mathfrak{E}$ in $(0, t]$} ) = \mathbf{E}_x [ e^{- \mathsf{m}_t (\mathfrak{E})}]. 
\end{align*}
Now, assume $\Omega$ is regular enough so that
\begin{align}
\mathsf{m}_t (\mathfrak{E}) = \int_0^t Q(\mathfrak{E} \cap \{\textrm{excursions started from } \partial \Omega \}) ds = n(\mathfrak{E}) \, t \tag{{\bf A}}
\end{align}
where $n$ denotes the It\^{o} excursion measure. Thus, under {\bf (A)}, in the local time scale,
\begin{align*}
\mathbf{P}_x(\textrm{no excursions in $\mathfrak{E}$ in $(0, t]$} ) = e^{-n(\mathfrak{E})\, t}, \quad t\geq 0.
\end{align*}
Moreover, we have that
\begin{align*}
\mathbf{P}_x(\textrm{no excursions in $\mathfrak{E}$ in $(0, t]$} ) = \mathbf{P}_x(\tau^+_\epsilon > r_t) =  \mathbf{P}_x(\gamma^+ \circ \tau^+_\epsilon > t)
\end{align*}
where the last step comes from $\gamma^+ \circ r_t =t$ and 
\begin{align*}
\mathbf{P}_x(\gamma^+ \circ \tau^+_\epsilon > t) = e^{-n(\mathfrak{E})\, t}, \quad t\geq 0
\end{align*}
where the It\^{o} measure $n(\mathfrak{E})$ depends on $\epsilon>0$ and provides information about the geometry of $\Omega$. This is confirmed (see Remark \ref{Rmk:regularityDomainPPP} above) when $\Omega$ is the unit disk and 
\begin{align*}
n(\mathfrak{E}) = \left( \ln \frac{1}{1-\epsilon} \right)^{-1}
\end{align*}
or (see Remark \ref{Rmk:regularityDomainPPPanyd}) when $\Omega$ is a ball of $\mathbb{R}^d$ with $d>2$ and
\begin{align*}
n(\mathfrak{E}) = \frac{(d-2)}{(1-\epsilon)^{2-d} - 1}.
\end{align*}
Hence, two excursion laws are translates of each other if the domain is invariant under rigid transformation (isometry) that takes one starting point into the other.
\begin{itemize}
\item[] \begin{minipage}{0.85\textwidth}
{\it "Notice that $\gamma^+ \circ \tau^+_\epsilon$ is the time (in the local time scale) at which a new long excursion (exiting $\Lambda_\epsilon$) starts from $\partial \Omega$".}
\end{minipage}
\end{itemize}

\subsubsection{The counting process in the local time scale: $\mathcal{N}^+$}
\label{sec:mathcalNpiu}
We now consider the martingale $\{\mathcal{N}^+_t - \mathsf{m}_t (\mathfrak{E})\}_{t\geq 0}$ where 
\begin{align*}
\mathcal{N}^+_t = \mathcal{N}^+_t(\mathfrak{E}) := \sharp\{\textrm{excursions in $\mathfrak{E}$ in $(0, t]$}\}
\end{align*}
which is a special case of those investigated in \cite{HSU86}. For $x \in \partial \Omega$, $t \geq 0$,
\begin{align*}
\mathbf{P}_x(\mathcal{N}^+_t = 0 \, | \, \mathfrak{X}_t ) = e^{- \mathsf{m}_t (\mathfrak{E})} = [\textrm{ under the assumption {\bf (A)} } ] = e^{-n(\mathfrak{E}) \, t} = \mathbf{P}_x(\mathcal{N}^+_t = 0).
\end{align*}
For clarity, we emphasize that, without {\bf (A)}, we have $\mathbf{P}_x(\mathcal{N}^+_t = 0) = \mathbf{E}_x[e^{-\mathsf{m}_t (\mathfrak{E})}]$. The counting process $\mathcal{N}^+=\{\mathcal{N}^+_t\}_{t\geq 0}$ is defined in the local time scale and,
\begin{align*}
\sum_{i=0}^{\mathcal{N}^+_t} e_i, \quad t \geq 0
\end{align*} 
is a coupled compound renewal process with  jump sizes distributed as the intertimes. Indeed, jumps and intertimes are both given by
\begin{align*}
(\eta/\sigma) \Big(\gamma^+ \circ \tau^+_i(\Lambda_\epsilon) - \gamma^+ \circ \tau^+_i(\Omega) \Big), \quad i \in \mathbb{N}_0
\end{align*}
and we can write, for the jump $J_i$ and the intertime (or waiting times) $I_i$, 
\begin{align*}
J_i = I_i = e_i \quad i \in \mathbb{N}_0
\end{align*}
where 
\begin{align*}
e_i \stackrel{d}{=} (\eta/\sigma) \gamma^+ \circ \tau^+_\epsilon
\end{align*}
under $\mathbf{P}_x$ with $x \in \partial \Omega$, for all $i \in \mathbb{N}_0$. The extra time $e_i=e_i(x)$ depends on the starting point $x \in \partial \Omega$ if {\bf (A)} does not hold. Thus, $\mathcal{N}^+$ is a renewal counting process and the perfect coupling between jumps and intertimes is such that
\begin{align*}
e_0 + e_1 + \ldots + e_{\mathcal{N}^+_t}, \quad t \geq 0
\end{align*}
provides both the spatial state and the (synchronized) temporal clock. We also write
\begin{align*}
e_{\mathcal{N}^+_t} = \left( t -  \sum_{i=0}^{\mathcal{N}^+_t-1} e_i \right) + \left( \sum_{i=0}^{\mathcal{N}^+_t} e_i - t \right) := \mathbf{a}^+_t + \mathbf{r}^+_t
\end{align*}
in terms of the undershoot (age $\mathbf{a}^+_t$) and the overshoot (residual lifetime $\mathbf{r}^+_t$) of $\mathcal{N}^+$. Under the assumption {\bf (A)}, we have that 
\begin{align*}
(\eta/\sigma) \Big(\gamma^+ \circ \tau^+_i(\Lambda_\epsilon) - \gamma^+ \circ \tau^+_i(\Omega) \Big) \stackrel{d}{=} \chi \quad \textrm{where} \quad \mathbf{P}(\chi > t) = e^{-(\sigma/\eta)\, n(\mathfrak{E})\, t}
\end{align*}
and $\mathcal{N}^+$ is a homogeneous Poisson process. We recall that $\mathcal{N}^+$ is a process in the local time scale associated with the boundary trace process $\mathfrak{X}$ restricted to the boundary hitting points $\{X^+ \circ \tau^+_i\}_i$ from which a new visit to $\Lambda_\epsilon$ originates (see Section \ref{sec:BoundaryProcesses} below).

\subsubsection{The counting process in the physical time scale: $N^+$}
We introduce the process in the physical time scale (that is, the one associated with $X^+_t$)
\begin{align}
\label{totalAmountLT}
\sum_{i=0}^{N^+_t} e_i \quad N^+_t = \max \{ i \in \mathbb{N}_0\,:\, \tau^+_i(\Lambda_\epsilon) < t\}
\end{align}
with $\mathbf{P}(N^+_t \geq i) = \mathbf{P}(\tau^+_i(\Lambda_\epsilon) < t)$. Now intertimes (or waiting times) are not distributed as the jump sizes. In particular, by O1) on page~\pageref{pageOcc12},
\begin{align*}
\mathbf{P}_x(N^+_t = 0) = \mathbf{P}_x(\tau^+_\epsilon >t), \quad t \geq 0, \quad x \in \partial \Omega
\end{align*}
and, in view of \eqref{EXITcharacterization}, we write the distribution of the other intertimes. For $X^+_0 \in \partial \Omega$, we can write (with $\sum_0^{-1}=0$)
\begin{equation}
(\eta/\sigma) \gamma^+_t = \left\lbrace
\begin{array}{ll}
\displaystyle \sum_{i=0}^{N^+_t - 1} e_i + (\eta/\sigma)\bigg( \gamma^+_t - \gamma^+ \circ \tau^+_{N^+_t}(\Omega) \bigg), & \tau^+_{N^+_t}(\Omega) \leq t < \tau^+_{N^+_t}(\Lambda_\epsilon),\\
\\
\displaystyle \sum_{i=0}^{N^+_t} e_i, & \tau^+_{N^+_t}(\Lambda_\epsilon) \leq t <  \tau^+_{N^+_t+1}(\Omega),
\end{array}
\right .
\label{gammaRepViaCompPoisson}
\end{equation}
from which
\begin{align*}
\sum_{i=0}^{N^+_t - 1} e_i \leq (\eta/\sigma) \gamma^+_t \leq \sum_{i=0}^{N^+_t} e_i.
\end{align*}
Recall that, as usual, 
\begin{align*}
\gamma^+ \circ \tau^+_{N^+_t}(\Omega) \quad \textrm{stands for} \quad \gamma^+_{\tau^+_{N^+_t}(\Omega)}.
\end{align*}
An inspection of the previous formula reveals that 
\begin{align*}
(\eta/\sigma) \bigg(  \gamma^+_t - \gamma^+ \circ \tau^+_{N^+_t}(\Omega) \bigg) = (\eta/\sigma) \bigg( \gamma^+_t \circ \tau^\texttt{last}_{N^+_t} - \gamma^+ \circ \tau^+_{N^+_t}(\Omega)\bigg), \quad \tau^\texttt{last}_{N^+_t} \leq t < \tau^+_{N^+_t}(\Lambda_\epsilon)
\end{align*}
is constant after 
\begin{align*}
\tau^\texttt{last}_{N^+_t} := \sup\{s < \tau^+_{N^+_t}(\Lambda_\epsilon)\,:\, X^+_s \in \partial \Omega \}
\end{align*}
and coincides with the plateau
\begin{align*}
(\eta/\sigma)  \bigg( \gamma^+ \circ\tau^\texttt{last}_{N^+_t} - \gamma^+ \circ \tau^+_{N^+_t}(\Omega)\bigg) = e_{N^+_t}, \quad \tau^+_{N^+_t}(\Lambda_\epsilon) \leq t <  \tau^+_{N^+_t+1}(\Omega).
\end{align*}
Thus, the right-hand side of formula \eqref{gammaRepViaCompPoisson} is continuous and provides a representation of $(\eta/\sigma) \gamma^+_t$ for $t\geq 0$. We emphasize that, for $\tau^+_{N^+_t}(\Omega) \leq t < \tau^+_{N^+_t}(\Lambda_\epsilon)$,
\begin{align*}
(\eta/\sigma)  \bigg( \gamma^+_t - \gamma^+ \circ \tau^+_{N^+_t}(\Omega)\bigg) \stackrel{d}{=} (\eta/\sigma)  \gamma^+ \circ ( t - \tau^+_{N^+_t}(\Omega))
\end{align*}
under $\mathbf{P}_x$ with $x=X^+ \circ \tau^+_{N^+_t}(\Omega)$ and, for $t \geq 0$, 
\begin{align*}
\big(\tau^+_{N^+_t}(\Lambda_\epsilon) - \tau^\texttt{last}_{N^+_t}\big) \approx 0 \quad \textrm{as} \quad 0 < \epsilon \ll 1.
\end{align*}
The process $N^+=\{N^+_t\}_{t \geq 0}$ is a renewal counting process in the physical time scale.

\subsubsection{The counting process in the physical time scale: $N$}
We introduce the process $N=\{N_t\}_{t\geq 0}$ to study the time accumulated on the boundary by $X$. For $\tau_i=\tau_i(\Lambda_\epsilon)$ we consider $N=\{N_t\}_{t\geq 0}$ and
\begin{align}
\sum_{i=0}^{N_t} e_i, \quad N_t = \max\{i \in \mathbb{N}_0\,:\, \tau_i < t\}
\label{totalAmountLTX}
\end{align}
with $\mathbf{P}(N_t \geq i) = \mathbf{P}(\tau_i < t)$  Also in this case, intertimes (or waiting times) are not distributed as the jump sizes. By O2) on page~\pageref{pageOcc12} we have
\begin{align*}
\mathbf{P}_x(N_t=0) = \mathbf{P}_x(\tau^+_\epsilon + e_0 >t), \quad t \geq 0, \quad x \in \partial \Omega
\end{align*}
where $\tau^+_\epsilon + e_0 =: V \circ \tau^+_\epsilon$ is the occupation time of $X$ in $\Lambda_\epsilon$ up to $\tau^+_\epsilon$ for $X_0=x \in \partial \Omega$. The time spent by $X$ respectively on $\Omega$ and $\partial \Omega$ can be written as
\begin{align*}
\tau_{N_t} + e_{N_t}, \quad t \geq 0.
\end{align*}
From \eqref{equivgamma} we see that, for $t\geq 0$, 
\begin{align*}
e_{N_t} = e_{N^+} \circ V^{-1}_t
\end{align*}
is the sticky holding time for $X_t$ (i.e., in the physical time scale). Observe that
\begin{align*}
\mathbf{P}_x(N^+ \circ V^{-1}_t = 0) = \mathbf{P}_x(\tau^+_\epsilon > V^{-1}_t) = \mathbf{P}_x(V \circ \tau^+_\epsilon >t), \quad t \geq 0, \quad x \in \partial \Omega.
\end{align*}
Since $\tau_0 = \tau^+_0 + e_0$ where $\tau^+_0 = \tau^+_\epsilon$ and $\tau^+_i=\tau^+_i(\Lambda_\epsilon)$, we have
\begin{align}
\tau_i - \tau^+_i = \sum_{j \leq i } e_j, \quad i \in \mathbb{N}_0
\label{SUMe}
\end{align} 
and
\begin{align*}
\tau_i(\Omega) - \tau^+_i(\Omega) = \sum_{j < i } e_j, \quad i \in \mathbb{N}
\end{align*}
from which, by following the same arguments as in \eqref{gammaRepViaCompPoisson},  
\begin{align*}
\tau_{N_t}(\Omega) - \tau^+_{N_t}(\Omega) \leq \gamma_t \leq \tau_{N_t} - \tau^+_{N_t}
\end{align*}
From the trace process $\mathfrak{X}$ we can determine the hitting points $X^+ \circ \tau^+_i(\Omega)$, provided that $\epsilon>0$. The delaying effect on the process $X$ does not alter the hitting points, which are now represented as $X \circ \tau_i(\Omega)$. Only if we cut the excursions exiting $\Lambda_\epsilon$, for example by considering the time change
\begin{align*}
X \circ \mathcal{V}^{-1}_t \quad \textrm{with} \quad \mathcal{V}^{-1}_t := \inf\left\{z\,:\, \int_0^z \mathbf{1}_{\Lambda_\epsilon}(X_s)ds >t \right\},
\end{align*}
then the process \eqref{totalAmountLTX} has intertimes $I_i = \tau^+_\epsilon + e_0$ under $\mathbf{P}_x$ with $x \in \partial \Omega$, for all $i \in \mathbb{N}_0$. Indeed, \eqref{totalAmountLTX} descries the jumps of the right-continuous $X \circ \mathcal{V}^{-1}_t$, $t \geq 0$ in the physical time scale.

\subsubsection{Closing considerations and implications}

We emphasize that, from the renewal (or semi-Markov) counting processes introduced above, the associated coupled compound renewal processes can be regarded as a coupled continuous-time random walk. In this respect, it is worth noting that the semi-Markov approach has recently led to a highly fruitful new formulation. We refer the reader to the recent survey \cite{BioToaCTRWs} and the references therein for a discussion of continuous-time random walks and semi-Markov models. Thus, our discussion has a clear link with subdiffusions (see Section \ref{sec:BoundaryProcesses}). We refer to the recent work \cite{BioCe-GiToa26} for the semi-Markov approach to subdiffusions and elegant results on undershoot and overshoot processes.\\

In \cite[Theorem 3.1]{BurdzyWillTRANS86} it was proved that if $\Omega$ is a $C^{1,\beta}$-domain with $\beta>0$, then all the excursion laws are standard and locally flat. This generalizes Theorem 5.1 of Burdzy (\cite{BurdzyTRANS86}) where the same conclusion was reached under the assumption that $\Omega$ is a $C^2$ domain.  Thus, in a $C^2$ domain, all excursion laws are locally flat and asymptotically equivalent to those of the half-space at time zero. This equivalence applies strictly to the birth of the excursion. Moreover, it allows for non-trivial asymptotic as $\epsilon$ approaches zero. These asymptotic behaviours will not be addressed in the present work, but will be investigated in a forthcoming paper. \\

The case of the half-space was investigated in \cite{Watanabe79} for the construction of diffusion processes with Wentzell boundary conditions by representing Brownian excursions as a stationary Poisson point process (i.e., assumption ${\bf (A)}$ of the present work holds). The compensated measure balances the divergence of infinitesimal excursions and calibrates the sticky boundary drift relative to the inward excursion flow.\\

In conclusion, the results in the subsequent sections hold for a general point process; in some cases, we also provide results under the assumption {\bf (A)} simply to serve as a clear reference model. This assumption turns out to be satisfied in case of strong (global) regularity of $\Omega$ - for example $\Omega$ is a ball of $\mathbb{R}^d$ - or in the case of locally half-space under suitable scaling limits.} Since, for $\epsilon>0$,
\begin{align}
e_0 = (\eta/\sigma) \gamma^+ \circ \tau^+_\epsilon,
\label{e0A}
\end{align}
then assuming strong regularity of $\Omega$, we have that
\begin{align}
\mathbf{P}_x(e_0 > t) = e^{-t (\sigma/\eta) n(\mathfrak{E})} \quad x \in \partial \Omega
\label{lawHoldingTimeX}
\end{align}
and
\begin{align*}
\mathbf{P}_x(e_0 > t) = \mathbf{P}(e_0 > t)=\mathbf{P}(e_i > t), \quad t >0, \quad \forall\, x \in \partial \Omega, \quad \forall \, i.
\end{align*}
{\color{\magenta} That is, under {\bf (A)},
\begin{align*}
\{e_i\}_i = \{e_i\,:\, i \in \mathbb{N}_0\} \quad \textrm{is a sequence of i.i.d. exponential random variables}
\end{align*}
whereas, without strong assumptions on the regularity of $\Omega$, $\{e_i \}_i$ is a sequence of independent sticky holding times as already mentioned in i)-iii) of page \pageref{propertiesSeqSHT}.\\

Henceforth, we refer to a visit to \(\Lambda _{\epsilon }\) during which the process hits the boundary before leaving \(\Lambda _{\epsilon }\) as a "useful visit".\\
 
}

We proceed with the discussion of the process $X$ on $\bar{\Omega}$ with generator $(G,D(G))$. 
\begin{lemma}
\label{LemmaMayOnlyHave}
Assume \eqref{lawHoldingTimeX} holds. The following statements hold true:
\begin{itemize}
\item[i)] $X$ has instantaneous reflections iff $\eta=0$;
\item[ii)] $X$ has slow (exponential sticky holding times) reflections iff $\eta>0$.
\end{itemize}
\end{lemma}
\begin{proof}
For $\eta=0$, $X$ is a strong Markov process on $\bar{\Omega}$, that is the elastic Brownian motion $X^{el}$ introduced in Section \ref{sec:intro}. The Robin boundary condition says that $X$ reflects instantaneously until the elastic kill. For $\eta>0$, $X$ is a strong Markov process on $\Omega$. Due to the  sticky holding times, the process $X$ reflects slowly. 

Since $\mathbf{P}(e_0 >t) = \exp( -(\sigma/\eta)n(\mathfrak{E}) t)$, we have that (by definition $\sigma <\infty$), if the process $X$ has only instantaneous reflections, then $\eta=0$. Conversely, if the process has  sticky holding times (that is, slow reflections), then $\eta >0$.
\end{proof}

\subsection{On $\hat{X}$ and the associated counting processes $\hat{N}$ and $\hat{\mathcal{N}}$}
\label{SecSubXhat}
We study the process $\hat{X}$ defined as $\hat{X} = X \circ \hat{L}$ in comparison with the (base) process $X$. Recall that $\hat{X}$ moves along the path of $X^+$ according to the time change $V_t$.  Now introduce
\begin{align*}
\hat{\tau}_i=\hat{\tau}_i(\Lambda_\epsilon) := \inf\{t > \hat{\tau}_i(\Omega)\, :\, \hat{X}_t \notin  \Lambda_\epsilon \}, \quad i \in \mathbb{N}_0
\end{align*} 
and
\begin{align*}
\hat{\tau}_i(\Omega) := \inf\{t > \hat{\tau}_{i-1}\, :\, \hat{X}_t \notin  \Omega \}, \quad i \in \mathbb{N} \quad \textrm{with} \quad \hat{\tau}_0(\Omega)=0
\end{align*}
for the process $\hat{X}$ on $\bar{\Omega}$ with $\hat{X}_0 \in \partial \Omega$. Considering the previous section, analogous arguments lead to $\hat{N}=\{\hat{N}_t\}_{t \geq 0}$ and 
\begin{align}
\label{totalAmountLTXhat}
\sum_{i=0}^{\hat{N}_t} \hat{e}_i, \quad \hat{N}_t = \max \{ i \in \mathbb{N}_0\,:\, \hat{\tau}_i < t\} 
\end{align}
with $\mathbf{P}(\hat{N}_t \geq i) = \mathbf{P}(\hat{\tau}_i < t)$. $\hat{N}$ is a counting process in the physical time scale and
\begin{align*}
\mathbf{P}_x(\hat{N}_t = 0) = \mathbf{P}_x(\tau^+_\epsilon + \hat{e}_0 > t), \quad t \geq 0, \quad x \in \partial \Omega.
\end{align*}
Notice that as $\alpha\to 1$, $L_t \to t$ (and $H_t \to t$) almost surely, then $\hat{\tau}_i \to \tau_i$ almost surely. In this case we have that $\hat{X} = X$ on $\bar{\Omega}$. Since $\hat{X}$ behaves like $X^+$ on $\Lambda_\epsilon \setminus \partial \Omega$, we must have
\begin{align*}
\hat{\tau}_0 = \tau^+_\epsilon + \hat{e}_0
\end{align*}  
for which $\hat{\tau}_0 - \tau_0 = \hat{e}_0 - e_0$ and 
\begin{align*}
\hat{\tau}_i - \tau_i = \sum_{j \leq i} (\hat{e}_i - e_i), \quad i \in \mathbb{N}_0
\end{align*}
that is, the additional time spent on the boundary up to the $i$-th (useful) visit to $\Lambda_\epsilon$. {\color{\magenta} From the previous formula, with \eqref{SUMe} at hand we get
\begin{align}
\hat{\tau}_i - \tau^+_i = \sum_{j \leq i} \hat{e}_j, \quad i \in \mathbb{N}_0.
\label{SUMehat}
\end{align}
The process $\hat{X}$ behaves like $X$ in the interior $\Omega$. Its local time is determined by taking into account that $\hat{X}$ behaves like $X\circ L$ on the boundary. We study the local time accumulated on $\partial \Omega$ which is obtained by considering the contribution of each (useful) visit to the $\epsilon$-neighbourhood.

\begin{definition}
The sequence $\{\hat{e}_i\}_i = \{\hat{e}_i,\, i \in \mathbb{N}_0\}$ of random variables is the sequence of sticky holding times for $\hat{X}$. Specifically, for a given $i \in \mathbb{N}_0$, $\hat{e}_i$ denotes the time the process $\hat{X}$ (started from $\partial \Omega$) spends on the boundary $\partial \Omega$ during the $i$-th visit to the $\epsilon$-neighbourhood $\Lambda_\epsilon$. 
\label{def:HTXhat}
\end{definition}

For the sequence $\{\hat{e}_i\}_i$ we have the following result.

\begin{theorem}
\label{thm:ehatHe}
For the sticky holding times $\{\hat{e}_i\}_i$ we have $\hat{e}_0 = H_{e_0}$, $\hat{e}_i \stackrel{d}{=} H_{e_i}$ for $i \in \mathbb{N}$ and $\{\hat{e}_i\}_i$ are independent. Under \eqref{lawHoldingTimeX}, $\{\hat{e}_i\}_i$ are i.i.d. Mittag-Leffler random variables.
\end{theorem}
}

\begin{proof}
The local time $\hat{\gamma}_t := \gamma \circ L_t$, that is
\begin{align*}
\mathbf{P}_x(\hat{\gamma}_t > s) = \mathbf{P}_x (\gamma \circ L_t >s) = \mathbf{P}_x(L_t > \gamma^{-1}_s) 
\end{align*}
Thus the trace process $X^*_s := X \circ \gamma^{-1}_s$ on $\partial \Omega$ runs for $0 \leq \gamma^{-1}_s < L_t$. For $\alpha=1$, the random interval simply becomes $0 \leq \gamma^{-1}_s < t$. Since $\hat{X}$ moves along the path of $X$, then $\hat{X}\circ \hat{\tau}_0 = X \circ \tau_0$ almost surely. From the trace processes 
\begin{align*}
X^*_s := X \circ \gamma^{-1}_s \quad \textrm{and} \quad \hat{X}^*_s := \hat{X} \circ \hat{\gamma}^{-1}_s = X^* \circ L_s
\end{align*}
with the random intervals 
\begin{align*}
\gamma^{-1}_s : (0, e_0) \to (0,\tau_0) \quad \textrm{and} \quad \hat{\gamma}^{-1}_s : (0, \hat{e}_0) \to (0, \hat{\tau}_0)
\end{align*}
we get $L_{\hat{e}_0} = e_0$ almost surely, that is
\begin{align*}
\hat{e}_0 = H_{e_0}.
\end{align*}

Consider now the pair $(X, H)$ of Markov processes where $X \indep H$ and $(X_0,H_0)=(x,0)$, $x \in \partial \Omega$. The process $\hat{X}$ is obtained by $(X,H)$ via $L=H^{-1}$ and $\alpha(\cdot)$ defined in Section \ref{sec:hatX}. We know the paths of $(X,H)$. Then we already know the sequence $\{e_i\}_i$ of  sticky holding times for $X$ and we are able to identify 
$$T^e_i = e_0 + e_1 + \ldots + e_i.$$ 
Concerning the subordinator $H$, the homogeneity and independence of the increments leads to the Markov property and
\begin{align*}
\textrm{for all } s\geq 0, \;\; \{H_{t+s}-H_s\}_{t \geq 0} \quad \textrm{has the same law as} \quad \{H_t\}_{t\geq 0}
\end{align*}
with $H_s \indep H_{t+s}-H_s$. By exploiting the strong Markov property of $H$ we have that
\begin{align*}
\{H_{t + e_0}-H_{e_0}\}_{t \geq 0} \quad \textrm{has the same law as} \quad \{H_t\}_{t\geq 0}
\end{align*}
with $H_{e_0} \indep H_{t + e_0}-H_{e_0}$. Consider $e_0 \indep e_1$, then 
\begin{align}
H_{e_0} \indep (H_{e_1+e_0} - H_{e_0}) \quad \textit{which has the same law as} \quad H_{e_1}. 
\label{hatWithH}
\end{align}
Since $\hat{X}$ moves along the path of $X$, we write $T^e_i = \gamma \circ \tau_i$ and, from the path of $H$, we identify the sequence $\{\hat{e}_i\}_i$ as
\begin{align}
\hat{e}_i = H \circ T^e_i - H \circ T^e_{i-1}, \quad T^e_{-1} = 0
\label{hatWithHb}
\end{align}   
where
\begin{align*}
H \circ T^e_i - H \circ T^e_{i-1} \indep H \circ T^e_{i-1}
\end{align*}
and
\begin{align*}
H \circ T^e_i - H \circ T^e_{i-1} = H \circ (e_i + T^e_{i-1}) - H \circ T^e_{i-1} \quad \textrm{has the same law as} \quad H_{e_i} .
\end{align*}
Thus, $\hat{e}_i \indep H \circ T^e_{i-1}$, that is $\hat{e}_i$ is independent from $\{e_0, \ldots, e_{i-1}\}$ as well as from $\{\hat{e}_0\, \ldots \hat{e}_{i-1}\}$. For example, formula \eqref{hatWithH} leads to the first few elements of $\{\hat{e}_i\}_i$:
\begin{itemize}
\item[-] $\hat{e}_0 = H_{e_0}$ (which is a consequence of the fact that $\mathbf{P}(\hat{X}_{\hat{e}_0} = X_{e_0})=1$);\\
\item[-] $\hat{e}_1 = H_{e_0+e_1} - H_{e_0} \indep H_{e_0}$ implies $\hat{e}_1 \indep \hat{e}_0$ and $\hat{e}_1 \stackrel{d}{=} H_{e_1}$ (according to \eqref{hatWithHb});\\
\item[-] $\hat{e}_2 = H_{e_0+e_1+e_2} - H_{e_0+e_1} \indep H_{e_0+e_1} = H_{e_0} + \hat{e}_1$ implies $\hat{e}_2 \indep \hat{e}_0 + \hat{e}_1$ and $\hat{e}_2 \stackrel{d}{=} H_{e_2}$; \\
\item[-] $\hat{e}_3 = H_{e_0+e_1+e_2 + e_3} - H_{e_0+e_1+e_2} \indep H_{e_0+e_1+e_2} = H_{e_0+e_1} + \hat{e}_2$
implies that
\begin{align*}
\hat{e}_3 \indep \sum_{j < 3} \hat{e}_j \; \textrm{ with } \; \hat{e}_0 \indep \hat{e}_1 \indep \hat{e}_2 \; \textrm{ and } \; \hat{e}_3 \stackrel{d}{=} H_{e_3}
\end{align*}
\item[-] $\hat{e}_i$ with $i >3$ by following the same arguments. 
\end{itemize}
We conclude that 
\begin{align}
\mathbf{P}\left( \bigcap_{i \geq 0} (\hat{e}_i > t_i) \right) = \prod_{i \geq 0} \mathbf{P}(H_{e_i} > t_i), \quad t_i \geq 0, \; i \in \mathbb{N}_0
\end{align}
and $\{\hat{e}_i\}_i$ is a sequence of independent random variables with $\hat{e}_0 = H_{e_0}$.\\

{\color{\magenta} When the law of $e_0$ is known to be exponential, that is, when \eqref{lawHoldingTimeX} holds, then 
\begin{align*}
\mathbf{P}(\hat{e}_0 > t) = \mathbf{P}(e_0 > L_t) = \mathbf{E} [ e^{-(\eta/\sigma) n(\mathfrak{E}) L_t} ] = E_\alpha(-(\eta/\sigma) n(\mathfrak{E}) t^\alpha),
\end{align*}
where the last step is justified by \eqref{LapML}. This concludes the proof.
}

\end{proof}

{\color{\magenta}

In the local time scale we can introduce the coupled compound process 
\begin{align*}
\sum_{i=0}^{\hat{\mathcal{N}}_t} \hat{e}_i, \quad t\geq 0
\end{align*}
where $\hat{\mathcal{N}} = \{\hat{\mathcal{N}}_t\}_{t\geq 0}$ is the counting process with
\begin{align*}
\mathbf{P}_x(\hat{\mathcal{N}}_t=0)= \mathbf{P}_x(\hat{e}_0 >t), \quad t\geq 0, \quad x \in \partial \Omega.
\end{align*}
In particular, the jump sizes and the intertimes (or waiting times) share the same distribution for all $i \in \mathbb{N}_0$. The jumps $\hat{J}_i$ and the inter-arrivals $\hat{I}_i$ have the same distribution of $\hat{e}_i$ for all $i \in \mathbb{N}_0$. More precisely,  $\hat{J}_i=\hat{I}_i=\hat{e}_i$ and $\{\hat{e}_i\}_i$ are independent (not necessary identically distributed) random variables. Thus, 
\begin{align*}
\hat{\mathcal{N}}_t = \max\{i \in \mathbb{N}_0\, ;\ T^{\hat{e}}_i < t \} \quad \textrm{with} \quad T^{\hat{e}}_t = \sum_{j \leq i} \hat{e}_j
\end{align*}
and, from the left-continuity of $\hat{\mathcal{N}}$, we write
\begin{align*}
\mathbf{P}(\hat{\mathcal{N}}_t \geq i ) = \mathbf{P}(T^{\hat{e}}_i < t).
\end{align*}
This process is well-defined once the trace process $\mathfrak{X}$ and the subset of the boundary hitting points are known, as described in items i)-iii) on page  \pageref{propertiesSeqSHT}. Under the assumption {\bf (A)}, the item iii) holds and the coupled compound process does not depend on $\mathfrak{X}$. This case will be further discussed below in Remark \ref{rmk:PPPlast}. 
}

\begin{remark}
Notice that  $L \circ H_t = t$ implies  $\hat{e}_0 = H_{e_0}$ which holds almost surely whereas $\hat{e}_i \stackrel{d}{=} H_{e_i}$ for $i >0$. 
\end{remark}

\begin{remark}
We observe that, from \eqref{totalAmountLT}, we write
\begin{align*}
\gamma \circ \tau_i = \sum_j e_j \, \mathbf{1}_{(j \leq N_{\tau_i})} = \sum_j e_j \, \mathbf{1}_{(\tau_j \leq \tau_i)} = \sum_{j \leq i} e_j
\end{align*}
whereas, from \eqref{totalAmountLTXhat}, we get that the telescoping series
\begin{align*}
\hat{\gamma} \circ \hat{\tau}_i = \sum_{j} \hat{e}_j\, \mathbf{1}_{(j \leq \hat{N}_{\hat{\tau}_i})} = \sum_{j \leq i} \hat{e}_j = [ \textrm{from } \eqref{hatWithHb} ]= H \circ \sum_{j \leq i} e_j.
\end{align*}
\end{remark}

\begin{remark}
As stated above $\{\hat{e}_i\}_i$ is a sequence of i.i.d. sticky holding times for $\hat{X}$. This accords well with the fact that $H\indep X$ and $\{e_i\}_i$ is a sequence of i.i.d. sticky holding times for $X$. Thus, under \eqref{lawHoldingTimeX}, for all $i$, we have the Mittag-Leffler distribution
\begin{align}
\label{holdingML}
\mathbf{P}(\hat{e}_i > t) = E_\alpha (-(\sigma / \eta) t^\alpha), \quad t\geq 0.
\end{align}
Since $E_\alpha \notin L^1(0, \infty)$, we deduce that the process $\hat{X}$ spends an infinite mean amount of time on $\partial \Omega$ with each visit. We remark that, in general, 
\begin{align*}
\mathbf{E}[\hat{e}_i] = \mathbf{E}[H_1]\, \mathbf{E}[e_i]
\end{align*}
for all $i$ and $\mathbf{E}[H_1]=\infty$.
\end{remark}

Now we state the following result on the independent delay for the reflection of the time-changed sticky process $\hat{X}$.
\begin{lemma}
\label{LemmaMayOnlyHaveHat}
Assume \eqref{lawHoldingTimeX} holds. The following statements hold true:
\begin{itemize} 
\item[i)] $\hat{X}$ has instantaneous reflections iff $\eta=0$;
\item[ii)] $\hat{X}$ has slow (delayed) reflections iff $\eta>0$.
\end{itemize}
\end{lemma}
\begin{proof}
We follow the same arguments as in the proof of Lemma \ref{LemmaMayOnlyHave}. Recall that 
\begin{align*}
\mathbf{P}(\hat{e}_0 >t) = E_\alpha(-(\sigma/\eta) t^\alpha)
\end{align*}
where the Mittag-Leffler function $E_\alpha$ has been introduced in Section \ref{sec:secXL}. Thus, we may have instantaneous reflection only if $\eta=0$ whereas, we may have Mittag-Leffler sticky holding times if $\eta>0$.

On the other hand, for $\eta=0$, we have that
\begin{align*}
\mathbf{E}_x[f(\hat{X}_t)] = \int_{\Omega} f(y) p(t,x,y) dy
\end{align*}
where $p(t,x,y)$ is now the continuous kernel for the elastic Brownian motion $X^{el}$.
In case $\eta>0$, the process $\hat{X}$ is defined as a time-changed elastic sticky Brownian motion $X$ for which the new clock $L$ introduces a delay effect (see Definition 5.2 in \cite{CapDovDelRus17}). Since $X$ slowly reflects (exponential sticky holding times) on the boundary, then $\hat{X} = X \circ L$ turns out to be delayed (Mittag-Leffler sticky holding times). Indeed, as $\eta \neq 0$, 
\begin{align*}
\mathbf{P}(e_i >t) \neq 0 \quad \Rightarrow \quad \mathbf{P}(\hat{e}_i > t) = \int_0^\infty \mathbf{P}(e_i >s) \mathbf{P}(L_t \in ds) \neq 0.
\end{align*}
\end{proof}

\begin{remark}
\label{rmk:HThatX} 
(Holding times for $\alpha \neq 1$) The random time $L$ delays the process $X$ and the jumps of $H$ introduce plateaus for the clock $L$. Thus, the time change $L$ produces holding times for $\hat{X}$. Recall that
\begin{align*}
X^*_s := X \circ \gamma^{-1}_s \quad \textrm{and} \quad \hat{X}^*_s := \hat{X} \circ \hat{\gamma}^{-1}_s = X^* \circ L_s
\end{align*}
as \eqref{probBoundary} entails. Thus the process $\hat{X}$ can spend a random time in a point of the boundary according to the jumps $\{H_t - H_{t-},\, t>0\}$. Let us define $J_H = \{t >0\,:\, H_{t} > H_{t-}\}$ and 
\begin{align*}
L^\Delta_t = \inf\{\delta>0\,:\, L_{t+\delta} > L_t\}, \quad t>0.
\end{align*}
Then $\mathbf{P}(L^\Delta_t >0,\, t>0) = 0$ and
\begin{align*}
\mathbf{P}(L^\Delta_{H_{t-}}> 0,\, t \in J_H) = 1.
\end{align*}
We observe that these holding times do not depend on the parameter $\epsilon\geq 0$ of the $\Lambda_\epsilon$. 
\end{remark}

\subsection{On $\bar{X}$ and the associated counting processes $\bar{N}$ and $\bar{\mathcal{N}}$}
\label{SecSubXbar}

{\color{\magenta} We continue our analysis by introducing 
\begin{align}
\label{totalAmountLTXbar}
\sum_{i=0}^{\bar{N}_t} \bar{e}_i, \quad \bar{N}_t = \max \{ i \in \mathbb{N}_0\,:\, \bar{\tau}_i < t\}
\end{align}
with $\mathbf{P}(\bar{N}_t \geq i) = \mathbf{P}(\bar{\tau}_i < t)$ where, analogously to the previous sections, 
\begin{align*}
\bar{\tau}_i = \bar{\tau}_i (\Lambda_\epsilon) := \inf\{t>\bar{\tau}_i(\Omega)\,:\, \bar{X}_t \notin \Lambda_\epsilon\}, \quad i \in \mathbb{N}_0,
\end{align*} 
and
\begin{align*}
\bar{\tau}_i(\Omega) := \inf\{t > \bar{\tau}_{i-1} \, : \, \bar{X}_t \notin \Omega\}, \quad i \in \mathbb{N},\quad \bar{\tau}_0(\Omega) = 0.
\end{align*}
$\bar{N}=\{\bar{N}_t\}_{t \geq 0}$ is a counting process in the physical time scale and
\begin{align*}
\mathbf{P}_x(\bar{N}_t = 0) = \mathbf{P}_x(\tau^+_\epsilon + \bar{e}_0 > t), \quad t \geq 0, \quad x \in \partial \Omega.
\end{align*}
The other intertimes (waiting times) can be determined according to items i)-iii) on page~\pageref{propertiesSeqSHT}. As for \eqref{totalAmountLTXhat} we can write
\begin{align*}
\bar{\tau}_i - \tau_i = \sum_{j \leq i} (\bar{e}_j - e_j), \quad i \in \mathbb{N}_0.
\end{align*}
Moreover, as in \eqref{SUMehat} we get
\begin{align}
\bar{\tau}_i - \tau^+_i = \sum_{j \leq i} \bar{e}_i, \quad i \in \mathbb{N}_0.
\label{SUMebar}
\end{align}  
We also recall that, for $\alpha = 1$ (or $H_t = t$ a.s.), we have
\begin{align*}
\textrm{for all $i \in \mathbb{N}_0$}, \quad \bar{\tau}_i = \tau_i \quad \textrm{and} \quad \bar{e}_i = e_i  \quad \textrm{almost surely}.  
\end{align*} 
}

For the process $\bar{X}$ we now discuss about sticky holding times and holding times for $\alpha \in (0,1]$. Observe that we may have holding times only for $\alpha \in (0,1)$. 

{\color{\magenta}
\begin{definition}
The sequence $\{\bar{e}_i\}_i = \{\bar{e}_i,\, i \in \mathbb{N}_0\}$ of random variables is the sequence of sticky holding times for $\bar{X}$. Specifically, for a given $i \in \mathbb{N}_0$, $\bar{e}_i$ denotes the time the process $\bar{X}$ (started from $\partial \Omega$) spends on the boundary $\partial \Omega$ during the $i$-th visit to the $\epsilon$-neighbourhood $\Lambda_\epsilon$. 
\label{def:HTXbar}
\end{definition}

For the sequence $\{\bar{e}_i\}_i$ we have the following result.

\begin{theorem}
\label{thm:holdingHatBar}
For the sticky holding times $\{\bar{e}_i\}_i$, we have $\bar{e}_0 = H_{e_0}$, $\bar{e}_i \stackrel{d}{=}H_{e_i}$ for $i \in \mathbb{N}$ and $\{\bar{e}_i\}_i$ are independent. Under \eqref{lawHoldingTimeX}, $\{\bar{e}_i\}_i$ are i.i.d. Mittag-Leffler random variables.
\end{theorem}
}
\begin{proof}
The process $\bar{X}$ moves on the path of $X^+$ according to the new clock $\bar{V}^{-1}$ which governs the time the process $\bar{X}$  spends on the boundary with each visit of $\Lambda_\epsilon$, that is the  sticky holding times for $\bar{X}$. Notice that $X^+$ may have only instantaneous reflections up to the elastic kill. Focus on $V$ and $\bar{V}$. By definition, $V_t$ and $\bar{V}_t$ equal $t$ in case $\eta=0$, there are no sticky holding times and we have instantaneous reflection as the path of $X^+$ entails. For $\eta>0$, the extra-time $(\eta/\sigma)\gamma^+_t$ the process $X^+$ spends on the boundary is considered in the construction of $X$. Similarly, $H \circ (\eta/\sigma) \gamma^+_t$ gives the extra-time the process $X^+$ spends on the boundary in order to have $\bar{X}$. Let us introduce $\bar{J}^\delta_t$ and $\bar{J}^\sharp_t$ by following the same arguments as in Section \ref{sec:comparisonXhatX}. That is, $\bar{J}^\delta_t$ is the set of inter-times between jumps for the trace process of $\bar{X}$ and therefore, it gives the  sticky holding times for $\bar{X}$. The set $\bar{J}^\sharp$ can be identified with the set $\mathbb{N}_0$. We can write $\bar{V}_t - t = H \circ (V_t - t)$ which is trivial in case $\alpha=1$. Indeed, $H_t$ becomes the elementary subordinator $t$. Recall that $H$ is independent from $X^+$. Thus, for a suitable random time we get the equality in law $\bar{e} = H \circ e$ with $\bar{e} \in \bar{J}^\delta_\infty$ and $e \in J^\delta_\infty$. In particular, assume the process $X^+$ starts away from $\partial \Omega$, then $V_{t} - t = 0$ for $0\leq t < \tau_{\partial \Omega}$ and $V_{\tau_{\partial \Omega}} - \tau_{\partial \Omega} > 0$ due to the instantaneous reflection of $X^+$. Observe that $V_t - t$ is the continuous clock for the right-continuous process $\bar{V}_t - t$ which jumps according to $H$. This means that the jumps of $\bar{V}_t - t$ can be properly identified. We consider the return (on the boundary) time $\tau^+_i$ (after the first hitting time $\tau^+_0=\tau_{\partial \Omega}$) for the excursion $i \in \bar{J}^\sharp$ of $\bar{X}$ on the interior $\Omega$ and we are able to identify the sequence $\{\bar{e}_i\}$, $i \in \bar{J}^\sharp$ where $\bar{e}_i$ has the same law as $H \circ e_i$. For $\alpha=1$, we get $\bar{e} = e$ (with $\mathbf{P}(e >t) =e^{-t\, (\sigma/\eta) n(\mathfrak{E}) }$ for the Poisson point process of excursions, that is under {\bf (A)}). Indeed, $\alpha=1$ in \eqref{probXbar} together with \eqref{SameArgs} gives the problem \eqref{probX} for the sticky Brownian motion $X$ discussed in Section \ref{sec:secX}. We conclude that $\{\bar{e}_i\}_i$ is a sequence of independent and identically distributed random variables. The discussion above entails a very interesting difference between the (Lebesgue measurable) occupation times for $\alpha=1$ and $\alpha \in (0,1)$.\\ 

In detail, we proceed as follows:
\begin{itemize}
\item[-] Recall that
\begin{align*}
T^e_i := \sum_{j \leq i} e_j, \quad e_j \in J^\delta
\end{align*}
according to Section \ref{sec:comparisonXhatX};
\item[-] For the extra time $V_t - t = (\eta/\sigma) \gamma^+_t$ we consider $V_{\tau^+_i} - \tau^+_i$. That is,
\begin{align}
(\eta/\sigma) \gamma^+ \circ \tau^+_i = T^e_i \in J, \quad i \in \mathbb{N}_0
\label{extraT1}
\end{align} 
where now $\tau^+_i=\tau^+_i(\Omega) < \tau^+_i(\Lambda_\epsilon)$ is defined according to Section \ref{sec:BoundaryProcesses};
\item[-] For the extra time $\bar{V}_t - t = H \circ (\eta/\sigma) \gamma^+_t$ we consider $\bar{V}_{\tau^+_i} - \tau^+_i$. That is,
\begin{align}
H \circ (\eta/\sigma) \gamma^+ \circ \tau^+_i = H \circ T^e_i \in \bar{J}, \quad i \in \mathbb{N}_0 ;
\label{extraT2}
\end{align} 
\item[-] By construction we get
\begin{align}
e_i = T^e_i - T^e_{i-1} \quad \textrm{and} \quad \bar{e}_i = H \circ T^e_i - H \circ T^e_{i-1}, \quad T^e_{-1}=0
\label{barWithHb}
\end{align}
where the last formula coincides with \eqref{hatWithHb}. As it is clear now, $\bar{e}_i = e_i$ as $\alpha=1$. For $\alpha \in (0,1]$:
\begin{itemize}
\item[-] $\bar{e}_0 = H_{e_0}$ (which is a consequence of \eqref{extraT1} and \eqref{extraT2});\\
\item[-] $\bar{e}_1 = H_{e_0+e_1} - H_{e_0} \indep H_{e_0}$ implies $\bar{e}_1 \indep \bar{e}_0$ and $\bar{e}_1 \stackrel{d}{=} H_{e_1}$ (according to \eqref{barWithHb});\\
\item[-] $\bar{e}_i$ with $i>1$ as above in the proof of Theorem \ref{thm:ehatHe}.
\end{itemize}
Thus, we obtain that
\begin{align*}
\textrm{$T^e_0 = e_0$ implies $\bar{e}_0 = H_{e_0}$}
\end{align*}
and
\begin{align*}
\bar{e}_i \;\; \textrm{has the same law as} \;\;  H \circ e_i \;\;  \textrm{for all} \;\; i > 0.
\end{align*}
\end{itemize}
This concludes the first part of the proof.\\

{\color{\magenta} When the law of $e_0$ is known to be exponential, that is, when \eqref{lawHoldingTimeX} holds, then 
\begin{align*}
\mathbf{P}(\bar{e}_0 > t) = \mathbf{P}(e_0 > L_t) = \mathbf{E} [ e^{-(\eta/\sigma) n(\mathfrak{E}) L_t} ] = E_\alpha(-(\eta/\sigma) n(\mathfrak{E}) t^\alpha),
\end{align*}
where the identity follows from \eqref{LapML}. This concludes the second part of the proof.
}

\end{proof}

{\color{\magenta}
In the local time scale we can introduce the process $\bar{\mathcal{N}} = \{\bar{\mathcal{N}}_t\}_{t\geq 0}$ with
\begin{align*}
\mathbf{P}_x(\bar{\mathcal{N}}_t = 0 ) = \mathbf{P}_x(\bar{e}_0 >t)
\end{align*}
and the coupled compound process 
\begin{align*}
\sum_{i=0}^{\bar{\mathcal{N}}_t} \bar{e}_i, \quad t \geq 0
\end{align*}
with jump sizes distributed as the intertimes (or waiting times). Analogously to $\hat{\mathcal{N}}$, this process is well-defined once the trace process $\mathfrak{X}$ and the subset of the boundary hitting points are known, as described in items i)-iii) on page  \pageref{propertiesSeqSHT}. The coupled compound process is characterized by the jumps $\bar{J}_i$ and the intertimes $\bar{I}_i$ such that $\bar{J}_i = \bar{I}_i = \bar{e}_i$ and 
\begin{align*}
\mathbf{P}(\bar{\mathcal{N}}_t \geq i) = \mathbf{P}(T^{\bar{e}}_i < t) \quad \textrm{with} \quad T^{\bar{e}}_i = \sum_{j \leq i} \bar{e}_j.
\end{align*} 
Under the assumption {\bf (A)}, in which case item iii) holds, the sequence $\{\bar{e}_i\}_i$ does not depend on $\mathfrak{X}$. Also this case will be further discussed in Remark \ref{rmk:PPPlast}.  

}

\begin{remark}
\label{rmk:HTbarX}
(Holding times for $\alpha \neq 1$) For the process $\bar{X}$, we observe that \begin{align*}
\bar{V}_t = t + H \circ (\eta/\sigma) \gamma^+_t \quad \textrm{and} \quad \bar{V}\circ r_t = r_t + H \circ (\eta/\sigma) t.
\end{align*}
The jumps of $r_t$ give the time the process $X^+$ spends on the interior $\Omega$ and the jumps of $H \circ (\eta/\sigma) t$ rushes the clock (w.r. to $\gamma^+$). The independent jumps of $H$ produce plateaus for $\bar{V}^{-1}_t$, thus the process $\bar{X}$ can spend a random time in a point of the boundary. We observe that these holding times do not depend on the parameter $\epsilon\geq 0$ of $\Lambda_\epsilon$. 
\end{remark}

We state the following result on the slow reflections of the sticky process $\bar{X}$.

\begin{lemma}
\label{LemmaMayOnlyHaveBar}
Assume \eqref{lawHoldingTimeX} holds. The following statements hold true:
\begin{itemize} 
\item[i)] $\bar{X}$ has instantaneous reflections iff $\eta=0$;
\item[ii)] $\bar{X}$ has slow (heavy tailed sticky holding times) reflections iff $\eta>0$.
\end{itemize}
\end{lemma}
\begin{proof}
We follow partially the proof of Lemma \ref{LemmaMayOnlyHaveHat}. Recall that  
\begin{align*}
\mathbf{P}(\bar{e}_0 >t) = E_\alpha(-(\sigma/\eta) t^\alpha).
\end{align*}
Instantaneous reflection implies $\mathbf{P}(\bar{e}_0 >t)=0$ for all $t>0$, that is $\eta=0$ whereas we have delayed reflection only in case $\eta>0$.

For $\eta=0$, $\bar{V}_t = t$ and $\bar{X}$ is an elastic Brownian motion which implies instantaneous reflections. Focus now on $\eta>0$. Recall that $S_{\beta, t}$ defined in the proof of Theorem \ref{thm:pBarContinuity}, generated by $G_\beta = \Delta$ with
\begin{align*}
D(G_\beta) = \left\{ \varphi, \Delta \varphi \in C(\bar{\Omega}), \varphi \in H^1(\Omega)\,:\, \eta \frac{\lambda^\alpha}{\lambda} (\Delta \varphi) |_{\partial \Omega} = - \sigma \partial_{\bf n} \varphi - c \varphi |_{\partial \Omega} \right\},
\end{align*}
is the semigroup of the elastic sticky Brownian motion with speed measure
\begin{align*}
m_\lambda(dx) = dx + \frac{\lambda^\alpha}{\lambda} (\eta/\sigma) \, m_\partial (dx).
\end{align*}
Thus, for $\alpha=1$, the process can only have slow reflection (see Lemma \ref{LemmaMayOnlyHave}) whereas, for $\alpha \in (0,1)$ the exponential sticky holding times are delayed by $L$. Indeed, from Theorem \ref{thm:holdingHatBar}, $\bar{e}_i$ equals in law a Mittag-Leffler random variable for $i\geq 0$. That is, we have heavy tailed distribution. 

\end{proof}

\begin{remark}
\label{rmk:PPPlast}
{\color{\magenta}
(Excursion and heavy-tailed Poisson process) We stress the following facts in the case where $X^+$ gives rise to a Poisson point process taking values in the space of excursions. From now on, we proceed under assumption ${\bf (A)}$. 

For $t \geq 0$, it holds that (see Section \ref{sec:mathcalNpiu})
\begin{align*}
\mathbf{P}(e_0 > t) = \mathbf{P}(\textrm{no excursions of $X$ reach $\bar{\Omega} \setminus \Lambda_\epsilon$ up to time $t$}) = \mathbf{P}(\mathcal{N}^+_{t} = 0). 
\end{align*}
Analogously, we can write (see Section \ref{SecSubXhat} and Section \ref{SecSubXbar})
\begin{align*}
\mathbf{P}_x(\hat{e}_0 > t) = \mathbf{P}(\textrm{no excursions of $\hat{X}$ reach $\bar{\Omega} \setminus \Lambda_\epsilon$ up to time $t$}) = \mathbf{P}(\hat{\mathcal{N}}_{t} = 0) 
\end{align*}
and
\begin{align*}
\mathbf{P}_x(\bar{e}_0 > t) = \mathbf{P}(\textrm{no excursions of $\bar{X}$ reach $\bar{\Omega} \setminus \Lambda_\epsilon$ up to time $t$}) = \mathbf{P}(\bar{\mathcal{N}}_{t} = 0) 
\end{align*}
where $\hat{\mathcal{N}}$ and $\bar{\mathcal{N}}$ equal in law the time-changed process 
\begin{align*}
\mathcal{N}^L=\{\mathcal{N}^L_t\}_{t \geq 0} \quad \textrm{with} \quad \mathcal{N}^L_t:=\mathcal{N}^+ \circ L_t,\;  t \geq 0.
\end{align*}
Recall that $\mathcal{N}^+$ is a homogeneous Poisson process independent of the inverse $L$ of the subordinator $H$ for which
\begin{equation*}
\mathbf{P}(\mathcal{N}^+_t = k) = e^{-\lambda t} \frac{(\lambda t)^k}{k!} , \quad k \in \mathbb{N}_0 \quad  \textrm{with }\; \mathbf{E}[\mathcal{N}^+_t] = \lambda t , \quad \lambda :=(\sigma/\eta) n(\mathfrak{E}).
\end{equation*} 
For the time-changed process $\mathcal{N}^L$,  also termed fractional Poisson process, we have
\begin{equation*}
\mathbf{P}(\mathcal{N}^L_t = k) = \mathbf{P}(\mathcal{N}^+ \circ L_t = k) = \int_0^\infty \mathbf{P}(\mathcal{N}^+_s = k)\, l(t,s)ds , \quad k \in \mathbb{N}_0.
\end{equation*}
The distribution of $\mathcal{N}^L$, denoted by $p_k^\alpha(t) = \mathbf{P}(\mathcal{N}^L_t = k)$, solves the equation 
\begin{align*}
D^\alpha_t p^\alpha_k(t) = - \lambda (I-K)p^\alpha_k(t), \quad t >0,\, k \in \mathbb{N}_0 
\end{align*}
where $K$ is the Frobenius-Perron operator corresponding to the transformation $T(k) = k + 1$ (see \cite[Remark 17]{dovWright}). This problem is among those addressed in Section \ref{sec:secXL}. 
We say that "no Poissonian events" coincides with "no excursions outside $\Lambda_\epsilon$". With \eqref{LapML} at hand, we obtain
\begin{equation*}
\mathbf{P}(\mathcal{N}^L_t = 0) = \int_0^\infty \mathbf{P}(\mathcal{N}^+_s = 0)\, l(t,s)ds = \int_0^\infty e^{- \lambda s}\, l(t,s)ds = E_\alpha(-\lambda t^\alpha)
\end{equation*}
which is the distribution of the intertimes $\hat{e}_0$ and $\bar{e}_0$. Thus, from Theorem \ref{thm:ehatHe} and Theorem \ref{thm:holdingHatBar} we can relate the sequences of sticky holding times with $\mathcal{N}^L$. 

We also observe that
\begin{align*}
\mathbf{E}[\mathcal{N}^L_t] = \mathbf{E}[\mathcal{N}^+_1]\, \mathbf{E}[L_t] = (\sigma/\eta)\, n(\mathfrak{E}) \frac{t^\alpha}{\Gamma(\alpha + 1)}, \quad t \geq 0
\end{align*}
where $\Gamma(\cdot)$ is the gamma function and $n(\mathfrak{E})$ is the It\^{o} measure of excursions. In the next section we use the fact that $\mathcal{N}^+$ is a process in the local time scale.
}
\end{remark}

\section{On the boundary processes}
\label{sec:BoundaryProcesses}
The following discussion is obtained as a by-product of the previous results. We proceed without a detailed proof. Recall that we deal only with static behaviour on the boundary (that is we have no diffusive part). The boundary trace processes we are interested in are expected to be pure jump processes written as
\begin{align*}
X \circ \gamma^{-1}_t \quad t \geq 0, \quad \hat{X} \circ \hat{\gamma}^{-1}_t \quad t\geq 0 \quad \textrm{and} \quad \bar{X} \circ \bar{\gamma}^{-1}_t \quad t\geq 0.
\end{align*}
{\color{\magenta}
Considering the excursions of the reflecting process $X^+$ within $\Lambda _{\epsilon}$, we disregard the jumps of the trace process corresponding to excursions that do not exit $\Lambda _{\epsilon}$. The elapsed running time of these inner excursions is now treated as - or added to - the waiting time of the trace process at any landing point of $X^+$ on $\partial \Omega$ from which the associated visit to $\Lambda_\epsilon$ originates.\\

We slightly change the notation of Section \ref{sec:OnX} for $X^+$ on $\bar{\Omega}$ with $X^+_0 \in \partial \Omega$, that is
\begin{align*}
\tau_i^+(\Lambda_\epsilon) := \inf\{t > \tau_{i}^+\, :\, X^+_t \notin  \Lambda_\epsilon \}, \quad i \in \mathbb{N}_0
\end{align*} 
and 
\begin{align*}
\tau^+_i (\Omega) := \inf\{t > \tau_{i-1}^+(\Lambda_\epsilon)\, :\, X^+_t \notin  \Omega \}, \quad i \in \mathbb{N} \quad \textrm{with} \quad \tau^+_0(\Omega)=0
\end{align*}
as above, whereas here,
\begin{align*}
\tau^+_i= \tau^+_i(\Omega), \quad i \in \mathbb{N}_0.
\end{align*}
Assume $\Omega$ is a ball of $\mathbb{R}^d$, that is {\bf (A)} holds true. Let us consider now the following processes on $\partial \Omega$. Along the path of $X^+$, for the path of $X$ we define $X^\partial = \{X^\partial_t\}_{t\geq 0}$ with   
\begin{align*}
X^\partial_t := X^+ \circ \int_0^t \big( \tau^+_{\mathcal{N}^+_s} - \tau^+_{\mathcal{N}^+_{s-}} \big) d\mathcal{N}^+_s = X^+ \circ \sum_{k=0}^{\mathcal{N}^+_t} \big( \tau^+_k - \tau^+_{k-1} \big) = X^+ \circ \tau^+_{\mathcal{N}^+_t}
\end{align*}
and, by following the same arguments for the processes $\hat{X}$ and $\bar{X}$, we define, for $t\geq 0$,
\begin{align*}
\hat{X}^\partial_t := X^+ \circ \tau^+_{\hat{\mathcal{N}}_t} \quad \textrm{and} \quad \bar{X}^\partial_t := X^+ \circ \tau^+_{\bar{\mathcal{N}}_t}.
\end{align*}
As discussed in Remark \ref{rmk:PPPlast}, we have equivalence in law for both the random times above with the (non-Markovian) time change
\begin{align*}
\int_0^t \big( \tau^+_{\mathcal{N}^L_s} - \tau^+_{\mathcal{N}^L_{s-}} \big) d\mathcal{N}^L_s = \sum_{k=0}^{\mathcal{N}^L_t} \big( \tau^+_k - \tau^+_{k-1} \big) = \tau^+_{\mathcal{N}^L_t}, \quad t >0.
\end{align*}
From the discussion in Section \ref{sec:EXCsubDomain}, it follows that:
\begin{itemize}
\item[] \begin{minipage}{0.85\textwidth}
{\it "The process $\{\tau^+_{\mathcal{N}^L_t}\}_{t\geq 0}$ takes values in the subset of boundary hitting times associated with the restricted set of boundary hitting points of $X^+$ that mark the termination of an excursion outside $\Lambda_\epsilon$ and the beginning of a new ``useful'' visit to $\Lambda_\epsilon$, for any fixed $\epsilon>0$".}
\end{minipage}
\end{itemize}
The process (in the local time scale) $X^\partial$ with $X^\partial_0=X^+_0$, after the waiting time $e_0$, jumps to $X^+ \circ \tau^+_1$. Analogously, the process (in the physical time scale) $X$ with $X_0=X^+_0 \in \partial \Omega$ exits $\Lambda_\epsilon$ after the residence time $\tau^+_\epsilon + e_0$. Then, it moves through the interior $\Omega$ until the next boundary hitting time $\tau_1=\tau_1(\Omega)$ for which the boundary hitting point $X \circ \tau_1$ coincides with $X^+ \circ \tau^+_1$.
 
The processes $\hat{X}^\partial = \{\hat{X}^\partial_t\}_{t\geq 0}$ and $\bar{X}^\partial = \{\bar{X}^\partial_t\}_{t\geq 0}$ depend on $\epsilon>0$ and are semi-Markov/renewal-type jump processes on the boundary $\partial \Omega$. The waiting times are given by the sticky holding times. The intertimes (waiting times) and the jumps of $X^\partial_t$ have the same exponential law only if assumption {\bf (A)} holds true. This implies, via time change with $L$, a Mittag-Leffler distribution for both waiting times and the jumps of both $\hat{X}^\partial$ and $\bar{X}^\partial$. In general, we are led to the open problem of whether the identity
\begin{align}
\mathbf{E}_x[f(\bar{X}^\partial_t)] = \mathbf{E}_x[f(\hat{X}^\partial_t)], \quad f\in C(\partial \Omega), \quad t\geq 0, \; x \in \partial \Omega
\label{idOpenProb}
\end{align}
holds in some sense also for $\epsilon \to 0$.} This should allow an equivalence between boundary traces of $\mathsf{w}$ and $u$ given above. A rigorous formulation of \eqref{idOpenProb} is still in progress. \\

The boundary trace process has generator given by the Dirichlet-Neumann operator which can be defined as the normal derivative of the solution to a Dirichlet problem (we refer to \cite{HSU86} for a general discussion and the probabilistic framework). Indeed, it is well-known that the Dirichlet-Neumann (or the Dirichlet-to-Neumann) operator, say $A_{DN}$, maps $U \in D(A_{DN})$ to $\partial_{\bf n} \mathcal{U} \in L^2(\partial \Omega)$ where $\mathcal{U} \in H^1(\Omega)$ is the harmonic function with $T \mathcal{U}=U$. In general, this operator is non-local and the associated motion is right-continuous with left limits. This operator turns out to be associated with the reflected Brownian motion $X^+$ (and therefore with instantaneous reflections).

It is clear that $X^\partial_t$ is a Markov process depending on $\epsilon>0$. Assume that $A^*_{DN}$ is the generator of the boundary process $X^\partial_t$. Then, $\varphi(t,x) = \mathbf{E}_x[f(X^\partial_t)]$ solves the problem
\begin{align*}
\frac{\partial \varphi}{\partial t} = A^*_{DN} \varphi, \quad \varphi_0 = f \in D(A^*_{DN}).
\end{align*}
Since $\bar{N}_t$ can be obtained via time change as $N \circ L_t$ where $L=\{L_t\}_{t\geq 0}$ is an inverse to a stable subordinator, then standard arguments say that $\bar{\varphi}(t,x) = \mathbf{E}_x[f(\bar{X}^\partial_t)]$ solves the non-local problem
\begin{align*}
D^\alpha_t \bar{\varphi} = A^*_{DN} \bar{\varphi}, \quad \bar{\varphi}_0 = f \in D(A^*_{DN}).
\end{align*}
Thus, the problem is of the type discussed in Section \ref{sec:secXL}.  Moreover, we notice that also the boundary process of $\hat{X}$ can be associated with  FCPs. Indeed, $\hat{X}$ has a representation along the path of $X^+$ and, as a consequence of Theorem \ref{thm:holdingHatBar}, we write $\hat{X}^\partial_t$ in terms of $\hat{N}_t$ which equals in law $N^L_t$. Thus, $\hat{\varphi}(t,x) = \mathbf{E}_x[f(\hat{X}^\partial_t)]$ solves the problems 
\begin{align*}
D^\alpha_t \hat{\varphi} = A^*_{DN} \hat{\varphi}, \quad \hat{\varphi}_0 = f \in D(A^*_{DN}).
\end{align*}
The rigorous characterization of $A^*_{DN}=A^*_{DN}(\epsilon)$ for $\epsilon \geq 0$ as well as the connection between $A_{DN}$ and the boundary trace processes of $X$ and $\bar{X}$ (or $\hat{X}$) is part of our recent investigations.

\section{Conclusions}
\label{Sec:ConOP}


We have introduced the process $\hat{X}$ for which we have an anomalous behaviour on the boundary. Despite this, it cannot be considered to solve the NLBVP for which instead, we have to introduce $\bar{X}$. The different processes $\hat{X}$ and $\bar{X}$ are leading to the same (static) behaviour on the boundary in terms of occupation times. The sticky condition (for $\alpha=1$) implies a finite mean  sticky holding time whereas, for $\alpha \in (0,1)$, the fractional sticky analogue implies an infinite mean sticky holding time. We recall that for $\alpha \in (0,1]$, we have $\mathbf{P}_x(\hat{e}_i >t) = \mathbf{P}_x(\bar{e}_i >t)$ for $i \in \mathbb{N}_0$, $x \in \partial \Omega$. The sticky holding times $\{\hat{e}_i\}_i$ and $\{\bar{e}_i\}_i$ are independent and identically distributed random variables as entailed in Theorem \ref{thm:ehatHe} and Theorem \ref{thm:holdingHatBar}. This implies that $\mathbf{E}[\hat{e}_i]$ and $\mathbf{E}[\bar{e}_i]$ are finite only for $\alpha=1$ (which is the case of $X$ with generator $(G,D(G))$ introduced in Section \ref{sec:secX} and characterized in Theorem \ref{thm:MAINI} for $\alpha=1$). {\color{\magenta} Sticky holding times have been introduced to describe the combination of paradigms 1 and 2 on page~\pageref{sec:comparisonXhatX}.} They can be taken as general as we need by considering the operator $D^\Phi_t \varrho$ in place of $D^\alpha_t \varrho$ where, according to \eqref{CD-LT}, 
\begin{align}
\int_0^\infty e^{-\lambda t} D^\Phi_t \varrho(t,x)\, dt 
= & \frac{\Phi(\lambda)}{\lambda} \big(\lambda (\mathcal{L}\varrho)(\lambda, x) - \varrho(0, x) \big), \quad \lambda>0
\label{genCD}
\end{align}
and $H$ is such that $\mathbf{E}_0[\exp(-\lambda H_t)] = \exp(- t \Phi(\lambda))$.
As discussed in \cite{CapDovDelRus17}, we may consider sticky holding times for $\hat{X}$ and $\bar{X}$ characterized by
\begin{align*}
\lim_{\lambda \to 0} \frac{\Phi(\lambda)}{\lambda} \quad \textrm{where} \quad \frac{\Phi(\lambda)}{\lambda} = \int_0^\infty e^{-\lambda y} \phi(y, \infty) dy
\end{align*} 
written in terms of the tail of $\phi$ and
\begin{align*}
\Phi(\lambda) = \int_0^\infty (1- e^{-\lambda y}) \phi(dy)
\end{align*}
characterizes the subordinator $H$ via the L\'{e}vy measure $\phi$. In this case, $\mathbf{E}[\hat{e}_i]$ and $\mathbf{E}[\bar{e}_i]$ are given by
\begin{align}
\mathbf{E}[H_{e_i}] = \mathbf{E}[H_1]\, \mathbf{E}[e_i] = \lim_{\lambda \to 0} \frac{\Phi(\lambda)}{\lambda}\, \mathbf{E}[e_i], \quad i \in \mathbb{N}_0.
\label{InfinteMeanH}
\end{align}
For $\alpha \in (0,1)$, the time that the processes $\hat{X}$ and $\bar{X}$ spend on the boundary $\partial \Omega$ (as a portion of the sticky holding times) is also determined by the "standard" holding times, each defined as the duration of a single trap (where time is delayed by the jumps of $H$). {\color{\magenta} Sticky holding times can be associated with continuous-time random walks modeling real data by sampling useful excursions according to $\epsilon>0$. Formulas \eqref{SUMehat} and \eqref{SUMebar} lead to 
\begin{align}
\tau_{\mathcal{N}^+_t} - \tau^+_{\mathcal{N}^+_t} = \sum_{i=0}^{\mathcal{N}^+_t} e_i, \quad \hat{\tau}_{\hat{\mathcal{N}}_t} - \tau^+_{\bar{\mathcal{N}}_t} = \sum_{i=0}^{\hat{\mathcal{N}}_t} \hat{e}_i, \quad \bar{\tau}_{\bar{\mathcal{N}}_t} - \tau^+_{\bar{\mathcal{N}}_t} = \sum_{i=0}^{\bar{\mathcal{N}}_t} \bar{e}_i 
\label{coupledCTRWs}
\end{align}
where the coupled continuous-time random walks are determined (in distribution) by
\begin{align*}
\sum_{i=0}^{\mathcal{N}^L_t} H_{e_i}
\end{align*}
and $L$ denoting the inverse of $H$. This suggests a sub-diffusive limit process obtained as $\epsilon \to 0$, which is subject of an ongoing project. The coupled continuous-time random walks in \eqref{coupledCTRWs} can be regarded as special cases of those investigated in \cite{OnTelegSign}, especially \cite[Section 4.3]{OnTelegSign}. In case of lateral diffusion, we have studied the trace of processes driven by non-local dynamic boundary condition in \cite{CapDovEECT}. \\
}

Our results on $X$, $\hat{X}$ and $\bar{X}$ are new. In particular, to the best of our knowledge, the non-local dynamic problem \eqref{probXbar} is introduced in the present paper together with the probabilistic representation of its solution. Furthermore, in case $\alpha=1$ the probabilistic representation of the solution on $\bar{\Omega} \subset \mathbb{R}^d$ for $d>1$  is novel and agrees with the representation obtained on the half-line in \cite{ItoMcK}. {\color{\magenta}On the other hand, due to the recurrence of the Brownian motion, analysing the occupation time of the boundary not only requires significant effort but also reveals distinct behaviours from those observed on the half-line.} \\

We briefly provide some examples of applications to financial, biological and traffic/data  models. 

Many phenomena can exhibit the same behaviour as the processes $\hat{X}$ and $\bar{X}$ near the boundary, that is as an endogenous property of the motion. An example is given in \cite{BKLT-financeSticky} where the authors deal with a model of sticky expectations in which investors update their beliefs too slowly. A further example is given by the bank interest rates (or the Vasicek model for instance)  which are termed sticky if they react slowly to changes in the corresponding market rates or in the policy rate (\cite{GNSS-finance-Sticky}). Our results in this regards give a control for the slow update/reaction depending on $\alpha \in (0,1)$.

Moreover, an intermediate phenomenon between microscopic and macroscopic scales is the synthetic of biological adhesion of colloidal particles to one another or to other surfaces, typically immersed in a fluid medium. Sticky diffusions may arise when modeling systems of mesoscale particles (those with diameters around 1 micrometer) which form the building blocks for many common materials (\cite{KBS-bio-sticky}). 

{\color{\magenta} Sticky Brownian motion is a key tool for modelling interacting particles because it accurately simulates temporary adhesion and boundary-layer trapping (for instance, in sticky-distance models). In this work, we introduce a wide variety of sticky behaviours.}

The problem \eqref{probXbar} in the $1$-dimensional case was recently considered in \cite{BonDov} in order to describe motions on star graphs. This result can also be considered in traffic models or motions on metric graphs. The motion slows down at a vertex (node) according to an independent holding time, thereby creating trap vertices. 

The problem \eqref{probXbar} in the $d$-dimensional case introduces a Brownian motion on a smooth domain with an irregular behaviour near the boundary. Roughly speaking we formulate the following conjecture:
\begin{itemize}
\item[] \begin{minipage}{0.85\textwidth}
{\it "$\bar{X}$ on $\Omega^s$ can be considered to describe the behaviour of $X^+$ on $\Omega^*$ where $\Omega^*$ is a trap domain (as defined in \cite{BCM06}) and $\Omega^s$ is a smooth approximation of $\Omega^*$, for example, $\operatorname{dist}(\partial \Omega^s, \partial \Omega^* ) < s=s(\epsilon)$ and $\partial \Omega^s$ is smooth".}
\end{minipage}
\end{itemize}
Trap domains can be obtained as in Figure \ref{Fig:KochModified} in terms of Koch curves. In case of Lipschitz domains we have results for the FCP on Koch (also random Koch) domain given in \cite{CapDovFCP21} which is the case on the left in Figure \ref{Fig:KochModified}. The fractal domain obtained as a limit of modified Koch domains (as in the right of Figure \ref{Fig:KochModified}) is a trap domain for the Brownian motion. In particular, by following the definition in \cite{BCM06}, we introduce the ball $\mathcal{B} \subset \Omega^*$ of positive radius where $\Omega^*$ is the (fractal) modified Koch domain. Then, for the reflected Brownian motion $X^+$ on $\Omega^*$, define $\tau_{\mathcal{B}} = \inf\{t\, :\, X^+_t \in \partial \mathcal{B}\}$ which can be regarded as the lifetime of $X^+$ on $\Omega^* \setminus \bar{\mathcal{B}}$. That is, $X^+$ is killed on $\partial \mathcal{B}$. Thus, $\Omega^*$ is trap for $X^+$ meaning that
\begin{align}
\sup_{x \in \Omega^* \setminus \bar{\mathcal{B}}} \mathbf{E}_x[\tau_{\mathcal{B}}] = \infty.
\label{DefTRAP}
\end{align}
A discussion on the modified Koch domain as an example of trap domain has been given in \cite{BBC08}. In the same work, the authors proved that the fractal Koch domain (obtained as a limit of the domains in the left of Figure \ref{Fig:KochModified}) is non trap for $X^+$. The condition \eqref{DefTRAP} provides a characterization of the boundary $\partial \Omega^*$. The process $X^+$ can be trapped in some region of the domain, near the irregular boundary. Then, the mean lifetime (time needed to reach the Dirichlet ball) is infinite.

Our conjecture is based on the results in \cite{BBC08} where a function $\mathcal{W}(a)$ is considered for the portion removed in the wall (of length $a>0$) closing each triangles of the Koch curves. Concerning Figure \ref{Fig:KochModified}, from the Koch curves of the picture on the left, we construct the modified curves of the picture on the right. Thus, $\mathcal{W}(a)$ gives the openings in triangles with side $a>0$. Under the action of the contractive similitudes, with a given contraction factor, we construct the (fractal) modified Koch domain and we obtain the (fractal) Koch domain from the Lipschitz domains. The problem is to identify $\alpha \to \mathcal{W}_\alpha(a)$ such that $\mathcal{W}_{\alpha_1} (a) > \mathcal{W}_{\alpha_2} (a)$ if $\alpha_1 > \alpha_2$ for all $a>0$. For example, $\mathcal{W}_1(a) = a$ (there is no wall) and $\mathcal{W}_0(a)=0$ (there is no window). Thus, the trap effect on the boundary $\partial \Omega^*$ for $X^+$ depends on $\alpha \in (0,1)$ which governs the holding times on $\partial \Omega^s$ for $\bar{X}$.

\begin{figure}
\centering
\includegraphics[scale=2]{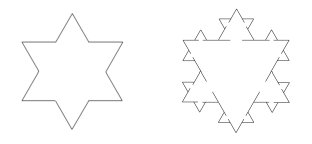} 
\caption{The Koch domain (pre-fractal, first step) on the left and the modified Koch domain (pre-fractal, second step) introduced in \cite{BCM06} on the right. In the second picture the passage between triangles is blocked by a wall with a small opening. The modified pre-fractal leads to a trap domain for the Brownian motion. Our conjecture says that $\alpha \in (0,1)$ for $\bar{X}$ on a smooth domain can be associated with the size of the openings in the second "bad" domain (modified Koch).}
\label{Fig:KochModified}
\end{figure}






{\bf Acknowledgments}\\
The author would like to thank both Reviewers for their valuable suggestions and meticulous reading of the manuscript. In particular, the author would like to extend special thanks to the second Reviewer. The suggested changes and corrections have significantly improved both the presentation and the overall value of the results.\\
The author thanks MUR for the support under the project PRIN 2022 - 2022XZSAFN: "Anomalous Phenomena on Regular and Irregular Domains: Approximating Complexity for the Applied Sciences" - CUP B53D23009540006 - PNRR M4.C2.1.1 (\url{https://www.sbai.uniroma1.it/~mirko.dovidio/prinSite/index.html} {\color{\magenta} or \url{https://sites.google.com/uniroma1.it/mirko-dovidio/prin-2022} activated after Department relocation}). The research was supported also by INdAM-GNAMPA and Sapienza University of Rome (Ateneo 2020).

\end{document}